\documentclass[11pt,reqno]{amsart}
  \usepackage[all,2cell,arrow]{xy}
  \usepackage[left=2.5cm,right=2.5cm,top=3cm,bottom=3cm]{geometry}
  \usepackage{amssymb}
 \numberwithin{equation}{section}
 \usepackage{xcolor}
\definecolor{darkgreen}{rgb}{0,0.45,0} 
\usepackage[colorlinks,citecolor=darkgreen,linkcolor=darkgreen]{hyperref}
\usepackage{enumerate,xspace}
\UseAllTwocells
\CompileMatrices
\allowdisplaybreaks
\usepackage{mathtools}

\newcommand{\act}{\triangleright}

\newcommand{\Hopf}{{\sf Hopf}\xspace}
\newcommand{\HKer}{{\sf HKer}\xspace}

\newcommand{\ox}{\otimes}
\newcommand{\C}{\mathcal{C}}
\newcommand{\ot}{\otimes}

\renewcommand{\phi}{\varphi}
  \newtheorem{proposition}{Proposition}[section]
  \newtheorem{lemma}[proposition]{Lemma}
  
  \newtheorem{corollary}[proposition]{Corollary}
  \newtheorem{theorem}[proposition]{Theorem}
  \theoremstyle{definition}
  \newtheorem{definition}[proposition]{Definition}
  
  \newtheorem{examples}[proposition]{Examples}
  \newtheorem{remark}[proposition]{Remark}
\usepackage{tikz}
\usepackage{tikz-cd}
\title{Crossed squares of cocommutative Hopf algebras}
\author{Florence Sterck}

\address{Institut de Recherche en Math\'ematique et Physique, Universit\'e catholique de Louvain, Chemin du Cyclotron 2, 1348 
Louvain-la-Neuve \and D\'epartement de Math\'ematique, Universit\'e Libre de Bruxelles, Campus de la Plaine – CP 210
Boulevard du Triomphe
1050 Bruxelles, Belgium }
\email{florence.sterck@uclouvain.be}


\begin{document}

\begin{abstract}
In this paper, we define the notion of \textit{Hopf crossed square} for cocommutative Hopf algebras extending the notions of crossed squares of groups and of Lie algebras. We prove the equivalence between the category of Hopf crossed squares and the category of double internal groupoids in the category of cocommutative Hopf algebras. The Hopf crossed squares turn out to be the internal crossed modules in the category of crossed modules in the category of cocommutative Hopf algebras.
\end{abstract} 

\maketitle

\section*{Introduction}

Motivated by the study of the properties of the boundary map
\[ \partial \colon \pi_2(X,A,x) \to \pi_1(A,x), \] 
relating the second relative homotopy group and the first homotopy group of a pair of pointed topological spaces, J.H.C Whithead introduced the new notion of crossed modules in 1949 \cite{Whitehead}.
A crossed module of groups is given by a morphism of groups 
\[ d : X \to B,\]
endowed with an action of $B$ on $X$ in a way that is ``compatible'' with the action given by conjugation (see Definition \ref{xmod grp}). 
Later on, the notion of crossed module has attracted a lot of attention as an algebraic structure on its own. R. Brown and C. Spencer \cite{BS} proved the equivalence between crossed modules of groups and internal groupoids in the category of groups (a result that they credit to Verdier).
 On the one hand, a crossed module of groups can be seen as the ``normalization'' of an internal groupoid. Indeed, if 
  \begin{equation}\label{rgintro}
  \begin{tikzpicture}[descr/.style={fill=white},scale=1.2,baseline=(A.base)]
\node (A) at (0,0) {$A_1$};
\node (D) at (3,0) {$A_0$};
 \path[->,font=\scriptsize]
([yshift=7pt]A.east) edge node[above] {$\delta$} ([yshift=7pt]D.west)
(D.west) edge node[descr] {$\iota$} (A.east)
([yshift=-7pt]A.east) edge node[below] {$\gamma$} ([yshift=-7pt]D.west);
\end{tikzpicture}
\end{equation} is the reflexive graph underlying an internal groupoid structure in the category of groups ($\delta$ is the ``domain morphism'', $\gamma$ the ``codomain morphism'', $\iota$ the ``identity morphism'', see Definition \ref{grpd}), its normalization 
\[  \gamma \cdot ker(\delta) \colon Ker(\delta) \to A_0 \] 
is a crossed module of groups. On the other hand, we can construct an internal groupoid from a crossed module of groups thanks to the semi-direct product. When $d \colon X \rightarrow B$ is a crossed module of groups, the following diagram   \[\begin{tikzpicture}[descr/.style={fill=white},scale=1.2,baseline=(A.base)]
\node (A) at (0,0) {$X \rtimes B$};
\node (D) at (3,0) {$B$};
 \path[->,font=\scriptsize]
([yshift=7pt]A.east) edge node[above] {$p_2$} ([yshift=7pt]D.west)
(D.west) edge node[descr] {$ i_2$} (A.east)
([yshift=-7pt]A.east) edge node[below] {$ p_1 $} ([yshift=-7pt]D.west);
\end{tikzpicture},\] where $p_2$ is the projection on $B$, $i_2$ is the inclusion of $B$ and $p_1$ is defined as $p_1(x,b) = d(x)b$, is underlying a unique internal groupoid structure in the category of groups. This equivalence relates two interesting notions and allows one to deal  with the concept of internal groupoid in an alternative way, that is useful for computations.

In \cite{LR} R. Lavendhomme and J.-R. Roisin were interested in the cohomology of algebraic structures, and they showed a similar equivalence in the context of Lie algebras and of rings. G. Guin-Waléry and J.L. Loday in \cite{GWL,Loday} studied the relation between crossed modules and the so-called \emph{1-cat-groups}, which are reflexive graphs as in \eqref{rgintro} such that the kernels of $\delta$ and $\gamma$ have a trivial commutator: $ [\ker(\delta), \ker(\gamma)] = \{1\}$. These authors also highlighted a relation between what they called \emph{crossed squares of groups} and \emph{2-cat-groups}.  A few years later G. Ellis introduced the notion  of \emph{${cat}^2$-Lie algebra} (and, more generally, of  \emph{$cat^n$-Lie algebra}) and obtained analogous results for Lie algebras \cite{Ellis}. More recently, G. Janelidze discovered that the theory of semi-abelian categories offers an adequate framework to internalize the notion of crossed modules. In 2003, besides defining the abstract notion of internal crossed module, he proved that there exists an equivalence between internal crossed structures and internal groupoids in any semi-abelian category \cite{Jan}, whose main examples are the category of groups and of Lie algebras. These categories are again semi-abelian \cite{BG2}, so the construction of crossed modules can be repeated. One obtains a ``2-dimensional" version of the preceding result providing an equivalence between internal double crossed modules and internal double groupoids. 

This higher dimensional equivalence (in groups) allows one to consider a double groupoid in the category of groups
\[
\begin{tikzpicture}[descr/.style={fill=white},scale=1]
\node (A) at (0,0) {$G_1$};
\node (B) at (0,3) {$G_{\{1,2\}}$};
\node (C) at (3,3) {$G_2$};
\node (D) at (3,0) {$G_\emptyset ,$};
 \path[->,font=\scriptsize]
 (C.west) edge node[descr] {$\iota_2$} (B.east)
(A.north) edge node[descr] {$\iota_1$}
 (B.south)
 (D) edge node[descr] {$\iota'_1$}
 (C)
 (D.west) edge node[descr] {$\iota'_2$} (A.east)
 ([xshift=7pt]B.south) edge node[right] {$\gamma_1$} ([xshift=7pt]A.north)
([xshift=-7pt]B.south) edge node[left] {$\delta_1$} ([xshift=-7pt]A.north)
 ([xshift=7pt]C.south) edge node[right] {$\gamma'_1$} ([xshift=7pt]D.north)
([xshift=-7pt]C.south) edge node[left] {$\delta'_1$} ([xshift=-7pt]D.north)
([yshift=7pt]A.east) edge node[above] {$\gamma'_2$} ([yshift=7pt]D.west)
([yshift=-7pt]A.east) edge node[below] {$\delta'_2$} ([yshift=-7pt]D.west)
([yshift=7pt]B.east) edge node[above] {$\gamma_2$} ([yshift=7pt]C.west)
([yshift=-7pt]B.east) edge node[below] {$\delta_2$} ([yshift=-7pt]C.west);
\end{tikzpicture}
\] as a commutative square with compatible actions - called \emph{crossed square of groups} -
\[ 
\begin{tikzpicture}[descr/.style={fill=white},scale=0.8]
\node (A) at (0,0) {$M'$};
\node (B) at (0,2) {$L$};
\node (C) at (2,2) {$M$};
\node (D) at (2,0) {$P.$};
  \path[-stealth]
 (B.south) edge node[right] {$ \lambda'$} (A.north) 
 (C.south) edge node[right] {$ \mu$} (D.north)
(A.east) edge node[above] {$ \mu'$} (D.west)
(B.east) edge node[above] {$ \lambda$} (C.west);
\end{tikzpicture}
\]
The quite complicated structure of double groupoid is equivalent to a commutative square of groups with some suitably defined and compatible actions, which turns out to be easier to handle in some relevant cases. In particular, in \cite{EllisC}, the author used the notion of crossed squares of groups to cover the homotopy classification of maps from a 3-dimensional connected CW-space $X$ to some other CW-space. In \cite{BL}, R. Brown and J.L. Loday developed a Van Kampen theorem for crossed squares of groups.

In a recent paper \cite{DV}, the authors used the non-explicit construction of the internal crossed squares in semi-abelian categories to develop an internal approach to the non-abelian tensor product in semi-abelian categories, in order to extend the classical non-abelian tensor product of groups \cite{BL}.

Furthermore, it was proved in \cite{GSV} (see also \cite{GKV}) that the category of cocommutative Hopf algebras over a field is semi-abelian. Thanks to this result and the theory of G. Janelidze, an equivalence between internal crossed modules and internal groupoids exists. In \cite{GSV} we used that equivalence to describe the notion of internal crossed modules of cocommutative Hopf algebras (in the sense of G. Janelidze), which turned out to coincide with the definition given independently by \cite{Majid} and by \cite{Vilaboa} in the case of cocommutative Hopf algebras. 

Motivated by this description of internal crossed modules of cocommutative Hopf algebras, our goal is to go one step further and describe internal \emph{Hopf crossed squares} of cocommutative Hopf algebras in order to obtain an equivalent notion to the one of double groupoid in the category of cocommutative Hopf algebras, $\sf Hopf_{K,coc}$. In this paper we define the suitable crossed structure corresponding to double groupoids in $\sf Hopf_{K,coc}$, and then show the equivalence between the corresponding categories. 
Furthermore, we obtain the definition of crossed squares of groups and of Lie algebras as special case. For the sake of convenience, we use other descriptions of internal groupoids or double groupoids, respectively \emph{cat$^1$-Hopf algebras} and \emph{cat$^2$-Hopf algebras}, since the categorical commutator can be easily expressed in our context by using a recent observation made in \cite{GSV}. This notion of cat$^2$-Hopf algebras is  easier to deal with than  the one of double internal groupoid, since it avoids the description of pullbacks.

The layout of this article is as follows: the first section contains some preliminaries and definitions. The second section is devoted to the definition of \emph{Hopf crossed square}, to some of its properties and to the relation with the concepts of crossed square of groups and of Lie algebras. In the last section we prove that this notion is indeed equivalent to the one of double internal groupoid. We obtain the equivalence between the notion of $2$-fold split epimorphisms and \emph{Hopf $2$-actions} as an intermediate result. In the main body of the paper we focus on the main results and the idea of the proofs, all technical proofs are collected in the appendices, for the convenience of the reader.

\section*{Acknowledgements}

The author would like to thank her supervisors Professor Marino Gran (UCLouvain) and Professor Joost Vercruysse (ULB) for all the advice in the realization of this paper. The
author thanks the anonymous referee for his/her useful remarks. The author's research is supported by a FRIA doctoral grant no. 27485 of the \textit{Communaut\'e fran\c{c}aise de Belgique}.

\section{Preliminaries}\label{preli}
We start our preliminaries by the definitions of some algebra structures over a field $K$: coalgebra, bialgebra and Hopf algebra. We refer to \cite{A,Sweedler} for more details about the basis of the theory of Hopf algebras.
\subsection{Hopf algebras}
A \emph{$K$-coalgebra} is the dual notion of the notion of $K$-algebras. More precisely, it is a vector space $C$ endowed with a comultiplication $\Delta \colon C\to C\ot C$ and a counit $\epsilon \colon C\to K$ satisfying the coassociativity and the counitality conditions: \[(Id\ox \Delta)\cdot \Delta=(\Delta\ox Id)\cdot\Delta,\] \[ (Id\ox \epsilon)\cdot \Delta=Id=(\epsilon\ox Id)\cdot\Delta.\]
We use the classical Sweedler notation for calculations with the comultiplication,
and we write $\Delta(c)=c_1\ox c_2$ for any $c\in C$ (with the usual summation convention, where $c_1\ot c_2$ stands for $\sum c_1\ot c_2$). 

Coassociativity and counitality can then be expressed by the formulas
\begin{eqnarray*}
&c_1\ot c_{2_1}\ot c_{2_2}=c_{1_1}\ot c_{1_2}\ot c_2=c_1\ot c_2\ot c_3,\\
&c_1\epsilon(c_2)=c=\epsilon(c_1)c_2.
\end{eqnarray*}
 
A \emph{morphism of $K$-coalgebras} $f : C \rightarrow D$ is a linear map such that for any $c$ in $C$, 
\begin{eqnarray*}
&f(c)_1 \ox f(c)_2 = f(c_1) \ox f(c_2),\\
&\epsilon \cdot f(c) = \epsilon(c).
\end{eqnarray*}
  
We recall that a \emph{$K$-bialgebra} $(A,M,\eta_A,\Delta,\epsilon)$ is a $K$-algebra $A$ with multiplication $M \colon A \ot A\to A$ and unit $\eta_A \colon K\to A \colon 1 \mapsto 1_A$ that is at the same time a $K$-coalgebra with comultiplication $\Delta \colon A\to A\ot A$ and counit $\epsilon \colon A \to K$ such that $M$ and $\eta_A$ are coalgebra morphisms or, equivalently, $\Delta$ and $\epsilon$ are algebra morphisms, which can be expressed in the Sweedler notation as
\begin{eqnarray*}
(ab)_1\ot (ab)_2 = a_1b_1\ot a_2b_2 && 1_{A_1}\ot 1_{A_2}=1_A\ot 1_A\\
\epsilon(ab)=\epsilon(a)\epsilon(b) && \epsilon(1_A)=1.
\end{eqnarray*}
 
A \emph{Hopf $K$-algebra} is a sextuple $(A,M,\eta_A,\Delta,\epsilon,S)$ where $(A,M,\eta_A,\Delta,\epsilon)$ is a $K$-bialgebra and $S \colon A\to A$ is a linear map, called the {\em antipode}, making the following diagrams commute
\[\xymatrix{
& A\otimes A \ar@<0.5ex>[rr]^-{S \otimes Id} \ar@<-0.5ex>[rr]_-{Id \otimes S} & &  A\otimes A \ar[dr]^-{M}  \\
A \ar[rr]_-{\epsilon} \ar[ur]^-{\Delta} & &K \ar[rr]_-{\eta_A}& & A.
   }\]
In the Sweedler notation, the commutativity of these diagrams can be written as
\[a_1S(a_2)=\epsilon(a)1_A=S(a_1)a_2,\]
for any $a\in A$.

A Hopf $K$-algebra $(A,M,\eta,\Delta,\epsilon,S)$ is \emph{cocommutative} if its underlying $K$-coalgebra is cocommutative, meaning that the comultiplication map $\Delta$ satisfies $\sigma \cdot \Delta = \Delta$, where $\sigma\colon A \otimes A \longrightarrow A \otimes A$ is the twist map $\sigma(a \otimes b)=(b \otimes a)$, for any $a\otimes b\in A \otimes A$. In Sweedler notation, the cocommutativity is expressed as 
\begin{equation}\label{coco}
a_1\ot a_2=a_2\ot a_1,
\end{equation} for any $a\in A$.
Note that the antipode of a cocommutative Hopf $K$-algebra satisfies $S^2 =Id$.

A \emph{morphism of Hopf $K$-algebras} is a linear map that is both an algebra and a coalgebra morphism (the antipode is then automatically preserved). 
$\sf Hopf_{K,coc}$ denotes the category whose objects are cocommutative Hopf $K$-algebras and whose morphisms are morphisms of Hopf $K$-algebras. From now on, we will say Hopf algebras for Hopf $K$-algebras. 

The kernel of a morphism $f \colon A\to B$ in $\sf Hopf_{K, coc}$ is given by the inclusion $\HKer(f)  \to A$ of the following Hopf subalgebra of $A$
\[\HKer(f) = \{ a \in A \, \mid \, f(a_1) \otimes a_2 = 1_B \otimes a \} { = \{ a \in A \, \mid \, a_1 \otimes f(a_2) =  a \otimes  1_B\}.}\]

\subsection{Semi-abelian category}
We recall the definition of a semi-abelian category, a notion introduced in \cite{JMT}. The aim of the definition was to give axioms to capture the properties of the categories of groups, non-unital rings, Lie algebras, etc. Roughly speaking, the relation between the semi-abelian categories and the groups is similar as the one between the abelian categories and the abelian groups. The semi-abelian categories share a lot of properties among them: the fact that all traditional homological lemmas hold \cite{BB}, one can develop a theory of commutators \cite{Bourncomm} and one can define categorical notion as \textit{action}, \textit{semi-direct product}, \textit{crossed module}, etc \cite{Jan}. 

\begin{definition}
 A finitely complete category $\mathcal{C}$ (i.e. which possesses equalizers and all the finite products) is \emph{regular} if any arrow $f \colon A \rightarrow B$ factors as a regular epimorphism $p \colon A \rightarrow I$ followed by a monomorphism $i \colon I \rightarrow B$ and if, moreover, these factorizations are pullback-stable. 
\end{definition}

\begin{definition}  A pointed category is \emph{protomodular} in the sense of \cite{Bourn} if the \emph{Split Short Five Lemma} holds in $\mathcal{C}$. This lemma means that for any diagram
 \begin{center}
\begin{tikzpicture}[descr/.style={fill=white},baseline=(A.base),scale=0.8] 
\node (A) at (0,0) {$A_1$};
\node (B) at (2.5,0) {$B_1$};
\node (C) at (-2.5,0) {$X_1$};
\node (A') at (0,-2) {$A_2$};
\node (B') at (2.5,-2) {$B_2$};
\node (C') at (-2.5,-2) {$X_2$};
\path[->,font=\scriptsize]
(B.south) edge node[right] {$ g$}  (B'.north)
 (C.south) edge node[left] {$ v $}  (C'.north)
(A.south) edge node[left] {$ p$} (A'.north)
(C'.east) edge node[above] {$\kappa_2$} (A'.west)
([yshift=-4pt]A'.east) edge node[below] {$\alpha_2$} ([yshift=-4pt]B'.west)
([yshift=2pt]B'.west) edge node[above] {$e_2$} ([yshift=2pt]A'.east)
(C.east) edge node[above] {$\kappa_1$} (A.west)
([yshift=-4pt]A.east) edge node[below] {$\alpha_1$} ([yshift=-4pt]B.west)
([yshift=2pt]B.west) edge node[above] {$e_1$} ([yshift=2pt]A.east);
\end{tikzpicture} 
 \end{center}
 where $\kappa_i \colon X_i \to A_i$ is the kernel of $\alpha_i \colon A_i \to B_i$ and $\alpha_i \cdot e_i = Id_{B_i}$ for any $i \in \{ 1,2\}$, then $p$ is an isomorphism whenever $v$ and $g$ are.
\end{definition}

\begin{definition}
A category is \emph{exact} in the sense of Barr \cite{Ba} if it is regular and any equivalence relation is effective.
\end{definition}

\begin{definition}
A category $\mathcal{C}$ is \emph{semi-abelian} when $\mathcal{C}$ is pointed, exact, protomodular 
 and $\mathcal{C}$ admits binary coproducts.
\end{definition}
We recall that a category $\mathcal{C}$ is abelian whenever $\mathcal{C}$ and $\mathcal{C}^{op}$ are semi-abelian.

The category of groups $\mathsf{Grp}$, the category of Lie algebras $\mathsf{Lie_K}$, the category of crossed modules $\mathsf{Xmod}$ are examples of semi-abelian categories. Another example is given by the following theorem proved in \cite{GSV} (see also \cite{GKV} for a special case of
this result).
\begin{theorem}
When $K$ is a field, $\sf Hopf_{K,coc}$ is a semi-abelian category.
\end{theorem}
The following result of Takeuchi \cite{Takeuchi} follows from this theorem (see \cite{GSV} for a proof of this corollary):
\begin{corollary}\label{ab}
The category of commutative and cocommutative Hopf algebras $\sf Hopf_{K,coc,comm}$ is abelian.
\end{corollary}

\subsection{Internal structures}

Let $\mathcal{C}$ be a category with pullbacks, in this category we recall how to define a reflexive graph, a reflexive-multiplicative graph and a groupoid (see \cite{CPP}).

 A  graph 
\begin{tikzpicture}[descr/.style={fill=white},yscale=1.2,baseline=(A.base)]
\node (A) at (0,0) {$A$};
\node (D) at (2,0) {$B$};
 \path[->,font=\scriptsize]
([yshift=-2pt]A.east) edge node[below] {$\delta$} ([yshift=-2pt]D.west)
([yshift=2pt]D.west) edge node[above] {$\iota$} ([yshift=2pt]A.east);
\end{tikzpicture}
 in $\mathcal{C}$ is called a split epimorphism if $\delta \cdot \iota = Id_B$. A morphism between two split epimorphisms 
\begin{tikzpicture}[descr/.style={fill=white},yscale=1.2,baseline=(A.base)]
\node (A) at (0,0) {$A'$};
\node (D) at (2,0) {$B'$};
 \path[->,font=\scriptsize]
([yshift=-2pt]A.east) edge node[below] {$\delta'$} ([yshift=-2pt]D.west)
([yshift=2pt]D.west) edge node[above] {$\iota'$} ([yshift=2pt]A.east);
\end{tikzpicture}
 and 
\begin{tikzpicture}[descr/.style={fill=white},yscale=1.2,baseline=(A.base)]
\node (A) at (0,0) {$A$};
\node (D) at (2,0) {$B$};
 \path[->,font=\scriptsize]
([yshift=-2pt]A.east) edge node[below] {$\delta$} ([yshift=-2pt]D.west)
([yshift=2pt]D.west) edge node[above] {$\iota$} ([yshift=2pt]A.east);
\end{tikzpicture}
is a pair of morphisms of $\mathcal{C}$ $(f \colon A' \to A, g \colon B' \to B)$ such that the two squares of the following diagram commute
\begin{center}
\begin{tikzpicture}[descr/.style={fill=white},scale=1.2]
\node (A) at (0,0) {$A$};
\node (B) at (0,1.5) {$A'$};
\node (C) at (3,1.5) {$B'$};
\node (D) at (3,0) {$B.$};
 \path[->,font=\scriptsize]
 (C.south) edge node[right] {$g$} (D.north)
  (B.south) edge node[right] {$f$} (A.north)
([yshift=-2pt]A.east) edge node[below] {$\delta$} ([yshift=-2pt]D.west)
([yshift=2pt]D.west) edge node[above] {$\iota$} ([yshift=2pt]A.east)
([yshift=-2pt]B.east) edge node[below] {$\delta'$} ([yshift=-2pt]C.west)
([yshift=2pt]C.west) edge node[above] {$\iota'$} ([yshift=2pt]B.east);
\end{tikzpicture}
\end{center}
They form the category of split epimorphisms (also called ``points") denoted by $\sf Pt(\mathcal{C})$.

\begin{definition}
A \emph{reflexive graph} $\mathbb A$
 in a category $\mathcal{C}$ is a diagram in $\mathcal{C}$ \begin{equation}\label{graph}
 \begin{tikzpicture}[descr/.style={fill=white},scale=1.2,baseline=(A.base)]
\node (A) at (0,0) {$A_1$};
\node (D) at (3,0) {$A_0$};
 \path[->,font=\scriptsize]
([yshift=7pt]A.east) edge node[above] {$\delta$} ([yshift=7pt]D.west)
(D.west) edge node[descr] {$\iota$} (A.east)
([yshift=-7pt]A.east) edge node[below] {$\gamma$} ([yshift=-7pt]D.west);
\end{tikzpicture},
\end{equation}
where $\delta$ is the ``domain morphism'', $\gamma$ the ``codomain morphism'', $\iota$ the ``identity morphism'' such that $\delta \cdot \iota = \gamma \cdot \iota = Id_{A_0}$.
 \end{definition}
 We follow the usual terminology of ``reflexive graph'' (as in \cite{CPP}, for instance). This notion is also called a ``reflexive pair'' in the literature.
 \begin{definition}\label{morph reflexive graph}
 \emph{A morphism of reflexive graphs} between 
 \begin{tikzpicture}[descr/.style={fill=white},yscale=1,xscale=0.8,baseline=(A.base)]
\node (A) at (0,0) {$A'_1$};
\node (D) at (3,0) {$A'_0$};
 \path[->,font=\scriptsize]
(D.west) edge node[descr] {$\iota'$} (A.east)
([yshift=7pt]A.east) edge node[above] {$\delta'$} ([yshift=7pt]D.west)
([yshift=-7pt]A.east) edge node[below] {$\gamma'$} ([yshift=-7pt]D.west);
\end{tikzpicture} and
 \begin{tikzpicture}[descr/.style={fill=white},yscale=1,xscale=0.8,baseline=(A.base)]
\node (A) at (0,0) {$A_1$};
\node (D) at (3,0) {$A_0$};
 \path[->,font=\scriptsize]
(D.west) edge node[descr] {$\iota$} (A.east)
([yshift=7pt]A.east) edge node[above] {$\delta$} ([yshift=7pt]D.west)
([yshift=-7pt]A.east) edge node[below] {$\gamma$} ([yshift=-7pt]D.west);
\end{tikzpicture} 
is given by a pair $(f_1 \colon A'_1 \to A_1 ,f_0 \colon A'_0 \to A_0) $ such that the three squares of the following diagram commute
\begin{center}
\begin{tikzpicture}[descr/.style={fill=white},scale=1.1]
\node (A) at (0,0) {$A_1$};
\node (B) at (0,1.5) {$A'_1$};
\node (C) at (3,1.5) {$A'_0$};
\node (D) at (3,0) {$A_0.$};
 \path[->,font=\scriptsize]
 (C.south) edge node[right] {$f_0$} (D.north)
  (B.south) edge node[right] {$f_1$} (A.north)
(D.west) edge node[descr] {$\iota$} (A.east)
(C.west) edge node[descr] {$\iota'$} (B.east)
([yshift=7pt]A.east) edge node[above] {$\delta$} ([yshift=7pt]D.west)
([yshift=-7pt]A.east) edge node[below] {$\gamma$} ([yshift=-7pt]D.west)
([yshift=7pt]B.east) edge node[above] {$\delta'$} ([yshift=7pt]C.west)
([yshift=-7pt]B.east) edge node[below] {$\gamma'$} ([yshift=-7pt]C.west);
\end{tikzpicture}
\end{center}
 \end{definition}
 \begin{definition}
In a category $\mathcal{C}$ with pullbacks, a reflexive graph 
 \begin{tikzpicture}[descr/.style={fill=white},yscale=1.2,baseline=(A.base)]
\node (A) at (0,0) {$A_1$};
\node (D) at (3,0) {$A_0$};
 \path[->,font=\scriptsize]
([yshift=7pt]A.east) edge node[above] {$\delta$} ([yshift=7pt]D.west)
(D.west) edge node[descr] {$\iota$} (A.east)
([yshift=-7pt]A.east) edge node[below] {$\gamma$} ([yshift=-7pt]D.west);
\end{tikzpicture} 
together with a morphism $m \colon A_1 \times_{A_0} A_1 \to A_1$ is called a \emph{reflexive-multiplicative graph} 
\begin{equation}\label{mgraph}
 \begin{tikzpicture}[descr/.style={fill=white},yscale=1.2,baseline=(A.base)]
\node (A) at (0,0) {$A_1$};
\node (D) at (3,0) {$A_0$};
\node (X) at (-3,0) {$A_1\times_{A_0}A_1$};
 \path[->,font=\scriptsize]
 (X.east) edge node[above] {$m$} (A.west)
([yshift=7pt]A.east) edge node[above] {$\delta$} ([yshift=7pt]D.west)
(D.west) edge node[descr] {$\iota$} (A.east)
([yshift=-7pt]A.east) edge node[below] {$\gamma$} ([yshift=-7pt]D.west);
\end{tikzpicture} ,
\end{equation}
where  $A_1\times_{A_0}A_1$ is the (object part of the) following pullback
\[
\xymatrix{A_1\times_{A_0} A_1 \ar[d]_{p_1} \ar[r]^-{p_2} & A_1 \ar[d]^{ \gamma} \\
A_1 \ar[r]_-{\delta}& A_0,
}
\] and $m$ is a multiplication that is required to satisfy the identities
\begin{equation}\label{RGM}
m \cdot (Id_{A_1} , \iota \cdot \delta ) = Id_{A_1} = m \cdot (\iota \cdot \gamma, Id_{A_1} ),
\end{equation} 
 where $(Id_{A_1} , \iota \cdot \delta ) \colon A_1 \to A_1\times_{A_0}A_1 $ and $(\iota \cdot \gamma, Id_{A_1} ) \colon A_1 \to A_1\times_{A_0}A_1 $ are induced by the universal property of the pullback $A_1\times_{A_0}A_1$.
 \end{definition}
 
  \begin{definition}\label{grpd}
An \emph{internal groupoid} is a reflexive-multiplicative graph such that the multiplication $m$ satisfies the additional identities
 \begin{equation*}
 \delta \cdot m = \delta \cdot { p_2},
 \end{equation*}
 \begin{equation*}
 \gamma \cdot m = \gamma \cdot { p_1},
 \end{equation*}
 \begin{equation*}
 m \cdot (1_{A_1},m) = m \cdot (m,1_{A_1}),
 \end{equation*}
 and there exists a morphism $i \colon A_1 \to A_1$  (playing the role of ``inverse") such that \begin{equation*}
 \delta  \cdot i = \gamma,
 \end{equation*}
 \begin{equation*}
 \gamma \cdot i = \delta,
 \end{equation*}
 \begin{equation*}
 m \cdot (i , Id_{A_1}) = \iota \cdot {\delta},
 \end{equation*}
 \begin{equation*}
 m \cdot (Id_{A_1}, i) = \iota \cdot {\gamma}.
 \end{equation*}
\end{definition}
Let $\C$ be a category with pullbacks, we denote by $\mathsf{RMG}({\C})$ the category of reflexive-multiplicative graphs with morphisms of reflexive-multiplicative graphs, which are morphisms of reflexive graphs preserving the multiplication. The category of internal groupoids with the morphisms of reflexive multiplicative graphs will be denoted by $\mathsf{Grpd}(\C)$.
It is well-known \cite{CPP} that in any pointed Mal’tsev category having the\textit{ Smith is Huq} property \cite{BG}, for a reflexive graph \eqref{graph}, the following properties are equivalent
  \begin{itemize}
  \item the reflexive graph \eqref{graph} is a reflexive-multiplicative graph
  \item the reflexive graph \eqref{graph} is an internal groupoid
  \item the reflexive graph \eqref{graph} satisfies the condition $[Ker(\delta),Ker(\gamma)] = 0$,  which means that $Ker(\delta)$ and $Ker(\gamma)$ commute. One says that two subobjects $x \colon X\rightarrow A $ and $y \colon Y \rightarrow A$ of the same object $A$ commute (in the sense of Huq) \cite{Huq} if and only if there exists an arrow $p$ making the following diagram commute:

\begin{equation*}
\begin{tikzpicture}[descr/.style={fill=white},scale=0.9]
\node (A) at (0,0) {$A.$};
\node (C) at (2,2) {$Y$};
\node (D) at (0,2) {$X \times Y$};
\node (B) at (-2,2) {$X$};
  \path[-stealth]
 (B.south) edge node[left] {${x\,} $} (A.north west) 
 (C.south) edge node[right] {${\, y} $} (A.north east) 
 (B.east) edge node[above] {$(1,0)$} (D.west) 
 (C.west) edge node[above] {$(0,1)$} (D.east);
   \path[-stealth,dashed]
 (D.south) edge node[right] {$p$} (A.north);
\end{tikzpicture}
\end{equation*}

\end{itemize}

This theory can be applied to the category $\sf Hopf_{K,coc}$. Thanks to a characterization of the commutation (in the sense of Huq) of two cocommutative Hopf subalgebras in \cite{GSV}, a reflexive graph satisfying the condition $[HKer(\delta),HKer(\gamma)]=0$ is a cat$^1$-Hopf algebra, in the sense of \cite{Vilaboa}, i.e.\ a reflexive graph \eqref{graph} such that 
  \[ ab=ba \] 
  for any $a \in HKer(\delta)$ and any $ b \in Hker(\gamma)$.
Since $\sf Hopf_{K,coc}$ is a Mal'tsev category which possesses the \emph{Smith is Huq} property,  these internal structures yield to three  full subcategories of the category of reflexive graphs and their morphisms (see \cite{CPP}): 
   \begin{itemize}
   \item The cat$^1$-Hopf algebras with morphisms of reflexive graphs form a category, denoted by $\sf{cat^1}(\Hopf_{K, coc})$;
   \item The reflexive-multiplicative graphs, with morphisms of reflexive graphs, form a category, that will be denoted by $\sf{RMG}(\Hopf_{K, coc})$;
   \item The category of internal groupoids where morphisms are the reflexive graph morphisms, will be denoted by $\sf{Grpd}(\Hopf_{K, coc})$. 
   \end{itemize}
   
   In conclusion, as observed in Proposition 5.4 of \cite{GSV}, we have the following proposition by applying the result of \cite{CPP} to the category of cocommutative Hopf algebras.
  \begin{proposition}\label{Iso groupoid RGM}
The following categories are all isomorphic
\begin{itemize}
\item $\sf{cat^1}(\Hopf_{K, coc})$,
\item $\sf{RMG}(\Hopf_{K, coc})$,
\item $\sf{Grpd}(\Hopf_{K, coc})$.
\end{itemize}
\end{proposition}
  \subsection{Hopf crossed modules and internal groupoids of Hopf algebras}
  In this subsection we recall how internal structures in $\sf Hopf_{K,coc}$ are related to crossed structures. 
  Recall that for a cocommutative Hopf algebra $B$, the category of left $B$-modules is symmetric monoidal and the forgetful functor to vector spaces is symmetric monoidal. Hence, one can consider a Hopf algebra in the category of left $B$-modules, which is then called a left $B$-module Hopf algebra. Explicitly, a (cocommutative) $B$-module Hopf algebra $X$ is a (cocommutative) Hopf algebra $X$ endowed with a linear map 
$ \xi \colon B \otimes X \to X,\ \xi (b \otimes x)  = b \triangleright x$, satisfying the following identities

\begin{equation}\label{m et act}
{(bb'} \triangleright x) = {b}\triangleright({b'}\triangleright x) ,
\end{equation}
\begin{equation}\label{1 et act}
 {1_B} \triangleright x=x ,
 \end{equation}
 \begin{equation}\label{act et m}
 (b  \triangleright xy) = ({b_1} \triangleright x) ({b_2}  \triangleright y ),
 \end{equation}
\begin{equation}\label{act et 1}
 b  \triangleright 1_X = \epsilon(b)1_X ,
 \end{equation}
\begin{equation}\label{act et delta}
(b \triangleright x)_1 \otimes (b \triangleright x)_2 = {b_1} \triangleright x_1 \otimes {b_2} \triangleright x_2,
 \end{equation}
\begin{equation}\label{act et epsilon}
 \epsilon(b  \triangleright x) = \epsilon(b)\epsilon(x),
\end{equation}
for any $b, b' \in B$, and $x,y \in X$.
 
 Let $X$ and $Y$ be $B$-module Hopf algebras, a Hopf algebra morphism $\lambda \colon X \to Y$ is called a morphism of $B$-module Hopf algebras if it is a $B$-linear map, i.e.\ $\lambda ( b \triangleright x) = b \triangleright \lambda(x)$, $\forall x \in X,$ and $\forall b \in B$. From these definitions we can describe the category $\sf Act(\mathsf{Hopf}_{K,coc})$ where the objects are the pairs $(B,X)$ where $B$ is a  cocommutative Hopf algebra and $X$ is a cocommutative $B$-module Hopf algebra, and the morphisms are the pairs $(f,g) \colon (B,X) \to (B',X')$ of morphisms of Hopf algebras $f \colon B \to B'$, $g \colon X \to X'$ such that \[g(b \act x) =  f(b) \act g(x).\]
 
 From an object in $\sf Act(Hopf_{K,coc})$, we can obtain one important construction. More precisely, from a cocommutative Hopf algebra $B$ and a cocommutative $B$-module Hopf algebra $X$, we can construct $X \rtimes B$, called {\emph semi-direct product} (or \textit{smash product})  \cite{Molnar}, the vector space $X \ox B$ endowed with the following Hopf algebra structure
 
\begin{align*}
& M_{X \rtimes B}((x\ox b)\ox (y \ox c))= x ({b_1} \act y) \ox b_2 c \\
& \Delta_{X \rtimes B}(x \ox b) = x_1 \ox b_1 \ox x_2 \ox b_2 \\
& 1_{X \rtimes B} = 1_X \ox 1_B\\
& \epsilon_{X \rtimes B} (x \ox b) = \epsilon(x)\epsilon(b)\\
&S_{X \rtimes B}(x \otimes b) = {S_B(b_1)} \act S_X(x) \otimes S_B(b_2)
\end{align*} 

This structure is a key element in the following result \cite{Molnar}, which gives a relation between module Hopf algebras and split epimorphisms. This result was essentially extended by Radford in \cite{Radford} by describing split epimorphisms in the category of Hopf algebras over a field in term of the so-called Radford's biproduct (later called ``bosonization'' by Majid \cite{MajidB}.)
 \begin{lemma}\label{iso crossed product}
 Let
\begin{tikzpicture}[descr/.style={fill=white},yscale=1.2,baseline=(A.base)]
\node (A) at (0,0) {$A$};
\node (D) at (2,0) {$B$};
 \path[->,font=\scriptsize]
([yshift=-2pt]A.east) edge node[below] {$\delta$} ([yshift=-2pt]D.west)
([yshift=2pt]D.west) edge node[above] {$\iota$} ([yshift=2pt]A.east);
\end{tikzpicture}
 be a split epimorphism in $\mathsf{Hopf_{K,coc}} $. Then, $\HKer(\delta)$ is a $B$-module Hopf algebra for the linear map $\xi \colon B \otimes \HKer(\delta) \to \HKer(\delta)$ defined by
\[\xi (b \otimes k) = \iota(b_1)k\iota(S(b_2)),\]  $\forall b \in B, \forall k \in X$. There exists a natural isomorphism \[\HKer(\delta) \rtimes B \cong A,\] which yields an equivalence of categories between the category  $\mathsf{Pt(Hopf_{K,coc})}$ of split epimorphisms in $\Hopf_{K,coc}$
 and the category $\mathsf{Act(Hopf_{K,coc})}$ of module Hopf algebras.
 \end{lemma}
 \begin{proof} 
The isomorphism between  $\HKer(\delta) \rtimes B $ and $A$ is given by \[ \Phi \colon \HKer(\delta) \rtimes B \to A \colon k \ox b \mapsto k \iota(b).\]
 We recall the functors inducing the equivalence of categories.
 The functor 

\[ H \colon  \mathsf{Pt(Hopf_{K,coc})} \to \mathsf{Act(Hopf_{K,coc})}\colon 
\begin{tikzpicture}[descr/.style={fill=white},yscale=1.2,baseline=(A.base)]
\node (A) at (0,0) {$A$};
\node (D) at (2,0) {$B$};
 \path[->,font=\scriptsize]
([yshift=-2pt]A.east) edge node[below] {$\delta$} ([yshift=-2pt]D.west)
([yshift=2pt]D.west) edge node[above] {$\iota$} ([yshift=2pt]A.east);
\end{tikzpicture}
 \mapsto (\xi \colon B \otimes \HKer(\delta) \to \HKer(\delta))\]
which sends a split epimorphism of cocommutative Hopf algebras to a $B$-module Hopf algebra,  and the functor $I$ below, which sends a $B$-module Hopf algebra to a split epimorphism of cocommutative Hopf algebras  \[
  I \colon  \mathsf{Act(Hopf_{K,coc})} \to \mathsf{Pt(Hopf_{K,coc})} \colon (\rho \colon B \ox X \to X) \mapsto 
\begin{tikzpicture}[descr/.style={fill=white},yscale=1.2,baseline=(A.base)]
\node (A) at (0,0) {$X \rtimes B$};
\node (D) at (2.5,0) {$B$};
 \path[->,font=\scriptsize]
([yshift=-2pt]A.east) edge node[below] {$p_2$} ([yshift=-2pt]D.west)
([yshift=2pt]D.west) edge node[above] {$e$} ([yshift=2pt]A.east);
\end{tikzpicture},
\] 
where $e$ is defined by $e(b) = 1_X \ox b$ and $p_2$ is the projection, $p_2(x\ox b) = \epsilon(x) b$.
\end{proof}
It is well-known that there is an equivalence as in Lemma \ref{iso crossed product}, in any semi-abelian category. The interested readers can find some variations of this result for monoids in \cite{BMS}, for monoids in symmetric monoidal category and for magmas in \cite{GJS}. This last paper motivated the case of non-cocommutative (non-associative) Hopf algebras, which is studied in \cite{Sterck}. We refer also to \cite{Bohm} for a different approach.

 Definitions of Hopf crossed modules were given independently in \cite{Majid} and \cite{Vilaboa}. These definitions are not equivalent in the general case. However, if we consider cocommutative Hopf algebras they coincide. We spell out the definiton in the cocommutative case.
 \begin{definition}
In $\mathsf{Hopf_{K,coc}}$, a \emph{Hopf crossed module} is a triple $(B, X, d)$ where $B$ is a cocommutative Hopf algebra, $X$ is a cocommutative $B$-module Hopf algebra and $d \colon X\to B$ is a Hopf algebra morphism satisfying
\[(CM1) \; d(b  \triangleright x) = b_1d(x)S(b_2),\]
\[(CM2) \; {d(y)} \triangleright x = y_1xS(y_2),\]
for any $x, y \in X$ and $b \in B$.
\end{definition}
Let $\mathsf{HXMod_{K, coc}}$ be the category of Hopf crossed modules, where a morphism \[(\alpha, \beta) \colon (B,X,d) \to (B',X',d')\] is a pair of Hopf algebra morphisms $\alpha \colon X \to X'$ and $\beta \colon B \to B'$ such that $d' \cdot \alpha = \beta \cdot d$ and $\alpha ({b}  \triangleright x ) = \, {\beta(b)}  \triangleright \alpha (x)$.
\begin{proposition}\label{equi xmod 1}
There is an equivalence of categories between $\mathsf{HXMod_{K, coc}}$ and $\sf{Grpd}(\Hopf_{K, coc})$.
\end{proposition}
\begin{proof}
The theorem was proved in \cite{GSV} (see also \cite{Bohm,Vilaboa}), we recall the construction of the functors since it will be frequently used in what follows. Thanks to Proposition \ref{Iso groupoid RGM}, it is enough to prove that $\mathsf{RMG}(\Hopf_{K, coc})$ and $\mathsf{HXMod_{K, coc}}$ are equivalent. 
With a reflexive multiplicative graph \eqref{graph} we associate the Hopf algebra morphism \[ d = \gamma \cdot {hker} (\delta) \colon \HKer (\delta) \to A_0\]
(where ${hker} (\delta) \colon \HKer(\delta) \to A_1$ is the kernel of $\delta$ in $\mathsf{Hopf_{K,coc}}$), equipped with $\HKer(\delta)$ the $A_0$-module, $A_0 \otimes \HKer(\delta) \to \HKer(\delta)$ defined by 
\[a \triangleright  k = \iota(a_1) k \iota(S(a_2)), \]
for any $a \in A_0$ and $k \in \HKer(\delta)$. 
Conversely, given a Hopf crossed module $(B, X, d)$, we define the reflexive graph
\begin{equation}\label{reflexive}
  \begin{tikzpicture}[descr/.style={fill=white},scale=1.2,baseline=(A.base)]
\node (A) at (0,0) {$X \rtimes B$};
\node (D) at (3,0) {$B$};
 \path[->,font=\scriptsize]
([yshift=7pt]A.east) edge node[above] {$p_2$} ([yshift=7pt]D.west)
(D.west) edge node[descr] {$e$} (A.east)
([yshift=-7pt]A.east) edge node[below] {$p_1$} ([yshift=-7pt]D.west);
\end{tikzpicture} ,
\end{equation}
where $p_2$ is the second projection $p_2 (x \otimes b) = \epsilon (x) b$, the morphisms $p_1$ and $e$ are defined by $p_1(x \otimes b)= d(x)b$ and $e(b) = 1_X \otimes b$ respectively. This reflexive graph is equipped with a multiplication by setting \[m(x \otimes b, x' \otimes b') = xx' \otimes \epsilon (b) b',\]
 for any $(x \otimes b, x' \otimes b')$ in the pullback $P$ defined by 
 \[\xymatrix{P \ar[r]^-{\pi_2} \ar[d]_{\pi_1} &  X \rtimes B \ar[d]^{p_1} \\
 X \rtimes B \ar[r]_-{p_2} & B.} \]
The correspondence described above naturally extends to morphisms, yielding two functors  
\begin{align*}
F &\colon \mathsf{RMG}(\Hopf_{K, coc}) \to \mathsf{HXMod_{K, coc}} ,\\
G &\colon \mathsf{HXMod_{K, coc}} \to \mathsf{RMG}(\Hopf_{K, coc}).
\end{align*} 
  
These functors give rise to an equivalence of categories.
\end{proof}
In \cite{Jan}, the author defined the internal crossed modules in any semi-abelian category $\C$, noted $\sf XMod(\C)$ and proved that they are equivalent to internal groupoids in $\C$. Hence, Theorem \ref{equi xmod 1} implies that the category $\mathsf{HXMod_{K, coc}}$ is equivalent to the category $\sf XMod(Hopf_{K,coc})$ of internal crossed modules in the category of cocommutative Hopf algebras (see \cite{GSV}).

\subsection{Cartier-Gabriel-Konstant-Milnor-Moore theorem}
The semi-direct product of cocommutative Hopf algebras is a useful tool allowing us to decompose any cocommutative Hopf algebras (over an algebraically closed field) into a ``group part" and a ``Lie algebra part". More precisely, we recall how we can decompose a Hopf algebra in a \textit{group-like} part and a \textit{primitive} part, thanks to the following functors,
\begin{align*}
&\mathcal{G} \colon \mathsf{Hopf_{K,coc}} \to \mathsf{Grp}, H \to G_H = \{ x \in H \mid \Delta(x) = x \otimes x, \epsilon(x) = 1 \},\\
&\mathsf{K}[-] \colon \mathsf{Grp} \to \mathsf{Hopf_{K,coc}} \colon G \to \mathsf{K}[G],\\
&\mathcal{P} \colon \mathsf{Hopf_{K,coc}} \to \mathsf{LieAlg_K}, H \to L_H = \{x \in H \mid \Delta(x) = 1_H \otimes x + x \otimes 1_H\},\\
&\mathsf{U} \colon \mathsf{LieAlg_K} \to \mathsf{Hopf_{K,coc}} \colon  L \to \mathsf{U}(L),
\end{align*} 
where $\mathsf{K}[G]$ denotes the group Hopf algebra of the group $G$, the Hopf algebra structure is given by $\Delta(g) = g \otimes g$, $\epsilon(g)=1$ and $S(g)=g^{-1}$ for any $g \in G$,
and $\mathsf{U}(L)$ is the universal enveloping algebra of the Lie algebra $L$ with a Hopf algebra structure given by $\Delta(x) = 1_L \otimes x + x \otimes 1_L $, $\epsilon(x) =0 $ and $S(x) = -x$ for any $x \in L$. 
The above functors can be summed up to the following adjunctions.
\begin{equation}\label{adj}
\begin{tikzpicture}[descr/.style={fill=white},scale=1.2,baseline=(A.base)]
\node at (-1.5,0) {$\perp$};
\node at (1.5,0) {$\perp$};
\node (A) at (0,0) {$\mathsf{Hopf_{K,coc}}$};
\node (B) at (3,0) {$\mathsf{LieAlg}_K.$};
\node (C) at (-2.5,0) {$\mathsf{Grp}$};
 \path[->,font=\scriptsize]
([yshift=5pt]C.east) edge node[above] {$\mathsf{K}[-]$} ([yshift=5pt]A.west)
 ([yshift=-5pt]A.west) edge node[below] {$\mathcal{G}$}([yshift=-5pt]C.east) 
 ([yshift=5pt]B.west) edge node[above] {$\mathsf{U}$}([yshift=5pt]A.east) 
([yshift=-5pt]A.east) edge node[below] {$\mathcal{P}$} ([yshift=-5pt]B.west);
\end{tikzpicture}
\end{equation}
Moreover, $\mathsf{U}(L_H)$ can be seen as a $\mathsf{K}[G_H]$-module thanks to the following linear map
\[ g \triangleright x = gxS(g),\]
for any $x \in L_H$ and $g \in G_H$. 

 The following Theorem
 gives us a way to decompose any cocommutative Hopf algebra over an algebraically closed field $K$ of characteristic zero in a  \textit{group-like} part and a \textit{primitive} part \cite{MM}.
 
 \begin{theorem}[Cartier-Gabriel-Konstant-Milnor-Moore] \label{CGKMM}
Let  $H$ be a cocommutative Hopf algebra over an algebraically closed field $K$ of characteristic zero. We have the following isomorphism   \[  U(L_H) \rtimes K[G_H] \to H \colon x \ox g \mapsto xg.\]
\end{theorem}
\subsection{Double internal groupoids and cat$^2$-Hopf algebras}
Since $\mathsf{HXMod_{K, coc}}$  is semi-abelian, we get the following result
\begin{lemma}\label{equiv preli}
We have the following equivalence of categories,
\begin{align*}
\mathsf{XMod}(\mathsf{XMod(\Hopf_{K, coc}})) &\cong \mathsf{XMod}(\mathsf{HXMod_{K, coc}}) \cong \sf{Grpd}(\mathsf{HXMod_{K, coc}}) \\ &\cong \sf{Grpd}(\mathsf{Grpd}(\Hopf_{K, coc})) = \sf{Grpd}^2(\Hopf_{K, coc}).
\end{align*}
\end{lemma}
Here $\mathsf{Grpd}^2(\Hopf_{K, coc})$ denotes the category of double internal groupoids in $\Hopf_{K, coc}$.
In the next section, we will describe the category $\mathsf{XMod}(\mathsf{HXMod_{K, coc}}) $. But first, let us give a description of $\sf{Grpd}^2(\Hopf_{K,coc})$ the category of double internal groupoids. A  \textit{double internal groupoid} of Hopf algebras is an internal groupoid in $\sf{Grpd}(\Hopf_{K, coc})$, i.e.\ it is defined by a commutative diagram
\begin{equation}\label{double groupoid}
\begin{tikzpicture}[descr/.style={fill=white},scale=1.2]
\node (A) at (0,0) {$H_1$};
\node (B) at (0,3) {$H_{\{1,2\}}$};
\node (C) at (3,3) {$H_2$};
\node (D) at (3,0) {$H_\emptyset ,$};
 \path[->,font=\scriptsize]
 (D.west) edge node[descr] {$\iota'_2$} (A.east)
 ([xshift=7pt]B.south) edge node[right] {$\gamma_1$} ([xshift=7pt]A.north)
(A.north) edge node[descr] {$\iota_1$} (B.south)
([xshift=-7pt]B.south) edge node[left] {$\delta_1$} ([xshift=-7pt]A.north)
 ([xshift=7pt]C.south) edge node[right] {$\gamma'_1$} ([xshift=7pt]D.north)
(D.north) edge node[descr] {$\iota'_1$} (C.south)
([xshift=-7pt]C.south) edge node[left] {$\delta'_1$} ([xshift=-7pt]D.north)
([yshift=7pt]A.east) edge node[above] {$\gamma'_2$} ([yshift=7pt]D.west)
([yshift=-7pt]A.east) edge node[below] {$\delta'_2$} ([yshift=-7pt]D.west)
([yshift=7pt]B.east) edge node[above] {$\gamma_2$} ([yshift=7pt]C.west)
(C.west) edge node[descr] {$\iota_2$} (B.east)
([yshift=-7pt]B.east) edge node[below] {$\delta_2$} ([yshift=-7pt]C.west);
\end{tikzpicture}
\end{equation} where the $(H_{\{1,2\}},H_1,\gamma_1, \delta_1, \iota_1)$, $(H_{\{1,2\}},H_2,\gamma_2, \delta_2, \iota_2)$, $(H_{2},H_\emptyset,\gamma'_1, \delta'_1, \iota'_1)$ and $(H_1,H_\emptyset,\gamma'_2, \delta'_2, \iota'_2)$ are the reflexive graphs underlying internal groupoid structures. Thanks to the isomorphism of Proposition \ref{Iso groupoid RGM} we can prove that the notion of double internal groupoids is equivalent to the notion of \textit{cat$^2$-Hopf algebras} that is defined as follows. 
\begin{definition}\label{Def_cat2}
A \emph{cat$^2$-Hopf algebra} is a 7-tuple $(H,N,M,s_N,t_N,s_M,t_M)$ where 
 \begin{tikzpicture}[descr/.style={fill=white},xscale=0.75,baseline=(A.base)]
\node (A) at (0,0) {$H$};
\node (D) at (3,0) {$N$};
 \path[->,font=\scriptsize]
(D.west) edge node[descr] {$i_N$} (A.east)
([yshift=7pt]A.east) edge node[above] {$s_N$} ([yshift=7pt]D.west)
([yshift=-7pt]A.east) edge node[below] {$t_N$} ([yshift=-7pt]D.west);
\end{tikzpicture}
 and 
   \begin{tikzpicture}[descr/.style={fill=white},xscale=0.75,baseline=(A.base)]
\node (A) at (0,0) {$H$};
\node (D) at (3,0) {$M$};
 \path[->,font=\scriptsize]
(D.west) edge node[descr] {$i_M$} (A.east)
([yshift=7pt]A.east) edge node[above] {$s_M$} ([yshift=7pt]D.west)
([yshift=-7pt]A.east) edge node[below] {$t_M$} ([yshift=-7pt]D.west);
\end{tikzpicture}
are  cat$^1$-Hopf algebras, where the injective morphisms $i_N \colon N \to H$ and $i_M \colon M \to H$ are from now on regarded as Hopf subalgebra inclusions and which will be omitted, satisfying the following four conditions of compatibility 
\begin{align*}
&(2C1) \;s_N \cdot s_M = s_M \cdot s_N,\\
&(2C2)\; t_N \cdot t_M = t_M \cdot t_N,\\
&(2C3) \;s_N \cdot t_M = t_M \cdot s_N,\\
&(2C4) \;t_N \cdot s_M = s_M \cdot t_N.\\
\end{align*} 
   
\end{definition}

\begin{definition}
A \emph{morphism of cat$^2$-Hopf algebras} from the cat$^2$-Hopf algebra $(H,N,M,s_N,t_N$, $s_M,t_M)$ to the cat$^2$-Hopf algebra $(H',N',M',s'_{N'},t'_{N'},s'_{M'},t'_{M'})$ is a morphism of Hopf algebras $f \colon H \to H'$ such that the image of $f|_N := f \cdot i_N$ ($f|_M := f \cdot i_M$) lies in $N'$ ($M'$) and the four following diagrams commute
   
\begin{center}
\begin{tikzpicture}[descr/.style={fill=white},scale=0.9]
\node (A) at (0,0) {$H'$};
\node (B) at (0,2) {$H$};
\node (C) at (2,2) {$N$};
\node (D) at (2,0) {$N'$};
  \path[-stealth]
 (B.south) edge node[right] {$f$} (A.north) 
 (C.south) edge node[right] {$f|_N$} (D.north);
 \path[->,font=\scriptsize]
([yshift=4pt]A.east) edge node[above] {$s'_{N'}$} ([yshift=4pt]D.west)
([yshift=-4pt]A.east) edge node[below] {$t'_{N'}$} ([yshift=-4pt]D.west)
([yshift=4pt]B.east) edge node[above] {$s_N$} ([yshift=4pt]C.west)
([yshift=-4pt]B.east) edge node[below] {$t_N$} ([yshift=-4pt]C.west);
\end{tikzpicture}
\begin{tikzpicture}[descr/.style={fill=white},scale=0.9]
\node (A) at (0,0) {$H'$};
\node (B) at (0,2) {$H$};
\node (C) at (2,2) {$M$};
\node (D) at (2,0) {$M'.$};
  \path[-stealth]
 (B.south) edge node[right] {$f$} (A.north) 
 (C.south) edge node[right] {$f|_M$} (D.north);
 \path[->,font=\scriptsize]
([yshift=4pt]A.east) edge node[above] {$s'_{M'}$} ([yshift=4pt]D.west)
([yshift=-4pt]A.east) edge node[below] {$t'_{M'}$} ([yshift=-4pt]D.west)
([yshift=4pt]B.east) edge node[above] {$s_M$} ([yshift=4pt]C.west)
([yshift=-4pt]B.east) edge node[below] {$t_M$} ([yshift=-4pt]C.west);
\end{tikzpicture}
\end{center}
In other words, a morphism $f \colon H \to H'$ is a morphism of cat$^2$-Hopf algebras if and only if it is a morphism of cat$^1$-Hopf algebras for the two underlying cat$^1$-Hopf algebras.
\end{definition}
The category of cat$^2$-Hopf algebras is denoted by $\sf{cat}^2(\Hopf_{K,coc})$.
 Proposition \ref{Iso groupoid RGM} implies an isomorphism between the cat$^2$-Hopf algebras and the double groupoids of Hopf algebras. For the sake of clarity we explicitly describe this isomorphism.
\begin{proposition}\label{iso group cat2}
There is an isomorphism between $\sf{cat}^2(\Hopf_{K,coc})$ and $\sf{Grpd}^2(\Hopf_{K,coc})$.
\end{proposition}

\begin{proof}
From a cat$^2$-Hopf algebra $(H,N,M,s_N,t_N,s_M,t_M)$ we construct the double groupoid
\begin{center}
\begin{tikzpicture}[descr/.style={fill=white},scale=1.2]
\node (A) at (0,0) {$N$};
\node (B) at (0,3) {$H$};
\node (C) at (3,3) {$M$};
\node (D) at (3,0) {$N \cap M,$};
 \path[->,font=\scriptsize]
 ([xshift=8pt, yshift=-4pt]B.south) edge node[right] {$s_N$} ([xshift=8pt]A.north)
(A.north) edge node[descr] {$i_N$} (B.south)
([xshift=-8pt]B.south) edge node[left] {$t_N$} ([xshift=-8pt]A.north)
 ([xshift=7pt]C.south) edge node[right] {$s_N\cdot i_M$} ([xshift=7pt]D.north)
(D.north) edge node[above] {$ $} (C.south) 
([xshift=-7pt, yshift=-4pt]C.south) edge node[left] {$t_N\cdot i_M$} ([xshift=-7pt]D.north)
([yshift=7pt, xshift=4pt]A.east) edge node[above] {$s_M \cdot i_N$} ([yshift=7pt]D.west)
(D.west) edge (A.east)
([yshift=-7pt]A.east) edge node[below] {$t_M \cdot i_N$} ([yshift=-7pt]D.west)
([yshift=7pt]B.east) edge node[above] {$s_M$} ([yshift=7pt]C.west)
(C.west) edge node[descr] {$i_M$} (B.east)
([yshift=-7pt]B.east) edge node[below] {$t_M$} ([yshift=-7pt]C.west);
\end{tikzpicture}
\end{center}
where the maps $i$ are all inclusions. The Hopf algebra $ N \cap M$ is the intersection of $N$ and $M$ seen as Hopf subalgebras of $H$, it is equivalent to say that it is the pullback of $i_M$ along $i_N$. All the squares in this diagram commute thanks to the definition of cat$^2$-Hopf algebra. Note that this square is well-defined, for example, $t_N \cdot i_M$ lands into the intersection. Indeed, let $m$ be an element of $M$, then it is clear that $t_N(m) (= t_N \cdot i_M(m))$ belongs to $N$ and $t_N (m) = t_N \cdot t_M  (m) \;  \stackrel{\mathclap{(2C2)}}{=} \;  t_M \cdot t_N  (m)$ belongs to $M$.

Moreover, by using the correspondence between cat$^1$-Hopf algebras and internal groupoids, the above diagram is a double groupoid of cocommutative Hopf algebras since the kernels and their restrictions commute, 
\begin{eqnarray*}
&[HKer(s_N), HKer(t_N)] = 0,\\
&[HKer(s_M), HKer(t_M)] = 0,\\
&[HKer(s_N \cdot i_M), HKer(t_N  \cdot i_M)] = 0,\\
&[HKer(s_M  \cdot i_N), HKer(t_M  \cdot i_N)] = 0.
\end{eqnarray*}
 
On the other way around, if we have a double groupoid \eqref{double groupoid}, the $7$-tuple $(H_{\{1,2\}},H_1,H_2,\iota_1 \cdot \gamma_1,\iota_1 \cdot \delta_1,\iota_2 \cdot \gamma_2, \iota_2 \cdot \delta_2 )$ is a cat$^2$-Hopf algebra. Indeed, the kernels have trivial commutators since the top reflexive graph and the left one are internal groupoids. Moreover, we have all the conditions of compatibility between the four maps thanks to the commutativity of \eqref{double groupoid}. For example, the condition $(2C1)$ holds as we can see in the following equalities,
\begin{align*}
(\iota_1 \cdot \gamma_1) \cdot (\iota_2 \cdot \gamma_2) &= \iota_1 \cdot \iota'_2 \cdot \gamma'_1 \cdot \gamma_2 \\
&= \iota_2 \cdot \iota'_1 \cdot \gamma'_2 \cdot \gamma_1 \\
&= (\iota_2 \cdot \gamma_2) \cdot (\iota_1 \cdot \gamma_1) .
\end{align*} 
\end{proof}

To finish this recall about double internal structures, we give a description of the category $\sf Pt^2(Hopf_{K,coc})$ of 2-fold split epimorphisms of cocommutative Hopf algebras. In a similar way that the split epimorphisms are underlying structures of the internal groupoids, the 2-fold split epimorphisms of cocommutative Hopf algebras will be underlying structures of the cat$^2$-Hopf algebras (and the internal double groupoids of Hopf algebras).

\begin{definition}
A \emph{$2$-fold split epimorphism} of cocommutative Hopf algebras is a split epimorphism in $\mathsf{Pt}(\sf Hopf_{K,coc})$, the category of split epimorphism in $\sf Hopf_{K,coc}$, i.e.\ a $2$-fold split epimorphism is given by the following diagram, where all the squares commute and $\gamma_i \cdot \iota_i = Id$ 
\begin{equation} 
\begin{tikzpicture}[descr/.style={fill=white},scale=1.2]
\node (A) at (0,0) {$H_{\{2\}}$};
\node (B) at (0,3) {$H_{\{1,2\}}$};
\node (C) at (3,3) {$H_{\{1\}}$};
\node (D) at (3,0) {$H_{\{\emptyset\}}.$};
 \path[->,font=\scriptsize]
 ([xshift=-4pt]B.south) edge node[left] {$\gamma_1$} ([xshift=-4pt]A.north)
([xshift=4pt]A.north) edge node[right] {$\iota_1$} ([xshift=4pt]B.south)
 ([xshift=-4pt]C.south) edge node[left] {$\gamma'_1$} ([xshift=-4pt]D.north)
([xshift=4pt]D.north) edge node[right] {$\iota'_1$} ([xshift=4pt]C.south)
([yshift=4pt]A.east) edge node[above] {$\gamma'_2$} ([yshift=4pt]D.west)
([yshift=-4pt]D.west) edge node[below] {$i'_2$} ([yshift=-4pt]A.east)
([yshift=4pt]B.east) edge node[above] {$\gamma_2$} ([yshift=4pt]C.west)
([yshift=-4pt]C.west) edge node[below] {$\iota_2$} ([yshift=-4pt]B.east);
\end{tikzpicture}
\end{equation}
\end{definition}
Notice that by identification of $H$ with $H_{\{1,2\}}$, $N$ with $H_1$, $M$ with $H_2$ and $H_\emptyset$ with $N \cap M$ in $\sf Hopf_{K,coc}$, this definition is equivalent to two Hopf subalgebras $N$, $M$ of $H$ with 2 morphisms $s_N \colon H \to N$ and $s_M \colon H \to M$ such that $s_N\cdot i_N = Id_N$, $s_M \cdot i_M = Id_M$ and $s_N \cdot s_M = s_M \cdot s_N$, where $i$ are the inclusions, we denote it by $(H,N,M,s_N,s_M)$.

\begin{definition}
A \emph{morphism between $2$-fold epimorphisms of Hopf algebras} $(H,N,M,s_N,s_M)$ and $(H',N',M',s_{N'},s_{M'})$ is given by a morphism of Hopf algebras $f \colon H \to H'$ such that the image of $f|_N := f \cdot i_N$ ($f|_M := f \cdot i_M$) lies in $N'$ ($M'$) and the following diagrams commute
\begin{center}
\begin{tikzpicture}[descr/.style={fill=white},scale=0.9]
\node (A) at (0,0) {$H'$};
\node (B) at (0,2) {$H$};
\node (C) at (2,2) {$N$};
\node (D) at (2,0) {$N'$};
  \path[-stealth]
 (B.south) edge node[right] {$f$} (A.north) 
 (C.south) edge node[right] {$f|_N$} (D.north);
 \path[->,font=\scriptsize]
(A.east) edge node[above] {$s'_{N'}$} (D.west)
(B.east) edge node[above] {$s_N$} (C.west);
\end{tikzpicture}
\begin{tikzpicture}[descr/.style={fill=white},scale=0.9]
\node (A) at (0,0) {$H'$};
\node (B) at (0,2) {$H$};
\node (C) at (2,2) {$M$};
\node (D) at (2,0) {$M'.$};
  \path[-stealth]
 (B.south) edge node[right] {$f$} (A.north) 
 (C.south) edge node[right] {$f|_M$} (D.north);
 \path[->,font=\scriptsize]
(A.east) edge node[above] {$s'_{M'}$} (D.west)
(B.east) edge node[above] {$s_M$} (C.west);
\end{tikzpicture}
\end{center}
\end{definition}

We denote the category of $2$-split epimorphisms of cocommutative Hopf algebras by 
$\sf{Pt}^2(\Hopf_{K,coc})$ $ = \sf Pt(Pt(\Hopf_{K,coc})) $. The category $\sf{Grpd^2}(\Hopf_{K, coc})$ is a (not full) subcategory of this category.

\section{Definition of a Hopf crossed square}

As a preliminary notion, we define what we call a \emph{Hopf $2$-action}. Similarly as in $\sf Grp$ and $\sf Lie_K$, a Hopf $2$-action is an underlying structure of a Hopf crossed square. We define this notion and the further ones in the context of cocommutative Hopf algebras. Let us notice that the definitions make sense also in the non-cocommutative case. The following definition is inspired by the definitions of $2$-actions of groups and of Lie algebras given by Ellis in his thesis \cite{Ellis phd}.

\begin{definition}\label{2-action}
A \emph{Hopf $2$-action} $(L,M,N,P,h)$ is given by 4 cocommutative Hopf algebras $L$, $M$, $N$, $P$  such that $L$, $M$, $N$ are $P$-module Hopf algebras, $L$ is an $N$-module Hopf algebra and an $M$-module Hopf algebra, and $h \colon M \otimes N \to L$ is a coalgebra morphism. They have to satisfy the following conditions for any $m,m' \in M$, $n,n' \in N$, $l \in L$ and $p \in P$,
\begin{itemize}
\item[(2A1)] $ ({{p_1} \triangleright m} ) \triangleright ( {p_2}  \triangleright l) = p \triangleright (m  \triangleright l)$,\\
$({{p_1}  \triangleright n})  \triangleright ({p_2} \triangleright l) = p  \triangleright (n  \triangleright l)$,
\item[(2A2)] $h$ is $P$-linear ,
\item[(2A3)]$h(1_M \otimes n) = \epsilon(n)1_L$,\\$h(m \otimes 1_N) = \epsilon(m)1_L$,
\item[(2A4)]$h(m \otimes nn') = h(m_1 \otimes n_1)( {n_2}  \triangleright h(m_2 \otimes n'))$ ,\\ $h(mm' \otimes n) = ({m_1}  \triangleright h(m' \otimes n_1))h(m_2 \otimes n_2)$ ,
\item[(2A5)]
$({m_1} \triangleright ({n_1} \triangleright l)) h( m_2 \otimes n_2) = h( m_1 \otimes n_1) ({n_2}  \triangleright ({m_2}  \triangleright l))$.
\end{itemize}
\end{definition}

\begin{definition}\label{morph hopf-2-action}
Let $(L,M,N,P,h)$ and $(L',M',N',P',h')$ be two  Hopf $2$-actions, a \emph{morphism of Hopf $2$-actions} is given by 4 Hopf algebra morphisms $\alpha \colon L \to L'$, $\beta \colon M \to M'$, $\gamma \colon N \to N'$ and $\delta \colon P \to P'$, such that 
$\alpha$, $\beta$, $\gamma$ and $\delta$ are morphisms of modules i.e.\
\[ \alpha(p \triangleright l) = {\delta(p)}  \triangleright \alpha(l),\]
\[ \alpha(m  \triangleright l) = {\beta(m)}  \triangleright \alpha(l),\]
\[ \alpha(n \triangleright l) = {\gamma(n)}  \triangleright \alpha(l),\]
\[ \beta(p  \triangleright m) = {\delta(p)}  \triangleright \beta(m),\]
\[ \gamma(p  \triangleright n) = {\delta(p)}  \triangleright \gamma(n). \]
Moreover, they are compatible with the coalgebra morphisms $h$ and $h'$, i.e.\ the following identity is satisfied \[ \alpha \cdot h = h' \cdot (\beta \otimes\gamma), \]
where $\beta \otimes \gamma \colon M \otimes N \to M' \otimes N'$.
\end{definition}
The category of Hopf $2$-actions is formed by Hopf $2$-actions and morphisms of Hopf $2$-actions and it is denoted by $\sf Act^2(\Hopf_{K,coc})$. We continue with the definitions of the objects and the morphisms of our category of interest: the category of Hopf crossed squares.
\begin{definition}\label{defHopf_square}
In $\mathsf{Hopf_{K,coc}}$, a \emph{Hopf crossed square} $(L,M,N,P,h, \lambda, \lambda', \mu, \nu)$ is a commutative square in $\mathsf{Hopf_{K,coc}}$,
\begin{center}
\begin{tikzpicture}[descr/.style={fill=white},scale=0.9]
\node (A) at (0,0) {$N$};
\node (B) at (0,2) {$L$};
\node (C) at (2,2) {$M$};
\node (D) at (2,0) {$P,$};
  \path[-stealth]
 (B.south) edge node[left] {$\lambda'$} (A.north) 
 (C.south) edge node[right] {$\mu$} (D.north)
(A.east) edge node[above] {$\nu$} (D.west)
(B.east) edge node[above] {$\lambda$} (C.west);
\end{tikzpicture}
\end{center}
where $L$, $M$, $N$ are P-modules Hopf algebras 
and $h \colon M \otimes N \to L$ a coalgebra map such that for any $l \in L$, $n,n' \in N$, $m,m' \in M$ and $p \in P$, the following properties hold
\begin{itemize}
\item[(CS1)] $(P,M,\mu)$, $(P,N,\nu)$ and $(P,L,\kappa = \mu \cdot \lambda = \nu \cdot \lambda')$ are Hopf crossed modules, 
\item[(CS2)] $\lambda$ and $\lambda'$ are $P$-linear,
\item[(CS3)] $h$ is $P$-linear,
\item[(CS4)] $\lambda \cdot h (m \otimes n ) = m_1 ({\nu(n)}  \triangleright S(m_2))$,  \\ $\lambda' \cdot h (m \otimes n ) = ({\mu(m)}  \triangleright n_1) S(n_2)$,
\item[(CS5)] $h( \lambda(l) \otimes n) = l_1({\nu(n)}  \triangleright S(l_2))$, \\$h (m \otimes \lambda'(l)) = ({\mu(m)}  \triangleright l_1)S(l_2)$,
\item[(CS6)]$h(m \otimes nn') = h(m_1 \otimes n_1)({\nu(n_2)} \triangleright h(m_2 \otimes n'))$, \\ $h(mm' \otimes n) = ({\mu(m_1)}  \triangleright h(m' \otimes n_1))h(m_2 \otimes n_2)$ .

\end{itemize}
\end{definition}

\begin{definition}
A \emph{morphism of Hopf crossed squares} between the crossed squares $(L,M,N,P,h,$ $ \lambda, \lambda', \mu, \nu)$ and $(\hat{L},\hat{M},\hat{N},\hat{P},\hat{h}, \hat{\lambda}, \hat{\lambda'}, \hat{\mu}, \hat{\nu})$ is given by 4 Hopf algebra morphisms $\alpha \colon L \to \hat{L}$, $\beta \colon M \to \hat{M}$, $\gamma \colon N \to \hat{N}$ and $\delta \colon P \to \hat{P}$, such that each face of the following cube commutes
\begin{center}
\begin{tikzpicture}[descr/.style={fill=white},baseline=(A.base)]
\node (X) at (0,-3) {$N$};
\node (A) at (0,0) {$L$};
\node (B) at (3,0) {$M$};
\node (C) at (3,-3) {$P$};
\node (X') at (1.5,-1.5) {$\hat{N}$};
\node (A') at (1.5,1.5) {$\hat{L}$};
\node (B') at (4.5,1.5) {$\hat{M}$};
\node (C') at (4.5,-1.5) {$\hat{P}$};
\path[-stealth] 
(X.north east) edge node[above] {$\gamma$} (X'.south west)
(C.north east) edge node[above] {$\delta$} (C'.south west)
 (A.north east) edge node[above]{$\alpha$}(A'.south west)
  (B.north east) edge node[above] {$\beta$} (B'.south west)
(A'.south) edge node[below,descr] {$\hat{\lambda}'$} (X'.north)
  (X'.east) edge node[above] {$\; \; \; \; \; \; \; \; \hat{\nu}$}(C'.west)
   (A'.east)  edge node[above] {$\hat{\lambda}$} (B'.west)
(B'.south) edge node[right] {$\hat{\mu} $} (C'.north)  
  (A.south) edge node[right] {$\lambda' $}  (X.north)
  (X.east) edge node[above] {$\nu $}  (C.west)
   (A.east)  edge node[above] {$ \; \; \; \; \; \; \lambda$} (B.west)
(B.south) edge node[descr] {$\mu$} (C.north);
\end{tikzpicture}
\end{center}
and $\alpha$, $\beta$, $\gamma$ and $\delta$ are morphisms of modules i.e.\
\[ \alpha(p  \triangleright l) = {\delta(p)}  \triangleright \alpha(l),\]
\[ \beta(p  \triangleright m) = {\delta(p)}  \triangleright \beta(m),\]
\[ \gamma(p  \triangleright n) = {\delta(p)}  \triangleright \gamma(n). \]
Moreover, they are compatible with the coalgebra morphisms $h$ and $\hat{h}$, i.e.\ we have the following identity \[ \alpha \cdot h = \hat{h} \cdot (\beta \otimes\gamma), \]
where $\beta \otimes \gamma \colon M \otimes N \to \hat{M} \otimes \hat{N}$.
\end{definition}

The Hopf crossed squares and the morphisms of Hopf crossed squares form the category of Hopf crossed squares, $\sf X^2(\Hopf_{K,coc})$. In a similar way as the action of Hopf algebras is a component of the definition of Hopf crossed module, there is a Hopf $2$-action structure underlying the notion of a Hopf crossed square. This is made clear thanks to the following proposition:

\begin{proposition}\label{act sub crossed square}
The category $\sf X^2(\Hopf_{K,coc})$ is a subcategory of $\sf Act^2(\Hopf_{K,coc})$.
\end{proposition}

\begin{proof}
Let  us consider a Hopf crossed square $(L,M,N,P,h, \lambda, \lambda', \mu, \nu)$, so that by the definition of a Hopf crossed square, $L$, $M$, $N$ are 3 $P$-module Hopf algebras. Thanks to the morphisms $\mu$ and $\nu$, we can define the linear map \[M \otimes L \to L \colon m \otimes l \mapsto {\mu(m)}  \act l,\] which turns $L$ into an $M$-module Hopf algebra and the linear map \[N \otimes L \to L \colon n \otimes l \mapsto {\nu(n)}  \triangleright l,\] which turns $L$ into an $N$-module Hopf algebra.
Moreover, they satisfy the 5 axioms of the definition of Hopf $2$-action (The verifications are done in Appendix \ref{A_1}).

Furthermore, it is easy to check that morphisms of Hopf crossed modules are also morphisms of Hopf $2$-actions.
\end{proof}
We notice that the subcategory $\sf X^2(\Hopf_{K,coc})$ is not a full subcategory of $\sf Act^2(\Hopf_{K,coc})$. Indeed, even in the case of $\sf{Grp}$, the category of crossed modules of groups $\sf{Xmod}$ (recalled later in Definition \ref{xmod grp}) is not a full subcategory of the category of actions of groups $\sf{Act}$, where an action of groups is a function $\act\colon G \times X \to X$ such that $(gg' \act x) = (g \act (g' \act x))$, $1_G \act x = x $ and $(g \act xx') = (g \act x)(g \act x')$ for any $g,g' \in G$, $x,x' \in X$. For example, let take ($\mathbb{Z}$, $\mathbb{Z}$, $d\colon \mathbb{Z} \to \mathbb{Z} \colon z \mapsto 2z$) and ($\mathbb{Z}$, $\mathbb{Z}$, $Id$) two crossed modules where the actions are the trivial ones. Then $(Id \colon \mathbb{Z} \to \mathbb{Z},Id \colon \mathbb{Z} \to \mathbb{Z})$ is a morphism of actions which is not a morphism of crossed modules. Another way of seeing this is by using the equivalence of categories between the actions and the split epimorphisms and the equivalence of categories between the crossed modules and the reflexive multiplicative graphs. Indeed, it is well-known that the category of reflexive graphs is a subcategory which is not full of the category of split epimorphisms.

\begin{remark}\label{Remark}
Let us make some observations about the Hopf crossed squares.
\begin{itemize}
\item[(1)]
Let $(L,M,N,P,\lambda, \lambda', \mu, \nu)$ be a Hopf crossed square, then $(M,L,\lambda)$ (where $M$ acts on $L$ as $M \otimes L \to L \colon m \otimes l \mapsto {\mu(m)}  \triangleright l $) and $(N,L,\lambda')$ (where $N$ acts on $L$ as $N \otimes L \to L \colon n \otimes l \mapsto {\nu(n)} \triangleright l)$ are crossed modules of Hopf algebras. We use  that $\lambda$ is $P$-linear, and that $\kappa = \mu \cdot \lambda$ and $\mu$ are crossed modules to prove the conditions (CM1) and (CM2) for $(M,L,\lambda)$.
\[\lambda({\mu(m)} \triangleright l) = {\mu(m)}  \triangleright \lambda(l) = m_1 \lambda(l) S(m_2),\]
\[({\mu \cdot \lambda (l)})  \triangleright l' = {\kappa(l)}  \triangleright l' = l_1l'S(l_2).\]
Same computations work for $(N,L,\lambda')$.
\item[(2)] $(\lambda, Id_P)$ is a morphism of crossed modules between $(P,L,\kappa)$ and $(P,M,\mu)$ \begin{center}
\begin{tikzpicture}[descr/.style={fill=white},scale=0.9, baseline = (B.base)]
\node (A) at (0,0) {$P$};
\node (B) at (0,2) {$L$};
\node (C) at (2,2) {$M$};
\node (D) at (2,0) {$P$};
  \path[-stealth]
 (B.south) edge node[right] {$\kappa$} (A.north) 
 (C.south) edge node[right] {$\mu$} (D.north)
(A.east) edge node[above] {$Id$} (D.west)
(B.east) edge node[above] {$\lambda$} (C.west);
\end{tikzpicture}
\end{center}
 since $\lambda$ is $P$-linear.

$( Id_L, \mu)$ is a morphism of crossed modules between $(M,L,\lambda)$ and $(P,L,\kappa)$
\begin{center}
\begin{tikzpicture}[descr/.style={fill=white},scale=0.9, baseline = (B.base)]
\node (A) at (0,0) {$M$};
\node (B) at (0,2) {$L$};
\node (C) at (2,2) {$L$};
\node (D) at (2,0) {$P$};
  \path[-stealth]
 (B.south) edge node[right] {$\lambda$} (A.north) 
 (C.south) edge node[right] {$\kappa$} (D.north)
(A.east) edge node[above] {$\mu$} (D.west)
(B.east) edge node[above] {$Id$} (C.west);
\end{tikzpicture}
\end{center}
by the construction of L as an $M$-module.

Similarly, we can prove that $(\lambda', Id_P)$ and $(Id_L , \nu)$ are morphisms of crossed modules. 
\item[(3)]In $\mathsf{Hopf_{K,coc}}$, if $(L,M,N,P,h, \lambda, \lambda', \mu, \nu)$ is a Hopf crossed square,
\begin{center}
\begin{tikzpicture}[descr/.style={fill=white},scale=0.9]
\node (A) at (0,0) {$N$};
\node (B) at (0,2) {$L$};
\node (C) at (2,2) {$M$};
\node (D) at (2,0) {$P,$};
  \path[-stealth]
 (B.south) edge node[right] {$\lambda'$} (A.north) 
 (C.south) edge node[right] {$\mu$} (D.north)
(A.east) edge node[above] {$\nu$} (D.west)
(B.east) edge node[above] {$\lambda$} (C.west);
\end{tikzpicture}
\end{center}
then $(L,N,M,P,S \cdot h \cdot \sigma, \lambda', \lambda, \nu, \mu)$, where $\sigma(x\ox y) = y \ox x$, is also a Hopf crossed square 
\begin{center}
\begin{tikzpicture}[descr/.style={fill=white},scale=0.9]
\node (A) at (0,0) {$M$};
\node (B) at (0,2) {$L$};
\node (C) at (2,2) {$N$};
\node (D) at (2,0) {$P.$};
  \path[-stealth]
 (B.south) edge node[right] {$\lambda$} (A.north) 
 (C.south) edge node[right] {$\nu$} (D.north)
(A.east) edge node[above] {$\mu$} (D.west)
(B.east) edge node[above] {$\lambda'$} (C.west);
\end{tikzpicture}
\end{center}
Indeed, we can easily check that all the conditions are satisfied, for example, we obtain (CS4) for any $n \in N$ and $m \in M$,
as follows
\begin{align*}
\lambda' \cdot S \cdot h \cdot \sigma (n \otimes m) &= S(\lambda'  \cdot h(m \otimes n))\\
&\stackrel{\mathclap{(CS4)}}{=} \; S (({\mu(m)}  \triangleright n_1) S(n_2)) \\
&\stackrel{\mathclap{\eqref{coco}}}{=}  \; n_1 ({\mu(m)}  \triangleright S(n_2))
\end{align*} 
  
\end{itemize}
\end{remark}
We illustrate the notion of Hopf crossed square with some examples: 
\begin{examples}\label{ex D}
\begin{itemize}
\item[(1)]Let $N$ and $M$ be two normal Hopf subalgebras of $P$ a cocommutative Hopf algebra,  we construct the following square of inclusions
\begin{center}
\begin{tikzpicture}[descr/.style={fill=white},scale=0.9]
\node (A) at (0,0) {$N$};
\node (B) at (0,2) {$N \cap M$};
\node (C) at (2,2) {$M$};
\node (D) at (2,0) {$P$};
  \path[-stealth]
 (B.south) edge node[right] {$ $} (A.north) 
 (C.south) edge node[right] {$  $} (D.north)
(A.east) edge node[above] {$ $} (D.west)
(B.east) edge node[above] {$  $} (C.west);
\end{tikzpicture}
\end{center}
with $h \colon M \otimes N \to N \cap M \colon m \ox n \mapsto m_1n_1S(m_2)S(n_2)$, where the modules are given by conjugation. This construction is a Hopf crossed square. In particular the unit square 
\begin{center}
\begin{tikzpicture}[descr/.style={fill=white},scale=0.9]
\node (A) at (0,0) {$A$};
\node (B) at (0,2) {$A$};
\node (C) at (2,2) {$A$};
\node (D) at (2,0) {$A$};
\draw[commutative diagrams/.cd, ,font=\scriptsize]
(B.south) edge[commutative diagrams/equal] (A.north)
(C.south) edge[commutative diagrams/equal] (D.north)
(A.east) edge[commutative diagrams/equal] (D.west)
(B.east)  edge[commutative diagrams/equal]  (C.west)
;
\end{tikzpicture}
\end{center}
is then a Hopf crossed square with $A \ox A \to A \colon a \ox a' \mapsto a_1a'S(a_2)$ and $h(a \ox a') = a_1a'_1S(a_2)S(a'_2)$.
Note that $A = A$ is the normalization of the largest equivalence relation on $A$. 
\item[(2)] Any commutative diagram of commutative (and cocommutative) Hopf algebras,
\begin{center}
\begin{tikzpicture}[descr/.style={fill=white},scale=0.9]
\node (A) at (0,0) {$D$};
\node (B) at (0,2) {$A$};
\node (C) at (2,2) {$B$};
\node (D) at (2,0) {$C$};
  \path[-stealth]
 (B.south) edge node[right] {$ g$} (A.north) 
 (C.south) edge node[right] {$ k$} (D.north)
(A.east) edge node[above] {$ l$} (D.west)
(B.east) edge node[above] {$ f$} (C.west);
\end{tikzpicture}
\end{center} can be seen as a Hopf crossed square where the modules are the trivial ones (i.e.\ $C \otimes B \to B \colon c \otimes b \mapsto b\epsilon(c)$), and the coalgebra map $h \colon B  \otimes D \to A$ is chosen as $h(b \otimes d) = \epsilon(b)\epsilon(d)1_A$.
Note that this example follows from the fact the category of commutative and cocommutative Hopf algebras is abelian (see corollary \ref{ab}).

 \item[(3)] Thanks to Theorem \ref{CGKMM}, we know that we can re-express any cocommutative Hopf algebra over an algebraically closed field of characteristic zero as 
\[ X \cong U(L_X) \rtimes K[G_X].\]
Let $B$ and $X$ be two commutative and cocommutative Hopf algebras over an algebraically closed field of characteristic zero, $d \colon X \to B$ be a crossed module, then the following square  
\begin{center}
\begin{tikzpicture}[descr/.style={fill=white},scale=0.9]
\node (A) at (0,0) {$K[G_X] $};
\node (B) at (0,2) {$U(L_X)$};
\node (C) at (4,2) {$U(L_B)$};
\node (D) at (4,0) {$K[G_B] $};
  \path[-stealth]
 (B.south) edge node[right] {$ \eta_{K[G_X]} \cdot \epsilon$} (A.north) 
 (C.south) edge node[right] {$ \eta_{K[G_B]} \cdot \epsilon$} (D.north)
(A.east) edge node[above] {$ d$} (D.west)
(B.east) edge node[above] {$ d$} (C.west);
\end{tikzpicture}
\end{center}
is a Hopf crossed square where $h \colon  U(L_B)  \otimes K[G_X]  \to U(L_X) $ is the trivial map ($h \colon  U(L_B)  \otimes K[G_X]  \to U(L_X) \colon k \ox g \mapsto 0$). The $K[G_B]$-actions on $U(L_X)$ and $K(G_X)$ are induced by the $B$-module structure of $X$, and the  $K[G_B]$-action on $U(L_B)$ is given by the linear map $ K[G_B] \ox U(L_B) \to U(L_B)$ defined by $g \act x =x$ for $g \in K[G_B]$, $x \in U(L_B)$.
Note that the vertical arrows correspond to the normalization of a totally disconnected groupoids, i.e a groupoid with $\delta = \gamma$ (see \cite{BB} for details). 
\item[(4)] Another example of Hopf crossed square is given by the following square

 \begin{center}
\begin{tikzpicture}[descr/.style={fill=white},scale=0.9]
\node (A) at (0,0) {$X$};
\node (B) at (0,2) {$X$};
\node (C) at (2,2) {$X$};
\node (D) at (2,0) {$B,$};
\draw[commutative diagrams/.cd, ,font=\scriptsize]
(B.south) edge[commutative diagrams/equal] (A.north)
(B.east)  edge[commutative diagrams/equal]  (C.west)
;
  \path[-stealth]
(C.south) edge node[right] {$ d$} (D.north)
(A.east) edge node[above] {$ d$} (D.west);
\end{tikzpicture}
\end{center}
where $d  \colon X \to B$ is a Hopf crossed module and $h \colon X \ox X \to X$ is given by  $h( x \ox y) = x_1y_1S(x_2)S(y_2)$.
\item[(5)] Let $d \colon X \to B$ be a crossed module, then the commutative square
\begin{center}
\begin{tikzpicture}[descr/.style={fill=white},scale=0.9]
\node (A) at (0,0) {$X$};
\node (B) at (0,2) {$K$};
\node (C) at (2,2) {$K$};
\node (D) at (2,0) {$B$};
\draw[commutative diagrams/.cd, ,font=\scriptsize]
(B.east)  edge[commutative diagrams/equal]  (C.west)
;
  \path[-stealth]
 (B.south) edge node[right] {$ \eta_X$} (A.north) 
 (C.south) edge node[right] {$\eta_B $} (D.north)
(A.east) edge node[above] {$ d $} (D.west);
\end{tikzpicture}
\end{center}
is a Hopf crossed square with $h := Id_K \ox \epsilon\colon K \ox  X \to K$ (this definition of $h$ is the only one possible since $K$ is the terminal object). This construction induces the discrete functor $\sf D' \colon HXMod_{K,coc} \to X^2(Hopf_{K,coc})$, which is for crossed squares what the discrete functor $\mathsf{ D \colon Hopf_{K,coc} \to HXMod_{K,coc}} \colon H \mapsto (H,K,\eta_H)$ is for crossed modules.

\end{itemize}

\end{examples}

An important property that we want for our definition of Hopf crossed squares is to recover the definitions of crossed squares of groups and of Lie algebras by restricting a Hopf crossed square to the group-like elements and primitive elements, respectively. In order to verify this, we recall the definition of crossed modules and crossed squares of groups and of Lie algebras. 
\begin{definition}\label{xmod grp}\cite{Whitehead}
A \emph{crossed module of groups} is given by a morphism of groups $d \colon X \to B$ endowed with an action of groups of $B$ on $X$ such that for any $b$ in $B$ and any $x$, $y$ in $X$,
\begin{equation}
d(b \act x) = bd(x)b^{-1},
\end{equation}
\begin{equation}
d(x) \act y = xyx^{-1}.
\end{equation}
\end{definition}
The definition of crossed squares of groups given below is equivalent to the one introduced by Loday in \cite{Loday}. Note that the observations (1) and (2) in Remark \ref{Remark}  are the analogues for cocommutative Hopf algebras, of the (redundant) conditions of the original definition of crossed squares of groups given by Loday.
\begin{definition}\cite{Loday}\label{x2mod group}
A \emph{crossed square of groups} is a commutative square of groups
\begin{center}
\begin{tikzpicture}[descr/.style={fill=white},scale=0.9]
\node (A) at (0,0) {$M'$};
\node (B) at (0,2) {$L$};
\node (C) at (2,2) {$M$};
\node (D) at (2,0) {$P,$};
  \path[-stealth]
 (B.south) edge node[right] {$ \lambda'$} (A.north) 
 (C.south) edge node[right] {$ \mu$} (D.north)
(A.east) edge node[above] {$ \mu'$} (D.west)
(B.east) edge node[above] {$ \lambda$} (C.west);
\end{tikzpicture}
\end{center}
together with an action of $P$ on $L$, $M$ and $M'$, and with a function $h \colon M\times M' \to L$ satisfying the
following axioms
\begin{itemize}
\item[(i)]  the homomorphisms $ \mu, \mu'$ and $\kappa = \mu \cdot \lambda = \mu' \cdot \lambda'$ are crossed modules.
\item[(ii)] $\lambda \cdot h (m ,n ) = m ({\mu'(n)}  \triangleright m^{-1})$,  \\ $\lambda' \cdot h (m , n ) = ({\mu(m)}  \triangleright n)n^{-1}$,
\item[(iii)]$h( \lambda(l) , n) = l({\mu'(n)}  \triangleright l^{-1})$, \\$h (m , \lambda'(l)) = ({\mu(m)}  \triangleright l)l^{-1}$,
\item[(iv)]$h(m , nn') = h(m , n)({\mu'(n)} \triangleright h(m , n'))$, \\ $h(mm' , n) = ({\mu(m)}  \triangleright h(m' , n))h(m , n)$,
\item[(v)] $^ph(m , n) = h(\;^pm , \;^pn)$,
\end{itemize}
for all $m,m'$ in $M$, $n,n'$ in $M'$, $l$ in $L$ and $p$ in $P$.
\end{definition}

\begin{definition}\cite{LR}\label{xmod Lie}
A \emph{crossed module of Lie algebras} is given by a morphism of Lie algebras $d \colon \mathfrak{a} \to \mathfrak{b}$ endowed with an action of Lie algebras of $\mathfrak{b}$ on $ \mathfrak{a}$ such that for any $b$ in $ \mathfrak{b}$ and any $x$, $y$ in $ \mathfrak{a}$,
\begin{equation}
d(b \act x) = [b,d(x)],
\end{equation}
\begin{equation}
d(x) \act y = [x,y].
\end{equation}
\end{definition}

The notion of crossed square of Lie algebras was introduced by Ellis in his thesis \cite{Ellis phd}. The definition below is an equivalent definition given later in \cite{CL}. 
\begin{definition}\label{x2mod lie}
A \emph{crossed square of Lie $K$-algebras} is a commutative square of Lie $K$-algebras
\begin{center}
\begin{tikzpicture}[descr/.style={fill=white},scale=0.9]
\node (A) at (0,0) {$\mathfrak{n}$};
\node (B) at (0,2) {$\mathfrak{l}$};
\node (C) at (2,2) {$\mathfrak{m}$};
\node (D) at (2,0) {$\mathfrak{p},$};
  \path[-stealth]
 (B.south) edge node[right] {$ \lambda'$} (A.north) 
 (C.south) edge node[right] {$ \mu$} (D.north)
(A.east) edge node[above] {$ \nu$} (D.west)
(B.east) edge node[above] {$ \lambda$} (C.west);
\end{tikzpicture}
\end{center}
together with an action of $\mathfrak{p}$  on $\mathfrak{l}$, $\mathfrak{m}$ and $\mathfrak{n}$ and with a function $h \colon \mathfrak{m} \times \mathfrak{n} \to \mathfrak{l}$ satisfying the
following axioms\begin{itemize}
\item[(i)] $\lambda$ and $\lambda'$ preserve the action of $\mathfrak{p}$,
\item[(ii)]  the homomorphisms $ \mu, \mu'$ and $\kappa = \mu \cdot \lambda = \mu' \cdot \lambda'$ are crossed modules,
\item[(iii)] $\alpha h(m,n) = h(\alpha m,n) = h(m,\alpha n)$,
\item[(iv)] $h(m+m',n) = h(m,n)+h(m',n)$,\\
$h(m,n+n')=h(m,n)+h(m,n')$,
\item[(v)] $h([m,m'] , n) = ({\mu(m)}  \triangleright h(m' , n)) - ({\mu(m')}  \triangleright h(m , n))$,\\$h(m , [n,n']) = ({\nu(n)} \triangleright h(m , n'))-({\nu(n')} \triangleright h(m , n))$, 
\item[(vi)] $^ph(m , n) = h(\;^pm , n) + h(m,\;^pn)$,
\item[(vii)] $\lambda \cdot h (m ,n ) =-({\nu(n)}  \triangleright m)$,  \\ $\lambda' \cdot h (m , n ) = ({\mu(m)}  \triangleright n)$,
\item[(viii)]$h( \lambda(l) , n) = -({\nu(n)}  \triangleright l)$, \\$h (m , \lambda'(l)) = ({\mu(m)}  \triangleright l)$,
\end{itemize}
for all $m,m'$ in $\mathfrak{m}$, $n,n'$ in $\mathfrak{n}$, $l$ in $\mathfrak{l}$, $p$ in $\mathfrak{p}$ and $\alpha \in K$.
\end{definition}

The morphisms of these four algebraic structures are defined as expected, and they respectively form the category of crossed modules of groups, $\sf XMod$, the category of crossed square of groups, $\sf X^2Mod$, the category of crossed modules of Lie algebras, $\sf XLieAlg_K$, and the category of crossed squares of Lie algebras, $\sf{X^2LieAlg}_K$. In the following result, we exhibit a lifting of the adjunctions \eqref{adj} to two different levels.

\begin{proposition}\label{double adjunction}
The adjunctions are lifted to the respective categories of crossed modules and crossed squares 
\begin{center}
\begin{tikzpicture}[descr/.style={fill=white},scale=1.2]
\node at (-1.5,0) {$\perp$};
\node at (1.6,0) {$\perp$};
\node (A) at (0,0) {$\sf{Hopf}_{K,coc}$};
\node (B) at (3,0) {$\sf{LieAlg}_K,$};
\node (C) at (-2.5,0) {$\sf{Grp}$};
\node at (-1.5,1.5) {$\perp$};
\node at (1.6,1.5) {$\perp$};
\node (A') at (0,1.5) {$\sf{HXMod}_{K,coc}$};
\node (B') at (3,1.5) {$\sf{XLieAlg}_K$};
\node (C') at (-2.5,1.5) {$\sf{XMod}$};
\node at (-1.5,3) {$\perp$};
\node at (1.6,3) {$\perp$};
\node (A'') at (0,3) {$\sf{X^2(Hopf_{K,coc})}$};
\node (B'') at (3,3) {$\sf{X^2LieAlg}_K$};
\node (C'') at (-2.5,3) {$\mathsf{X^2Mod}$};
 \path[->,font=\scriptsize]
 (A'.north)  edge node[descr] {$\sf{D'}$} (A''.south)
 (B'.north) edge node[descr] {$\sf{D'}$} (B''.south) 
 (C'.north) edge node[descr] {$\sf{D'}$} (C''.south) 
 ([yshift=5pt]C''.east) edge node[above] {$\mathsf{K}[-]$} ([yshift=5pt]A''.west)
 ([yshift=-5pt]A''.west) edge node[below] {$\mathcal{G}$}([yshift=-5pt]C''.east) 
 ([yshift=5pt]B''.west) edge node[above] {$\mathsf{U}$}([yshift=5pt]A''.east) 
([yshift=-5pt]A''.east) edge node[below] {$\mathcal{P}$} ([yshift=-5pt]B''.west)
  (A.north) edge node[descr] {$\sf{D}$} (A'.south) 
(B.north) edge node[descr] {$\sf{D}$}  (B'.south) 
(C.north) edge node[descr] {$\sf{D}$} (C'.south) 
 ([yshift=5pt]C'.east) edge node[above] {$\mathsf{K}[-]$} ([yshift=5pt]A'.west)
 ([yshift=-5pt]A'.west) edge node[below] {$\mathcal{G}$}([yshift=-5pt]C'.east) 
 ([yshift=5pt]B'.west) edge node[above] {$\mathsf{U}$}([yshift=5pt]A'.east) 
([yshift=-5pt]A'.east) edge node[below] {$\mathcal{P}$} ([yshift=-5pt]B'.west)
([yshift=5pt]C.east) edge node[above] {$\mathsf{K}[-]$} ([yshift=5pt]A.west)
 ([yshift=-5pt]A.west) edge node[below] {$\mathcal{G}$}([yshift=-5pt]C.east) 
 ([yshift=5pt]B.west) edge node[above] {$\mathsf{U}$}([yshift=5pt]A.east) 
([yshift=-5pt]A.east) edge node[below] {$\mathcal{P}$} ([yshift=-5pt]B.west);
\end{tikzpicture}
\end{center}
where $K$ is the ground field and $\sf D$ and $\sf D'$ are the discrete functors. 
The explicit descriptions of $\sf D$ and $\sf D'$ are given in Examples \ref{ex D} (5).

\end{proposition}

\begin{proof}

First, it is clear that $\sf{K[-]}$ sends a crossed module of groups to a Hopf crossed module and $\mathcal{G}$ does the opposite. We give some details of how the functor $\mathcal{P}$ sends a Hopf crossed module to a crossed module of Lie algebras. We recall that $\mathcal{P}(H) = \{x \in  H | \Delta(x) = x \ox 1_H + 1_H \ox x\}$ is a Lie algebra where  the bracket is given by $[x,y] = xy-yx$. Let $(B,X,d)$ be a Hopf crossed module, then by restricting this definition to the primitive elements, we obtain a morphism of Lie algebras, such that 
\[ d(b \triangleright x) = b_1d(x)S(b_2) = bd(x) - d(x)b = [b , d(x)]\]
\[ {d(x)}  \triangleright x' = x_1x'S(x_2) = xx' - x'x = [x , x']. \]
This is the definition of a crossed module of Lie algebras (Definition \ref{xmod Lie}). The adjunction structure is also lifted to the ``crossed structure".
 Indeed, we know that for $G$ a group and $H$ a Hopf algebra we have the isomorphism
\begin{equation}\label{iso adjunction}
\mathsf{Hom}_{\mathsf{Hopf_{K,coc}}}(\mathsf{K}[G], H) \cong \mathsf{Hom}_{\mathsf{Grp}}(G, \mathcal{G}(H))
\end{equation}
Let $(B,X,d)$ be a crossed module of groups and $(H,Y,\mu)$ a crossed module of Hopf algebras. Then by using \eqref{iso adjunction} we can prove that there is an isomorphism between $\mathsf{Hom}_{\mathsf{HXMod_{K,coc}}}(\mathsf{K}[(B,X,d)], (H,Y,\mu))$ and $\mathsf{Hom}_{\mathsf{XMod}}((B,X,d), \mathcal{G}(H,Y,\mu))$
\begin{center}
\begin{tikzpicture}[descr/.style={fill=white},scale=1]
\node (A) at (0,0) {$Y$};
\node (B) at (0,2) {$\mathsf{K}[X]$};
\node (C) at (2,2) {$\mathsf{K}[B]$};
\node (D) at (2,0) {$H$};
\node at (4,1) {$\mapsto$};
  \path[-stealth]
 (B.south) edge node[right] {$f_1$} (A.north) 
 (C.south) edge node[right] {$ f_2$} (D.north)
(A.east) edge node[above] {$ \mu$} (D.west)
(B.east) edge node[above] {$ \mathsf{K}[d]$} (C.west);
\end{tikzpicture}
\hspace{1cm}
\begin{tikzpicture}[descr/.style={fill=white},scale=1]
\node (A) at (0,0) {$\mathcal{G}(Y)$};
\node (B) at (0,2) {$X$};
\node (C) at (2,2) {$B$};
\node (D) at (2,0) {$\mathcal{G}(H)$};
  \path[-stealth]
 (B.south) edge node[right] {$\tilde{f_1}$} (A.north) 
 (C.south) edge node[right] {$ \tilde{f_2}$} (D.north)
(A.east) edge node[above] {$ \mathcal{G}(\mu)$} (D.west)
(B.east) edge node[above] {$ d$} (C.west);
\end{tikzpicture}
\end{center}
where $\tilde{f_i}$ for $i=1,2$ is the restriction of $f_i$ to the elements of $B$ and $X$ respectively. 

The adjunctions for the crossed square structures work with similar computations (Some details are explained in Appendix \ref{E}).
\end{proof}

The definition of a Hopf crossed square allows us to recover the definitions of crossed square of groups and of Lie algebras, exactly as the definition of crossed modules of cocommutative Hopf algebras generalizes the definitions of crossed modules of groups and of Lie algebras.
In the next section we verify that this definition is equivalent to the internal crossed modules in the category of Hopf crossed modules, by showing that the category $\sf{X}^2(\mathsf{Hopf}_{K,coc})$ of Hopf crossed squares and the category $\sf{cat}^2(Hopf_{K,coc})$ of cat$^2$-Hopf algebras are equivalent.

\section{Equivalence between crossed squares of Hopf algebras and cat$^2$-Hopf algebras}

In order to prove a higher dimensional version of the equivalence between ``internal structures" and ``crossed structures" in $\sf Hopf_{K,coc}$, we show as a first step that the category of Hopf $2$-actions is equivalent to the category of $2$-fold split epimorphisms of Hopf algebras. Then, by using this equivalence, we go one step further and prove the equivalence between the category $\sf{X}^2(\mathsf{Hopf}_{K,coc})$ of Hopf crossed squares and the category  $\sf{cat}^2(Hopf_{K,coc})$ of cat$^2$-Hopf algebras.

\begin{proposition}\label{equi action/split epi}
There is an equivalence of categories between the category $\sf Act^2(\mathsf{Hopf}_{K,coc})$ of Hopf $2$-actions and the category $\sf{Pt}^2(\Hopf_{K,coc})$ of $2$-fold split epimorphisms of Hopf algebras.
\end{proposition}

\begin{proof}
On the one hand, let us consider a $2$-fold split epimorphism $(H,N,M,s_N,s_M)$, 
we construct the five following linear maps given by the action by conjugation 
\begin{align*}
 (N \cap M) \otimes (N \cap HKer(s_M)) & \to N \cap HKer(s_M) \colon x \ox k \mapsto x_1kS(x_2),\\
(N \cap M) \otimes  (HKer(s_N) \cap M) & \to HKer(s_N) \cap M \colon x \ox k \mapsto x_1kS(x_2),\\
(N \cap M) \otimes (HKer(s_N)  \cap HKer(s_M)) & \to HKer(s_N) \cap HKer(s_M) \colon x \ox k \mapsto x_1kS(x_2),\\
( N \cap HKer(s_M)) \otimes (HKer(s_N)  \cap HKer(s_M)) & \to HKer(s_N) \cap HKer(s_M) \colon x \ox k \mapsto x_1kS(x_2),\\
 (HKer(s_N) \cap M) \otimes (HKer(s_N)  \cap HKer(s_M)) & \to HKer(s_N) \cap HKer(s_M) \colon x \ox k \mapsto x_1kS(x_2).
\end{align*} 
  
Thanks to these applications we can check that $N \cap HKer(s_M)$, $HKer(s_N) \cap M $ and $HKer(s_N)  \cap HKer(s_M)$ are $(N \cap M)$-module Hopf algebras. Moreover, $HKer(s_N)  \cap HKer(s_M)$ is also an $(N \cap HKer(s_M))$-module Hopf algebra and a $(HKer(s_N) \cap M)$-module Hopf algebra.

We define $h \colon (N \cap HKer(s_M)) \otimes (HKer(s_N) \cap M) \to HKer(s_N) \cap HKer(s_M)$ as \[h(n \otimes m) = n_1m_1S(n_2)S(m_2).\]
In the appendix \ref{B} we verify that this construction is a Hopf $2$-action.

On the other hand, let $(L,M,N,P,h)$ be a Hopf $2$-action.
Since $M$ is a $P$-module Hopf algebra and $L$ is an $N$-module Hopf algebra, we use the semi-direct product to construct the two following Hopf algebras: $M \rtimes P$ and $L \rtimes N$. Moreover, in Appendix \ref{C_1}, we verify in details that $L \rtimes N$ can be seen as an $(M \rtimes P)$-module Hopf algebra thanks to the following linear map
\[ {(m \otimes p)}  \triangleright (l \otimes n) = ({m_1}  \triangleright ({p_1}  \triangleright l))h(m_2 \otimes {p_2}  \triangleright n_1) \otimes {p_3}  \triangleright n_2.\]

Then, we can define the semi-direct product $(L\rtimes N)\rtimes (M \rtimes P)$, where the Hopf algebra structure is given as follows
\begin{align*}
&(l \otimes n \otimes m \otimes p) (l' \otimes n' \otimes m' \otimes p')= l \Big( {n_1} \triangleright \big({m_1} \triangleright ({p_1} \triangleright l')\big) \Big)({n_2} \triangleright h(m_2 \otimes {p_2} \triangleright n'_1)) \\ &  \hspace{2.5cm} \otimes n_3({p_3} \triangleright n'_2) \otimes m_3({p_4}\triangleright m') \otimes p_5p',\\
&1_{(L\rtimes N)\rtimes (M \rtimes P)} = 1_L \otimes 1_N \otimes 1_M \otimes 1_P,\\
&\Delta_{(L\rtimes N)\rtimes (M \rtimes P)} (l \otimes n \otimes m \otimes p) = l_1 \otimes n_1 \otimes m_1 \otimes p_1 \otimes l_2 \otimes n_2 \otimes m_2 \otimes p_2,\\
&\epsilon_{(L\rtimes N)\rtimes (M \rtimes P)}(l \otimes n \otimes m \otimes p) = \epsilon(l)\epsilon(n)\epsilon(m)\epsilon(p),\\
&S_{(L\rtimes N)\rtimes (M \rtimes P)}(l \otimes n \otimes m \otimes p) = \big({S(p_1)} \triangleright \big({S(m_1)} \triangleright ({S(n_1)} \triangleright S(l))\big)\big)\\
& \hspace{2cm} h({S(p_2)} \triangleright S(m_2) \otimes {S(p_3)} \triangleright S(n_2)) \ox {S(p_4)} \triangleright S(n_3) \otimes {S(p_5)} \triangleright S(m_3) \otimes S(p_6).
\end{align*} 

By considering the following projections, we obtain 2 split epimorphisms of Hopf algebras defined by
\begin{align*}
&s_1 \colon  (L\rtimes N)\rtimes (M \rtimes P) \to (M \rtimes P) \colon s_1(l \otimes n \otimes m \otimes p) = \epsilon(l)\epsilon(n) (m \otimes p), \\
&s_2 \colon  (L\rtimes N)\rtimes (M \rtimes P) \to(N \rtimes P) \colon s_2(l \otimes n \otimes m \otimes p) = \epsilon(l)\epsilon(m) (n \otimes p),
\end{align*} 
 
where the splitting is given by the morphisms $e_1(m \ox p)= 1_L \ox 1_N \ox m \ox p$ and $e_2(n \ox p)= 1_L \ox n \ox 1_M \ox p$. Moreover, it is easy to check that the compatibility condition is satisfied. Hence, the $5$-tuple $((L \rtimes N)\rtimes (M \rtimes P), M \rtimes P, N \rtimes P, s_1,s_2)$ is a $2$-fold split epimorphism.

This construction can be extended to the morphisms (Appendix \ref{C_3}). We obtain two functors $\tilde{F} \colon  \sf{Pt^2(\mathsf{Hopf}_{K,coc})} \to  \sf{Act^2(\mathsf{Hopf}_{K,coc})}$  and $\tilde{G} \colon \sf{Act^2(\mathsf{Hopf}_{K,coc})} \to \sf{Pt}^2(\Hopf_{K,coc})$, giving us an equivalence of categories which is explained in details in Appendix \ref{C_4}.
\end{proof}


Now we go one step further, and we prove that we have an equivalence between Hopf crossed squares and cat$^2$-Hopf algebras. Since $ \sf{X^2(\mathsf{Hopf}_{K,coc})}$ and $ \sf{cat^2(\mathsf{Hopf}_{K,coc})}$ are subcategories of the categories $ \sf{Act^2(\mathsf{Hopf}_{K,coc})}$ and  $ \sf{Pt^2(\mathsf{Hopf}_{K,coc})} $, we build the new functors by using the functors $\tilde{F}$ and $\tilde{G}$ defined in the proof of Proposition \ref{equi action/split epi}.

\begin{theorem}\label{equi crossed squate 2-cat}
There is an equivalence of categories between $\sf{X^2(Hopf_{K,coc})}$ and $\sf{cat^2(Hopf_{K,coc})}$.
\end{theorem}

\begin{proof}
On the one hand, let us consider a cat$^2$-Hopf algebra $(H, N, M, s_N , t_N , s_M , t_M )$. We have an underlying  $2$-fold spilt epimorphism $(H, N, M, s_N , s_M )$, from which we can construct a Hopf $2$-action as in the proof of Proposition \ref{equi action/split epi}. With this Hopf $2$-action, the following square satisfies the properties of a Hopf crossed square
\begin{center}
\begin{tikzpicture}[descr/.style={fill=white},baseline=(A.base)]
\node (X) at (0,-3) {$HKer(s_N) \cap M$};
\node (A) at (0,0) {$HKer(s_N) \cap HKer(s_M)$};
\node (B) at (5,0) {$N \cap HKer(s_M)$};
\node (C) at (5,-3) {$N \cap M.$};
\path[-stealth] 
  (A.south) edge node[right] {$t_M$}  (X.north)
  (X.east) edge node[above] {$t_N$}  (C.west)
   (A.east)  edge node[above] {$t_N$} (B.west)
(B.south) edge node[right] {$t_M$} (C.north);
\end{tikzpicture}
\end{center}
In Appendix \ref{D_2}, we check that it is indeed a Hopf crossed square.

On the other hand, if we have a Hopf crossed square $(L,M,N,P,h, \lambda, \lambda', \mu, \nu)$, in particular we have a corresponding Hopf $2$-action (see Proposition \ref{act sub crossed square}).

Thanks to (3) of Remark \ref{Remark} and Proposition \ref{act sub crossed square}, we know that $(L,N,M,P,S \cdot h \cdot \sigma)$ satisfies all the conditions of the definition of Hopf $2$-action where the roles of $M$ and $N$ are exchanged. Hence, similarly to the construction in Proposition \ref{equi action/split epi}, we can construct the semi-direct product $(L\rtimes M)\rtimes (N \rtimes P)$, with $(L \rtimes M)$ seen as an $(N \rtimes P)$-module, thanks to the following linear map 
\[ {(n \otimes p)} \triangleright (l \otimes m) = ({n_1}\triangleright ({p_1} \triangleright l))S(h({p_2} \triangleright m_1 \otimes n_2)) \otimes {p_3} \triangleright m_2.\]
The multiplication in  $(L\rtimes M)\rtimes (N \rtimes P)$, is given by the following linear map \begin{align*}
(l \otimes m \otimes n \otimes p) (l' \otimes m' \otimes n' \otimes p')= l \Big({m_1}\triangleright &\big({n_1}\triangleright ({p_1} \triangleright l')\big) \Big)({m_2} \triangleright S(h({p_2} \triangleright m'_1 \otimes n_2))) \\ & \otimes m_3 ({p_3} \triangleright m'_2)\otimes  n_3 ({p_4} \triangleright n') \otimes p_5p' .
\end{align*}
Note that the semi-direct products $(L\rtimes N)\rtimes (M \rtimes P)$ and $(L\rtimes M)\rtimes (N \rtimes P)$ are isomorphic: the isomorphisms are given by

\begin{equation}\label{iso}
\psi \colon (L\rtimes N)\rtimes (M \rtimes P) \to (L\rtimes M)\rtimes (N \rtimes P) \colon (l \otimes n \otimes m \otimes p) \mapsto lS(h(m_1 \otimes n_1)) \otimes m_2 \otimes n_2 \otimes p
\end{equation}
\begin{equation*}
\psi^{-1} \colon (L\rtimes M)\rtimes (N \rtimes P) \to (L\rtimes N)\rtimes (M \rtimes P) \colon (l \otimes m \otimes n \otimes p) \mapsto lh(m_1 \otimes n_1) \otimes n_2 \otimes m_2 \otimes p.
\end{equation*}

These maps are isomorphisms of Hopf algebras (the verifications are done in Appendix \ref{C_2}.

Remark that $d \colon L \rtimes N \to M \rtimes P$ and $d' \colon L \rtimes M \to N \rtimes P$ defined by
\[ d (l \otimes n) = \lambda(l) \otimes \nu(n),\]
\[ d' (l \otimes m) = \lambda'(l) \otimes \mu(m),\]
are crossed modules (the verifications are done in Appendix \ref{D_3}). 
  Thus, we apply the equivalence in one dimension of Proposition \ref{equi xmod 1}, to obtain two cat$^1$-Hopf algebras where the source and the target maps are defined as  
\begin{align*}
&s_1 \colon  (L\rtimes N)\rtimes (M \rtimes P) \to (M \rtimes P) \colon s_1(l \otimes n \otimes m \otimes p) = \epsilon(l)\epsilon(n) (m \otimes p) ,\\
&t_1 \colon  (L\rtimes N)\rtimes (M \rtimes P) \to (M \rtimes P) \colon t_1(l \otimes n \otimes m \otimes p) = \lambda(l)({\nu(n_1)} \triangleright m) \otimes \nu(n_2)p,\\
&s'_1 \colon  (L\rtimes M)\rtimes (N \rtimes P) \to (N \rtimes P) \colon s'_1(l \otimes m \otimes n \otimes p) = \epsilon(l)\epsilon(m) (n \otimes p),\\
&t'_1 \colon  (L\rtimes M)\rtimes (N \rtimes P) \to (N \rtimes P) \colon t'_1(l \otimes m \otimes n \otimes p) =  \lambda'(l)({\mu(m_1)} \triangleright n) \otimes \mu(m_2)p,
\end{align*} 
and they have trivial commutators.

 By using the isomorphism $\psi$ as defined in \eqref{iso} we obtain the following two cat$^1$-Hopf algebras
\begin{align*}
&s_1 \colon  (L\rtimes N)\rtimes (M \rtimes P) \to (M \rtimes P) \colon s_1(l \otimes n \otimes m \otimes p) = \epsilon(l)\epsilon(n) (m \otimes p) ,\\
&t_1 \colon  (L\rtimes N)\rtimes (M \rtimes P) \to (M \rtimes P) \colon t_1(l \otimes n \otimes m \otimes p) = \lambda(l)({\nu(n_1)} \triangleright m) \otimes \nu(n_2)p, \\
&s_2 \colon  (L\rtimes N)\rtimes (M \rtimes P) \to (N \rtimes P) \colon s_2(l \otimes n \otimes m \otimes p) = \epsilon(l)\epsilon(m) (n \otimes p), \\
&t_2 \colon  (L\rtimes N)\rtimes (M \rtimes P) \to (N \rtimes P) \colon t_2(l \otimes n \otimes m \otimes p) =  \lambda'(l)n \otimes \mu(m)p,
\end{align*} 
where $s_2 = s'_1 \cdot \psi$ and $t_2 =  t'_1 \cdot  \psi$. 
  
We note that the kernels of $s_2$ and $t_2$ commute since $[HKer(s'_1), HKer(t'_1)]=0$. Indeed, let $a$ be an element of $HKer(s_2)$ (i.e.\ $s_2(a_1) \otimes a_2 = 1 \otimes a$) and $b$ and element of $HKer(t_2)$, then $\psi(a)$ belongs to $HKer(s'_1)$ ($s'_1 \cdot \psi (a_1) \otimes \psi (a_2) = 1 \otimes \psi(a) $) and $\psi(b)$ belongs to $HKer(t'_1)$. Hence, $ab= \psi^{-1}(\psi(a)\psi(b)) = \psi^{-1}(\psi(b)\psi(a)) = ba$. 

 We easily check the conditions of compatibility for being a cat$^2$-Hopf algebra in the Appendix \ref{D_4}.
This construction extends to the morphisms (Appendix \ref{D_5}). We obtain two functors that we denote $\hat{F} \colon  \sf{cat^2(\sf{Hopf}_{K,coc})} \to  \sf{X^2(\sf{Hopf}_{K,coc})}$ and $\hat{G} \colon  \sf{X^2(\sf{Hopf}_{K,coc})} \to  \sf{cat^2(\sf{Hopf}_{K,coc})}$. They describe an equivalence of categories between Hopf crossed squares and cat$^2$-Hopf algebras, this equivalence is proved in Appendix \ref{D_6}.
\end{proof}
Finally, thanks to this theorem we have completely described the notion of internal crossed modules (in the sense of Janelidze) in the category of crossed modules of cocommutative Hopf algebras, which are what we called Hopf crossed squares. This remark is clarified in the following corollary.

\begin{corollary}
$\mathsf{XMod}(\mathsf{HXMod_{K, coc}}) $ is equivalent to $\sf{X^2(\Hopf_{K,coc})}$.
\end{corollary}
\begin{proof}
This equivalence is clear thanks to Lemma \ref{equiv preli} and Theorem \ref{equi crossed squate 2-cat}
\end{proof}

In \cite{Emir}, K. Emir
 worked on a closely related concept, the notion of \emph{2-crossed modules of cocommutative Hopf algebras}. He defined a concept which unifies the  2-crossed modules of groups and of Lie algebras. Moreover, similarly to the case of groups and Lie algebras, he showed an equivalence between the category of simplicial
cocommutative Hopf algebras with Moore complex of length 2 and the category of 2-crossed modules of cocommutative Hopf algebras.
We know that this notion and the notion of Hopf crossed square are not equivalent. Indeed, in the particular case of groups, Conduch\'e \cite{Conduche} explained that from a crossed square we can form two crossed 2-modules whose associated simplicial groups
(see Proposition 3.4 in \cite{Conduche}) are homotopic, but they are not isomorphic. Moreover, Faria Martins  proved that in groups, 2-crossed modules are equivalent to what he called \emph{neat crossed squares} \cite{Martins}. However, it would be interesting to study the link between these two $2$-dimensional generalizations of crossed modules. Another interesting perspective would be to describe the non-abelian tensor product of cocommutative Hopf algebras by following the approach using the crossed squares in \cite{DV}. The notions of non-abelian tensor of groups \cite{BL} and Lie algebras \cite{Ellis2} would be special cases of this notion of non-abelian tensor product of cocommutative Hopf algebras.


\appendix
\section{A Hopf crossed square is a Hopf $2$-action}
\label{A_1}
We give some details of the proof of Proposition \ref{act sub crossed square}.

We consider a Hopf crossed square $(L,M,N,P,h, \lambda, \lambda', \mu, \nu)$, thanks to $\mu$ and $\nu$, we construct \[M \otimes L \to L \colon m \otimes l \mapsto {\mu(m)}  \triangleright l,\]  \[N \otimes L \to L \colon n \otimes l \mapsto {\nu(n)}  \triangleright l,\] which turns $L$ into an $M$-module Hopf algebra and  an $N$-module Hopf algebra.

We use the first property of crossed modules (CM1), \eqref{m et act} and the property of the antipode to obtain the axiom (2A1),
\[ {\mu({p_1}  \triangleright m)} \triangleright ({p_2}  \triangleright l) \stackrel{(CM1)}{=}  ({p_1\mu(m)S(p_2)) \act (p_3}  \triangleright l) \stackrel{\eqref{m et act}}{=}({p_1\mu(m)S(p_2)p_3})  \triangleright l \stackrel{\eqref{m et act}}{=}  {p}  \triangleright ({\mu(m)}  \triangleright l),\]
\[ {\nu({p_1} \triangleright n)} \triangleright ({p_2}  \triangleright l) \stackrel{(CM1)}{=}  ({p_1\nu(n)S(p_2)) \act (p_3}  \triangleright l) \stackrel{\eqref{m et act}}{=}  ({p_1\nu(n)S(p_2)p_3})  \triangleright l \stackrel{\eqref{m et act}}{=}  {p}  \triangleright ({\nu(n)}  \triangleright l).\]
Thanks to  the condition $(CS5)$ and \eqref{act et 1}, we prove that $h$ satisfies (2A3), for any $m \in M $ and any $n \in N$,
\[ h(1_M \otimes n) = h(\lambda(1_L) \otimes n) \stackrel{(CS5)}{=} 1_L ({\nu(n)}  \triangleright 1_L) \stackrel{\eqref{act et 1}}{=} \epsilon(n)1_L \]
\[ h(m \otimes 1_N) = h(m \otimes \lambda'(1_L)) \stackrel{(CS5)}{=} ({\mu(m)}  \triangleright 1_L)S(1_L) \stackrel{\eqref{act et 1}}{=} \epsilon(m)1_L\]
Since $(2A2)=(CS3)$ and $(2A4)=(CS6)$, the last condition to prove is  (2A5), i.e.\
 \begin{equation*}
({\mu(m_1)\nu(n_1)} \triangleright l )h( m_2 \otimes n_2) = h( m_1 \otimes n_1)  ({\nu(n_2)\mu(m_2)}  \triangleright l),
\end{equation*} 
for any $m \otimes n \otimes l \in M \otimes N \otimes L$. It is a consequence of the conditions $(CS4)$ and $(CS5)$. Indeed, on the one hand, we have for any $m' \otimes n' \otimes l'$ in $M \otimes N \otimes L$, 
\begin{align*}
h(\lambda(h(m' \otimes n')) \otimes \lambda'(l')) \stackrel{(CS5)}{=} &({\mu \cdot \lambda (h(m' \otimes n'))} \triangleright l'_1) S(l'_2) \\
\stackrel{(CS4)}{=}& \big({\mu (m'_1({\nu(n')} \triangleright S(m'_2)))}  \triangleright l'_1\big)S(l'_2) \\
\stackrel{\; \; \; \; \; \; \; \; \; \; }{=}& ({\mu (m'_1)\mu({\nu(n')} \triangleright S(m'_2))}  \triangleright l'_1)S(l'_2) \\
\stackrel{(CM1)}{=}& ({\mu (m'_1){\nu(n_1')} \mu(S(m'_2))S(\nu(n'_2))}  \triangleright l'_1)S(l'_2) 
\end{align*} 
  and on the other hand
\begin{align*}
h(\lambda(h(m' \otimes n')) \otimes \lambda'(l')) \stackrel{(CS5)}{=} &h(m' \otimes n')_1({\nu \cdot \lambda'(l') }  \triangleright S(h(m' \otimes n')_2))\\
\stackrel{\; \; \; \; \; \; \; \; \; \; }{=}& h(m'_1 \otimes n'_1)({\nu \cdot \lambda'(l') } \triangleright S(h(m'_2 \otimes n'_2)))  \\
 \stackrel{(CM2)}{=}& h(m'_1 \otimes n'_1)l'_1S(h(m'_2 \otimes n'_2))S(l'_2).
\end{align*} 
By combining these two equalities, we have the following identity
\begin{align*}
&\big({\mu (m'_1)\nu(n'_1)\mu(S(m'_2))S(\nu(n'_2))}  \triangleright l'_1\big)S(l'_2) \otimes l'_3 \otimes h(m'_3 \otimes n'_3)\\ & \hspace{2cm} = h(m'_1 \otimes n'_1)l'_1S(h(m'_2 \otimes n'_2))S(l'_2) \otimes l'_3 \otimes h(m'_3 \otimes n'_3) .
\end{align*}
By multiplying the four components of the tensor products together, we obtain 
\begin{equation*}
\big( {\mu (m'_1)\nu(n'_1)\mu(S(m'_2))S(\nu(n'_2))}  \triangleright l' \big) h(m'_3 \otimes n'_3) = h(m' \otimes n')l'.
\end{equation*}
For the particular element  $m' \otimes n' \otimes l' = m_1 \otimes n_1 \otimes ({\nu(n_2)\mu (m_2)}  \triangleright l)$, in $M \otimes N \otimes L$, with any $m \in M$, any $n \in N$ and any $l \in L$, we recover 
the condition (2A5) (thanks to \eqref{m et act}),
\[{(\mu(m_1) \act (\nu(n_1)}  \triangleright l ))h( m_2 \otimes n_2) = h( m_1 \otimes n_1)  ({\nu(n_2) \act (\mu(m_2)}  \triangleright l)).\]
We verified the five axioms of the definition of Hopf $2$-action (Definition \ref{2-action}), hence we conclude that $\sf X^2(Hopf_{K,coc})$ is a subcategory of $\sf Act^2(Hopf_{K,coc})$.

\section{Crossed square of Lie algebras}\label{E}
We give some elements of the proof of Proposition \ref{double adjunction}.
Let $(L,M,N,P,h, \lambda, \lambda', \mu, \nu)$ be a Hopf crossed square. By applying the functor $\mathcal{P} \colon \sf{Hopf}_{K,coc} \to \mathsf{Lie}_K$, we obtain the definition of a crossed square of Lie algebras. To give an idea of how it works, we prove that the condition (CS6) implies the condition $(v)$ of Definition \ref{x2mod lie}. We use the explicit form of the Hopf algebra structure of primitive elements (i.e.\ $\Delta (x) = x \ox 1_H + 1_H \ox x$ for any $x \in \mathcal{P}(H)$, with $H$ a cocommutative Hopf algebra). For any $m,m' \in M$ and $n \in N$, we have
\begin{align*}
h([m,m']\otimes n) &= h((mm' - m'm) \otimes n)\\
&= h(mm'\otimes n) - h(m'm \otimes n)\\
 & \stackrel{\mathclap{(CS6)}}{=} \;  ({\mu(m_1)}  \triangleright h(m' \otimes n_1))h(m_2 \otimes n_2)- ({\mu(m'_1)}  \triangleright h(m \otimes n_1))h(m'_2 \otimes n_2)\\
&= ({\mu(m)}  \triangleright h(m' \otimes n_1))h(1 \otimes n_2) + ({\mu(1)}  \triangleright h(m' \otimes n_1))h(m \otimes n_2) \\
& \hspace{2cm} - ({\mu(m')}  \triangleright h(m \otimes n_1))h(1 \otimes n_2) - ({\mu(1)}  \triangleright h(m \otimes n_1))h(m' \otimes n_2)\\
&= ({\mu(m)}  \triangleright h(m' \otimes n))h(1 \otimes 1) + ({\mu(1)}  \triangleright h(m' \otimes n))h(m \otimes 1) \\
& \hspace{2cm} {\mu(m)}  \triangleright h(m' \otimes 1))h(1 \otimes n) + ({\mu(1)}  \triangleright h(m' \otimes 1))h(m \otimes n) \\
& \hspace{2cm} - ({\mu(m')}  \triangleright h(m \otimes n))h(1 \otimes 1) - ({\mu(1)}  \triangleright h(m \otimes n))h(m' \otimes 1) \\
& \hspace{2cm} - ({\mu(m')}  \triangleright h(m \otimes 1))h(1 \otimes n) - ({\mu(1)}  \triangleright h(m \otimes 1))h(m' \otimes n) \\
&= ({\mu(m)}  \triangleright h(m' \otimes n)) - ({\mu(m')}  \triangleright h(m \otimes n)).
\end{align*} 
  
The last equality follows from the observations that a Hopf crossed square is a Hopf $2$-action, then it satisfies (2A3) and that $\epsilon(x) =0$ when $x$ is a primitive element.
\section{Equivalence of categories between $\sf Act^2(\mathsf{Hopf}_{K,coc})$ and $\sf Pt^2(\mathsf{Hopf}_{K,coc}) $}
  The appendixes \ref{B}, \ref{C_1}, \ref{C_3}, \ref{C_4} are used in the proof of Proposition \ref{equi action/split epi}.

\subsection{Construction of a Hopf $2$-action from a $2$-fold split epimorphism}\label{B}
Let consider $(H,N,$ $M,s_N,s_M)$ a $2$-fold split epimorphism.
The linear application  $h \colon (N \cap HKer(s_M)) \otimes (HKer(s_N) \cap M) \to HKer(s_N) \cap HKer(s_M)$ defined as \[h(n \otimes m) = n_1m_1S(n_2)S(m_2),\] is well-defined. Indeed,
 by using the fact that $s_N$ is a morphism of Hopf algebras, that $m \in HKer(s_N)$ and the cocommutativity, we can see that for $n \in N \cap HKer(s_M)$ and $m \in  HKer(s_N)\cap M$, $n_1m_1S(n_2)S(m_2)$ belongs to $HKer(s_N) \cap HKer(s_M)$ thanks to the following identities,
\begin{align*}
&s_N((n_1m_1S(n_2)S(m_2))_1) \otimes (n_1m_1S(n_2)S(m_2))_2 \\ &= s_N(n_1)s_N(m_1)s_N(S(n_4))s_N(S(m_4)) \otimes n_2m_2S(n_3)S(m_3) \\
&=  s_N(n_1)s_N(S(n_4)) \otimes n_1m_2S(n_2)S(m_3) \\
&=  1_N \otimes n_1m_1S(n_2)S(m_2).
\end{align*} 
 Similar computations work for $HKer(s_M)$.

Since we are dealing with cocommutative Hopf algebras, $h$ is a coalgebra morphism. Let us check that it is also an $(N \cap M)$-module morphism (i.e.\ condition $(2A2)$). With $p \in N \cap M$, $n \in N \cap HKer(s_M)$ and $m \in HKer(s_N) \cap M$, we obtain
\begin{align*}
h({p_1}  \triangleright n \otimes {p_2}  \triangleright m) &= ({p_1}  \triangleright n)_1({p_2}  \triangleright m)_1S(({p_1}  \triangleright n)_2)S(({p_2}  \triangleright m)_2) \\
&= (p_{1_1}nS(p_{1_2}))_1(p_{2_1}mS(p_{2_2}))_1S((p_{1_1}nS(p_{1_2}))_2)S((p_{2_1}mS(p_{2_2}))_2)\\
&= p_{1}n_1S(p_{2})p_{3}m_1S(p_{4})S(p_{5}n_2S(p_{6}))S(p_{7}m_2S(p_{8}))\\
 &=  p_{1}n_1m_1S(n_2)S(m_2)S(p_2)\\
&= p  \triangleright h(n \otimes m).
\end{align*} 

The properties $(2A1)$, $(2A3)$, $(2A4)$ and $(2A5)$ of the definition of Hopf $2$-action are easily checked by definition of the modules and of $h$ as it is shown in the following equalities, where $p \in N \cap M$, $n \in N \cap HKer(s_M)$ and $m,m' \in HKer(s_N) \cap M$,
\begin{align*}
({{p_1}  \triangleright m})  \triangleright ({p_2}  \triangleright l) &= ({p_1mS(p_2)})  \triangleright (p_3lS(p_4)) \\
&= (p_1mS(p_2))_1(p_3lS(p_4))S((p_1mS(p_2))_2) \\
&\stackrel{\mathclap{\eqref{coco}}}{=} \; p_1m_1lS(m_2)S(p_2)\\
&= p  \triangleright (m  \triangleright l)
\end{align*} 
\begin{align*}
h(1_N \otimes m) = 1_Nm_11_NS(m_2)= \epsilon(m)1_{HKer(s_N) \cap HKer(s_M)}
\end{align*} 
\begin{align*}
h(n \otimes mm') &= n_1m_1m'_1S(n_2)S(m'_2)S(m_2)\\
&= n_1m_1S(n_2)S(m_2)m_3n_3m'_1S(n_4)S(m'_2)S(m_4) \\
&=  h(n_1 \otimes m_1)({m_2}  \triangleright h(n_2 \otimes m'))
\end{align*} 
\begin{align*}
({n_1}  \triangleright ({m_1}  \triangleright l)) h( n_2 \otimes m_2) &= n_1m_1lS(m_2)S(n_2)n_3m_3S(n_4)S(m_4)\\
&= n_1m_1lS(n_2)S(m_2)\\
 &= n_1m_1S(n_2)S(m_2)m_3n_3lS(n_4)S(m_4)\\
&= h( n_1 \otimes m_1)  ({m_2}  \triangleright ({n_2}  \triangleright l)).
\end{align*} 
The second equalities also hold for $(2A1)$, $(2A3)$ and $(2A4)$ with a similar computation as above. Hence, we conclude that we formed a Hopf $2$-action from the $2$-split epimorphism $(H,N,M,s_N,t_N)$.

\subsection{Construction of an $(L \rtimes N)$-module Hopf algebra from a Hopf 2-action}\label{C_1}
Let $(L, M, N,$ $ P, h)$ be a Hopf $2$-action. We check that the following linear map defines $L \rtimes N$ as an $(M \rtimes P)$-module Hopf algebra,
\begin{equation} \label{action C1}
{(m \otimes p)}  \triangleright (l \otimes n) = ({m_1}  \triangleright ({p_1}  \triangleright l))h(m_2 \otimes {p_2}  \triangleright n_1) \otimes {p_3}  \triangleright n_2.
\end{equation}
We verify the condition \eqref{m et act} for the linear map defined as \eqref{action C1} thanks to the following computations for any $l \in L$, $m,m' \in M$, $n \in N$ and $p,p' \in P$,
\begin{align*}
& {(m \otimes p)}  \triangleright \big({(m' \otimes p')}  \triangleright (l \otimes n) \big) = {(m \otimes p)}  \triangleright \big(({m'_1}  \triangleright ({p'_1}  \triangleright l))h(m'_2 \otimes {p'_2}  \triangleright n_1) \otimes {p'_3}  \triangleright n_2 \big) \\
&\stackrel{ \; \; \; \; \; \; \; \; \; \; \; \; \; \; \; \; \;}{=}\Big({m_1} \triangleright \big({p_1}  \triangleright ({m'_1}  \triangleright ({p'_1}  \triangleright l))h(m'_2 \otimes {p'_2}  \triangleright n_1) \big)\Big)h(m_2 \otimes {p_2}  \triangleright ({p'_3}  \triangleright n_2)_1) \otimes {p_3}  \triangleright ({p'_3}  \triangleright n_2)_2\\
&\stackrel{ \eqref{act et m} + \eqref{m et act}}{=} \big({m_1}  \triangleright ({p_1}  \triangleright ({m'_1}  \triangleright ({p'_1}  \triangleright l)))\big)\big({m_2}\triangleright ({p_2}  \triangleright h(m'_2 \otimes {p'_2}  \triangleright n_1))\big)  h(m_3 \otimes ({p_3p'_3}) \triangleright n_2)\\
& \hspace{2cm } \otimes ({p_4p'_4})  \triangleright n_3\\
&\stackrel{ \; \; \; \;  (2A2) \; \; \; \; }{=} \big({m_1}  \triangleright ({p_1}  \triangleright ({m'_1} \triangleright ({p'_1}  \triangleright l)))\big)\big({m_{2}}  \triangleright h({p_{2_1}}  \triangleright m'_2 \otimes ({p_{2_2}}  \triangleright ({p'_2}  \triangleright n_1)))\big)h(m_{3} \otimes ({p_3p'_3})  \triangleright n_2) \\ 
& \hspace{2cm} \otimes ({p_4p'_4})  \triangleright n_3\\
&\stackrel{ \; \; \; \; \eqref{m et act} \; \; \; \; }{=}  \big({m_1}  \triangleright ({p_1}  \triangleright ({m'_1} \triangleright ({p'_1}  \triangleright l)))\big)\big({m_{2_1}}  \triangleright h({p_2}  \triangleright m'_2 \otimes (({p_3p'_2})  \triangleright n_1)_1)) h(m_{2_2} \otimes (({p_3p'_2}) \triangleright n_1)_2) \\ 
& \hspace{2cm} \otimes ({p_4p'_3})  \triangleright n_2\\
&\stackrel{ \; \; \; \; (2A4) \; \; \; \; }{=}  ({m_1} \triangleright ({p_1} \triangleright ({m'_1} \triangleright ({p'_1} \triangleright l))))h(m_2 ({p_2} \triangleright m'_2) \otimes ({p_3p'_2})\triangleright n_1)  \otimes ({p_4p'_3})\triangleright n_2\\
&\stackrel{ \; \; \; \;  (2A1) \; \; \; \;}{=}  ({m_1}\triangleright (({{p_{1_1}} \triangleright m'_1}) \triangleright ({p_{1_2}}\triangleright ({p'_1}\triangleright l))))h(m_2({p_2} \triangleright m'_2) \otimes ({p_3p'_2}) \triangleright n_1)  \otimes ({p_4p'_3}) \triangleright n_2\\
&\stackrel{ \eqref{coco} + \eqref{m et act}}{=}  ({(m({p_1} \triangleright m'))_1} \triangleright ({(p_2p')_1}\triangleright l))h((m({p_1} \triangleright m'))_2 \otimes {(p_2p')_2} \triangleright n_1)  \otimes {(p_2p')_3} \triangleright n_2\\
&\stackrel{\;  def. \;\;  \eqref{action C1}}{=} ({m({p_1} \triangleright m') \otimes p_2p')} \triangleright (l \otimes n) \\
&\stackrel{ def. \; M_{M \rtimes P}}{=} \big({(m \otimes p)(m' \otimes p')}\big) \triangleright (l \otimes n) .
\end{align*} 

The cocommutativity of the Hopf algebras $M$ and $P$, the property \eqref{act et delta} and the fact that $h$ is a morphism of coalgebras imply the condition \eqref{act et delta} for the linear map defined by \eqref{action C1} for any $l \in L$, $m \in M$, $n \in N$ and $p \in P$.
\begin{align*}
&({(m \otimes p)} \triangleright (l \otimes n))_1  \otimes ({(m \otimes p)} \triangleright (l \otimes n))_2\\ & = (({m_1} \triangleright ({p_1} \triangleright l))h(m_2 \otimes {p_2} \triangleright n_1) \otimes {p_3} \triangleright n_2)_1 \otimes (({m_1} \triangleright ({p_1} \triangleright l))h(m_2 \otimes {p_2} \triangleright n_1) \otimes {p_3} \triangleright n_2)_2 \\
&\stackrel{ \mathclap{\eqref{act et delta}}}{=} \; ({m_{1_1}} \triangleright ({p_{1_1}} \triangleright l_1))h(m_2 \otimes {p_2} \triangleright n_1)_1 \otimes {p_{3_1}} \triangleright n_{2_1} \otimes ({m_{1_2}} \triangleright ({p_{1_2}} \triangleright l_2))h(m_2 \otimes {p_2} \triangleright n_1)_2 \otimes {p_{3_2}} \triangleright n_{2_2} \\
&= ({m_{1}} \triangleright ({p_{1}} \triangleright l_1))h(m_3 \otimes {p_3} \triangleright n_1) \otimes {p_{5}} \triangleright n_{3} \otimes ({m_{2}} \triangleright ({p_{2}} \triangleright l_2))h(m_4 \otimes {p_4} \triangleright n_2) \otimes {p_{6}} \triangleright n_{4} \\
&\stackrel{ \mathclap{\eqref{coco}}}{=} \;  ({m_{1}} \triangleright ({p_{1}} \triangleright l_1))h(m_2 \otimes {p_2} \triangleright n_1) \otimes {p_{3}} \triangleright n_{2} \otimes ({m_{3}} \triangleright ({p_{4}} \triangleright l_2))h(m_4 \otimes {p_5} \triangleright n_3) \otimes {p_{6}} \triangleright n_{4} \\
&= {(m_1 \otimes p_1)} \triangleright (l_1 \otimes n_1) \otimes {(m_2 \otimes p_2)} \triangleright (l_2 \otimes n_2).
\end{align*} 

Thanks to the following equalities, for any $l \in L$, $m \in M$, $n \in N$ and $p \in P$,
\begin{align*}
(\epsilon \otimes \epsilon)({(m \otimes p)} \triangleright (l \otimes n)) &=(\epsilon \otimes \epsilon)(({m_{1}} \triangleright ({p_{1}} \triangleright l))h(m_2 \otimes {p_2} \triangleright n_1) \otimes {p_3} \triangleright n_2) \\
&= \epsilon({m_{1}} \triangleright ({p_{1}} \triangleright l))\epsilon(h(m_2 \otimes {p_2} \triangleright n_1)) \otimes \epsilon({p_3} \triangleright n_2)) \\
&= \epsilon(l) \epsilon(h(m \otimes p \triangleright n)) \\
&= \epsilon(l)\epsilon(m) \epsilon(p) \epsilon(n),
\end{align*} 
  the linear map \eqref{action C1} satisfies \eqref{act et epsilon} thanks to the properties of the linear map $\epsilon$ (\eqref{act et epsilon}, $\epsilon \cdot h =h$).

The conditions \eqref{1 et act} and \eqref{act et 1} for \eqref{action C1} are satisfied thanks to (2A3) for any $l \in L$, $m \in M$, $n \in N$ and $p \in P$:
\begin{align*}
{(1_M \otimes 1_P)} \triangleright (l \otimes n) &= ({1_M} \triangleright (1_P \triangleright l))h(1_M \otimes {1_P} \triangleright n_1) \otimes {1_P} \triangleright n_2\\
 &\stackrel{\mathclap{(2A3)}}{=} \; l \epsilon(n_1) \otimes n_2\\
&= l \otimes n,
\end{align*} 

\begin{align*}
{(m \otimes p)} \triangleright (1_L \otimes 1_N )&= ({m_{1}} \triangleright ({p_{1}} \triangleright 1_L))h(m_2 \otimes {p_2} \triangleright 1_N) \otimes {p_3} \triangleright 1_N \\
&\stackrel{\mathclap{\eqref{act et epsilon}}}{=} \;  h(m\epsilon(p) \otimes 1_N) \otimes 1_N \\
&\stackrel{\mathclap{(2A3)}}{=} \; \epsilon(m)\epsilon(p)(1_L \otimes 1_N).
\end{align*} 
  
Finally, we verify the condition \eqref{act et m} for \eqref{action C1} for any $l,l' \in L$, $m \in M$, $n,n' \in N$ and $p \in P$. It holds thanks to several properties of the definition of Hopf $2$-actions and the modules, as we can see in the following identities.
\begin{align*}
&{(m \otimes p)} \triangleright \big((l \otimes n) (l' \otimes n')\big)= {(m \otimes p)} \triangleright (l ({n_1} \triangleright l' )\otimes n_2n') \\
&= ({m_{1}} \triangleright ({p_{1}} \triangleright l({n_1} \triangleright l')))h(m_2 \otimes {p_2} \triangleright (n_2n'_1)) \otimes {p_3} \triangleright (n_3n'_2)\\
&\stackrel{\mathclap{\eqref{act et m}}}{=} \;  ({m_{1}} \triangleright ({p_{1}} \triangleright l)) ({m_{2}} \triangleright ({p_{2}} \triangleright ({n_1} \triangleright l')))h(m_3 \otimes {p_3} \triangleright (n_2n'_1)) \\ & \hspace{2cm}\otimes {p_4} \triangleright (n_3n'_2)\\
&\stackrel{\mathclap{(2A1)}}{=}\; ({m_1} \triangleright ({p_1} \triangleright l)) ({m_{2}}\triangleright (({{p_{2_1}} \triangleright n_1}) \triangleright ({p_{2_2}} \triangleright l')))h(m_3 \otimes {p_3} \triangleright (n_2n'_1)) \\ & \hspace{2cm} \otimes {p_4} \triangleright (n_3n'_2)\\
&\stackrel{\mathclap{\eqref{act et m}}}{=}\; ({m_1} \triangleright ({p_1} \triangleright l)) ({m_{2}}\triangleright (({{p_{2}}\triangleright n_1})\triangleright ({p_{3}}\triangleright l')))h(m_3 \otimes ({p_4} \triangleright n_2) (p_5 \triangleright n'_1)) \\ & \hspace{2cm} \otimes {p_6} \triangleright (n_3n'_2)\\
&\stackrel{\mathclap{(2A4)}}{=}\; ({m_1} \triangleright ({p_1} \triangleright l)) ({m_{2}}\triangleright (({{p_{2}}\triangleright n_1}) \triangleright  ({p_{3}} \triangleright l')))h(m_3 \otimes ({p_4} \triangleright n_2)_1) \\ & \hspace{2cm}(({{p_4}\triangleright n_2)_2} \triangleright h(m_4 \otimes {p_5} \triangleright n'_1))\otimes {p_6} \triangleright (n_3n'_2)\\
&\stackrel{\mathclap{\eqref{act et delta}}}{=}\; ({m_1} \triangleright ({p_1} \triangleright l)) ({m_{2}} \triangleright (({{p_{2}} \triangleright n_1}) \triangleright ( {p_{3}} \triangleright l')))h(m_3 \otimes {p_4} \triangleright n_2)   \\
& \hspace{2cm } (({{p_5} \triangleright n_3})\triangleright h(m_4 \otimes {p_6} \triangleright n'_1))\otimes  {p_7}\triangleright (n_4n'_2)\\
&\stackrel{ \mathclap{\eqref{coco}}}{=} \; ({m_1} \triangleright ({p_1} \triangleright l))( {m_{2}}\triangleright ({({p_{2}} \triangleright n_1)_1}\triangleright ({p_{3}} \triangleright l')))h(m_3 \otimes ({p_2} \triangleright n_1)_2) \\ & \hspace{2cm}(({{p_4} \triangleright n_2})\triangleright h(m_4 \otimes {p_5} \triangleright n'_1))\otimes {p_6} \triangleright (n_3n'_2)\\
&\stackrel{\mathclap{(2A5)}}{=} \;({m_1} \triangleright ({p_1} \triangleright l)) h(m_2 \otimes ({p_2} \triangleright n_1)_1) ({({p_2} \triangleright n_1)_2} \triangleright ({m_3}\triangleright ({p_3} \triangleright l'))) \\ & \hspace{2cm}(({{p_4} \triangleright n_2}) \triangleright h(m_4 \otimes {p_5} \triangleright n'_1)) \otimes {p_6} \triangleright (n_3n'_2)\\
&\stackrel{\mathclap{\eqref{act et m}}}{=}\; ({m_1} \triangleright ({p_1} \triangleright l)) h(m_2 \otimes ({p_2} \triangleright n_1)) \\ & \hspace{2cm} \big({({p_3} \triangleright n_2)}\triangleright (({m_3}\triangleright ({p_4} \triangleright l'))h(m_4 \otimes {p_5} \triangleright n'_1))\big)\otimes {p_6} \triangleright (n_3n'_2)\\
&\stackrel{\mathclap{\eqref{act et m}}}{=} \; \Big(({m_1} \triangleright ({p_1} \triangleright l))h(m_2 \otimes {p_2} \triangleright n_1) \otimes {p_3} \triangleright n_2\Big)\\ & \hspace{2cm} \Big(({m_3} \triangleright ({p_4} \triangleright l'))h(m_4 \otimes {p_5} \triangleright n'_1) \otimes {p_6} \triangleright n'_2\Big)\\
&= ({(m_1 \otimes p_1)} \triangleright (l \otimes n))( {(m_2 \otimes p_2)} \triangleright (l' \otimes n') ).
\end{align*} 
  
In conclusion, $L \rtimes N$ is an $(M \rtimes P)$-module Hopf algebra via the linear application defined in \eqref{action C1}.

\subsection{Extension of the construction to the morphisms}\label{C_3}
Let $f \colon H \to H'$ be a morphism of $2$-fold split epimorphisms of Hopf algebras
\begin{center}
\begin{tikzpicture}[descr/.style={fill=white},scale=0.9]
\node (A) at (0,0) {$H'$};
\node (B) at (0,2) {$H$};
\node (C) at (2,2) {$N$};
\node (D) at (2,0) {$N'$};
  \path[-stealth]
 (B.south) edge node[right] {$f$} (A.north) 
 (C.south) edge node[right] {$f|_N$} (D.north);
 \path[->,font=\scriptsize]
(A.east) edge node[above] {$s'_{N'}$} (D.west)
(B.east) edge node[above] {$s_N$} (C.west);
\end{tikzpicture}
\begin{tikzpicture}[descr/.style={fill=white},scale=0.9]
\node (A) at (0,0) {$H'$};
\node (B) at (0,2) {$H$};
\node (C) at (2,2) {$M$};
\node (D) at (2,0) {$M'$};
  \path[-stealth]
 (B.south) edge node[right] {$f$} (A.north) 
 (C.south) edge node[right] {$f|_M$} (D.north);
 \path[->,font=\scriptsize]
(A.east) edge node[above] {$s'_{M'}$} (D.west)
(B.east) edge node[above] {$s_M$} (C.west);
\end{tikzpicture}
\end{center}
We construct the  morphisms of the associate Hopf $2$-actions, defined by the restrictions of $f$. For any $ p \in N \cap M,$ $ l \in HKer(s_N) \cap HKer(s_M)$, $n \in N \cap HKer(s_M)$ and $ m \in HKer(s_N) \cap M$, we have
\begin{align*}
f(p \triangleright n) &= f(p_1)f(n)f(S(p_2)) = {f(p)} \triangleright f(n), \\
f(p \triangleright l) &= f(p_1)f(l)f(S(p_2)) = {f(p)} \triangleright f(l),  \\
f(p \triangleright m) &= f(p_1)f(m)f(S(p_2)) = {f(p)} \triangleright f(m),\\
f(m \triangleright l) &= f(m_1)f(l)f(S(m_2)) = {f(m)} \triangleright f(l), \\
f(n \triangleright l) &= f(n_1)f(l)f(S(n_2)) = {f(n)} \triangleright f(l), \\
f \cdot h (n \otimes m) &= f(n_1m_1S(n_2)S(m_2)) = h'(f(n) \otimes f(m)).
\end{align*} 
  They satisfy the properties of a morphism of Hopf $2$-actions.

Conversely, if we have $(\alpha \colon L \to L', \beta \colon M \to M', \gamma \colon N \to N' , \delta \colon P \to P')$ a morphism of Hopf $2$-actions we can form $\alpha \otimes \gamma \otimes \beta \otimes \delta \colon (L \rtimes N) \rtimes (M \rtimes P) \to (L' \rtimes N') \rtimes (M' \rtimes P') $, which is a morphism of $2$-fold split epimorphisms.

First, let us check that $\alpha \otimes \gamma \otimes \beta \otimes \delta \colon (L \rtimes N) \rtimes (M \rtimes P) \to (L' \rtimes N') \rtimes (M' \rtimes P') $ is a morphism of Hopf algebras (the only non-trivial part is the respect of the multiplication), thanks to the axioms of the definition of a morphism of Hopf $2$-actions (Definition \eqref{morph hopf-2-action}): 
\begin{align*}
& (\alpha \otimes \gamma \otimes \beta \otimes \delta)((l \otimes n \otimes m \otimes p) (l' \otimes n' \otimes m' \otimes p')) \\ &= (\alpha \otimes \gamma \otimes \beta \otimes \delta)(l \Big( {n_1} \triangleright \big({m_1} \triangleright ({p_1} \triangleright l')\big) \Big)({n_2} \triangleright h(m_2 \otimes {p_2} \triangleright n'_1))  \otimes n_3({p_3} \triangleright n'_2) \otimes m_3({p_4}\triangleright m') \\ & \hspace{2cm}  \otimes p_5p')\\
&= \alpha \big( l \Big( {n_1} \triangleright \big({m_1} \triangleright ({p_1} \triangleright l')\big) \Big)({n_2} \triangleright h(m_2 \otimes {p_2} \triangleright n'_1))\big) \otimes \gamma(n_3 ({p_3} \triangleright n'_2)) \otimes \beta(m_3 ({p_4} \triangleright m')) \otimes \delta(p_5p')\\
&= \alpha( l) \Big( {\gamma(n_1)} \triangleright \big({\beta(m_1)} \triangleright ({\delta(p_1)} \triangleright \alpha(l'))\big) \Big)({\gamma(n_2)} \triangleright \alpha \cdot h(m_2 \otimes {p_2} \triangleright n'_1)) \\ & \hspace{2cm} \otimes \gamma(n_3)({\delta(p_3)} \triangleright \gamma(n'_2)) \otimes \beta(m_3)({\delta(p_4)} \triangleright \beta(m')) \otimes \delta(p_5)\delta(p')\\
&= \alpha( l) \Big( {\gamma(n_1)} \triangleright \big({\beta(m_1)} \triangleright ({\delta(p_1)} \triangleright \alpha(l'))\big) \Big)({\gamma(n_2)} \triangleright h(\beta(m_2) \otimes {\delta(p_2)} \triangleright \gamma(n'_1))) \\ & \hspace{2cm} \otimes \gamma(n_3)({\delta(p_3)} \triangleright \gamma(n'_2)) \otimes \beta(m_3)({\delta(p_4)} \triangleright \beta(m')) \otimes \delta(p_5)\delta(p')\\
&= (\alpha(l) \otimes \gamma(n) \otimes \beta(m) \otimes \delta(p))(\alpha(l') \otimes \gamma(n') \otimes \beta(m') \otimes \delta(p')).
\end{align*} 
  
This morphism of Hopf algebras is moreover a morphism of $2$-fold split epimorphisms, since the two following diagrams commute
\begin{center}
\begin{tikzpicture}[descr/.style={fill=white},scale=0.9]
\node (A) at (0,0) {$(L' \rtimes N') \rtimes (M' \rtimes P')$};
\node (B) at (0,2) {$(L \rtimes N) \rtimes (M \rtimes P)$};
\node (C) at (4,2) {$M \rtimes P$};
\node (D) at (4,0) {$M' \rtimes P'$};
  \path[-stealth]
 (B.south) edge node[descr] {$\alpha \otimes \gamma \otimes \beta \otimes \delta$} (A.north) 
 (C.south) edge node[right] {$\beta \otimes \delta$} (D.north);
 \path[->,font=\scriptsize]
(A.east) edge node[above] {$s'_1$} (D.west)
(B.east) edge node[above] {$s_1$} (C.west);
\end{tikzpicture}
\begin{tikzpicture}[descr/.style={fill=white},scale=0.9]
\node (A) at (0,0) {$(L' \rtimes N') \rtimes (M' \rtimes P')$};
\node (B) at (0,2) {$(L \rtimes N) \rtimes (M \rtimes P)$};
\node (C) at (4,2) {$N \rtimes P$};
\node (D) at (4,0) {$N' \rtimes P'$};
  \path[-stealth]
 (B.south) edge node[descr] {$\alpha \otimes \gamma \otimes \beta \otimes \delta$} (A.north) 
 (C.south) edge node[right] {$ \gamma \otimes \delta$} (D.north);
 \path[->,font=\scriptsize]
(A.east) edge node[above] {$s'_2$} (D.west)
(B.east) edge node[above] {$s_2$} (C.west);
\end{tikzpicture}
\end{center}
  as it is proven in the following equalities,
\begin{align*}
(\beta \otimes \delta) \cdot s_1 (l \otimes n \otimes m \otimes p) &= (\beta \otimes \delta)(\epsilon(l) \epsilon(n) m \otimes p) \\
&= \epsilon(l)\epsilon(n) \beta(m) \otimes \delta(p)\\
&= s'_1(\alpha(l) \otimes \gamma(n) \otimes \beta(m) \otimes \delta(p))
\end{align*} 
\begin{align*}
(\gamma \otimes \delta) \cdot s_2 (l \otimes n \otimes m \otimes p) &= (\gamma \otimes \delta)(\epsilon(l) \epsilon(m) n \otimes p) \\
&= \epsilon(l)\epsilon(m) \gamma(n) \otimes \delta(p)\\
&= s'_2(\alpha(l) \otimes \gamma(n) \otimes \beta(m) \otimes\delta(p)).
\end{align*} 

\subsection{Equivalence of categories between $\sf Act^2(Hopf_{K,coc})$ and $\sf Pt^2(Hopf_{K,coc})$}\label{C_4}
Let us consider a $2$-fold split epimorphism $(H,M,N,s_N, s_M)$, by applying the functors $\tilde{F}$ and $\tilde{G}$, we obtain the following $2$-fold split epimorphism,
\begin{align*}
&H' = (HKer(s_N) \cap HKer(s_M)) \rtimes( HKer(s_N) \cap M ) \rtimes (N \cap HKer(s_M)) \rtimes( N \cap M), \\ &M'=(HKer(s_N) \cap M) \rtimes (N \cap M), \\
&N' =(N \cap HKer(s_M)) \rtimes (N \cap M) \\
&s_1 \colon  H' \to N'\colon s_1(l \otimes n \otimes m \otimes p) = \epsilon(l)\epsilon(n) (m \otimes p) \\
&s_2 \colon  H' \to M' \colon s_2(l \otimes n \otimes m \otimes p) = \epsilon(l)\epsilon(m) (n \otimes p)
\end{align*} 

Thanks to the equivalence between modules of Hopf algebras and split epimorphisms that was recalled in Lemma \ref{iso crossed product} we have the following isomorphisms,
\begin{align*}
N'& = (N \cap HKer(s_M)) \rtimes (N \cap M) \cong N \cap H = N\\
M'& = (HKer(s_N) \cap M) \rtimes (N \cap M) \cong H \cap M = M\\
H'& = (HKer(s_N) \cap HKer(s_M)) \rtimes( HKer(s_N) \cap M ) \rtimes (N \cap HKer(s_M)) \rtimes( N \cap M)  \cong H,
\end{align*} 
  
induced by the multiplication in $H$. Moreover, the isomorphism
 \begin{equation}\label{iso split epi equi}
 \phi \colon H' \to H \colon (l \otimes n \otimes m \otimes p) \mapsto lnmp,
 \end{equation} 
 is a morphism of $2$-fold split epimorphisms, since it makes the two following diagrams commute \begin{center}
\begin{tikzpicture}[descr/.style={fill=white},scale=0.9]
\node (A) at (0,0) {$H$};
\node (B) at (0,2) {$H'$};
\node (C) at (4,2) {$N'$};
\node (D) at (4,0) {$N$};
  \path[-stealth]
 (B.south) edge node[descr] {$\phi$} (A.north) 
 (C.south) edge node[right] {$\phi|_{N'}$} (D.north);
 \path[->,font=\scriptsize]
(A.east) edge node[above] {$s_N$} (D.west)
(B.east) edge node[above] {$s_1$} (C.west);
\end{tikzpicture}
\begin{tikzpicture}[descr/.style={fill=white},scale=0.9]
\node (A) at (0,0) {$H$};
\node (B) at (0,2) {$H'$};
\node (C) at (4,2) {$M'$};
\node (D) at (4,0) {$M.$};
  \path[-stealth]
 (B.south) edge node[descr] {$\phi$} (A.north) 
 (C.south) edge node[right] {$ \phi|_{M'}$} (D.north);
 \path[->,font=\scriptsize]
(A.east) edge node[above] {$s_M$} (D.west)
(B.east) edge node[above] {$s_2$} (C.west);
\end{tikzpicture}
\end{center} Indeed, for any $l \in (HKer(s_N) \cap HKer(s_M))$, $ n \in ( HKer(s_N) \cap M )$, $m \in (N \cap HKer(s_M))$ and $p \in ( N \cap M)$ we have, thanks to the property of $HKer(s_N)$, the following equalities  
\begin{align*}
\phi|_{N'} \cdot s_1(l \otimes n \otimes m \otimes p) &= \phi|_{N'} (\epsilon(l)\epsilon(n)m \otimes p) \\
&= \epsilon(l)\epsilon(n)mp\\
&= s_N(lnmp)\\
&= s_N \cdot \phi( l \otimes n \otimes m \otimes p)
\end{align*} 
Note that similar computations as above make the right diagram commutes. 

Conversely, let $(L,M,N,P,h)$ be a Hopf $2$-action, if we apply $\tilde{G}$ and then $\tilde{F}$ we obtain the following Hopf $2$-action  
\begin{align*}
&((N \rtimes P) \cap (M \rtimes P)) \otimes (HKer(s_1) \cap HKer(s_2)) \to HKer(s_1) \cap HKer(s_2) \colon x \ox k \mapsto x_1kS(x_2),\\
&((N \rtimes P) \cap (M \rtimes P)) \otimes ((M \rtimes P) \cap HKer(s_2)) \to M \rtimes P) \cap HKer(s_2) \colon x \ox k \mapsto x_1kS(x_2),\\
&((N \rtimes P) \cap (M \rtimes P))\otimes (HKer(s_1) \cap (N \rtimes P)) \to HKer(s_1) \cap (N \rtimes P) \colon x \ox k \mapsto x_1kS(x_2),\\
&((M \rtimes P) \cap HKer(s_2)) \otimes (HKer(s_1) \cap HKer(s_2))  \to HKer(s_1) \cap HKer(s_2) \colon x \ox k \mapsto x_1kS(x_2),\\
& (HKer(s_1) \cap (N \rtimes P)) \otimes (HKer(s_1) \cap HKer(s_2))  \to HKer(s_1) \cap HKer(s_2) \colon x \ox k \mapsto x_1kS(x_2),
\end{align*} 
defined by the action of conjugation, with $h \colon ((M \rtimes P) \cap HKer(s_2)) \otimes (HKer(s_1) \cap (N \rtimes P)) \to (HKer(s_1) \cap HKer(s_2)), h(n \ox m) = n_1m_1S(n_2)S(m_2)$, where 
\begin{align*}
&s_1 \colon  (L\rtimes N)\rtimes (M \rtimes P) \to (M \rtimes P) \colon s_1(l \otimes n \otimes m \otimes p) = \epsilon(l)\epsilon(n) (m \otimes p), \\
&s_2 \colon  (L\rtimes N)\rtimes (M \rtimes P) \to (N \rtimes P) \colon s_2(l \otimes n \otimes m \otimes p) = \epsilon(l)\epsilon(m) (n \otimes p).
\end{align*} 
  
Since $s_1$ and $s_2$ are  just projections, it is obvious that \[HKer(s_1) \cong L \rtimes N\]
\[ HKer(s_2) \cong L \rtimes M.\]

The expression $(N \rtimes P)  \cap (M \rtimes P)$ stands for the intersection of the Hopf subalgebras $K \ox N \ox K \ox P $ and $K \ox K \ox M \ox P $ of $ (L\rtimes N)\rtimes (M \rtimes P)$.

More precisely, the elements in $(N \rtimes P)  \cap (M \rtimes P)$ are given by \[ \{ l \ox n \ox m \ox p \in (L\rtimes N)\rtimes (M \rtimes P) \\ \mid l \ox n \ox m \ox p \in (K \ox N \ox K \ox P ) \cap (K \ox K \ox M \ox P ) \}.\]

Let $l \ox n \ox m \ox p$ be an element in $(N \rtimes P)  \cap (M \rtimes P) $, it implies that there exists $\tilde{n} \in N$, $\hat{m} \in M$ and $\tilde{p},\hat{p} \in P$ such that \[l \ox n \ox m \ox p = 1_L \ox \tilde{n} \ox 1_M \ox \tilde{p} = 1_L \ox 1_N \ox \hat{m} \ox \hat{p}.\]
Since $1_N$ is non-zero and hence linearly independent, it can be completed to a base $\{1_N,\tilde{n}_i\}_{i\in I}$ for $N$, where $I$ is an index set whose cardinality equals the dimension of $N$. We decompose the element $ 1_L \ox \tilde{n} \ox 1_M \ox \tilde{p} $ as 
\[ 1_L \ox 1_N \ox 1_M \ox q + \sum_{i \in I} 1_L \ox \tilde{n}_i \ox 1_M \ox \tilde{p}_i, \] where the elements $q,\tilde{p}_i$ belong to $P$ for $i \in I$ (this decomposition exists and is unique). This expression  has to be equal to $1_L \ox 1_N \ox \hat{m} \ox \hat{p}$.

Since $1_N$ is linearly independent of the $\tilde{n}_i$, it follows that $\sum_{i \in I} 1_L \ox \tilde{n}_i \ox 1_M \ox \tilde{p}_i = 0$ and therefore $l \ox n \ox m \ox p = 1_L \ox \tilde{n} \ox 1_M \ox \tilde{p} =  1_L \ox 1_N \ox 1_M \ox q$.

Hence, $(N \rtimes P)  \cap (M \rtimes P) \subseteq \{ 1_L \ox 1_N \ox 1_M \ox p  \mid p \in P\} \subseteq  (L\rtimes N)\rtimes (M \rtimes P)$. The other inclusion is obvious and the canonical isomorphism between $K \ox K \ox K \ox P$ and $P$ allows us to say that there is an isomorphism of Hopf algebras between $(N \rtimes P)  \cap (M \rtimes P)$ and $P$.

Similar arguments imply that $HKer(s_1) \cap HKer(s_2) \cong L$,  $(M \rtimes P) \cap HKer(s_2) \cong M$ and $ HKer(s_1) \cap (N \rtimes P) \cong N$, via the canonical isomorphisms. It is clear that these canonical isomorphisms form a morphism of Hopf $2$-actions.

We can conclude that we have an equivalence of categories.
 
\section{Equivalence of categories between $\sf{cat}^2(\mathsf{Hopf}_{K,coc})$ and $\sf{X}^2(\mathsf{Hopf}_{K,coc})$}

   The following appendixes, \ref{D_2}, \ref{C_2}, \ref{D_3}, \ref{D_4}, \ref{D_5}, \ref{D_6}, are used in the proof of Proposition \ref{equi crossed squate 2-cat}.
 \subsection{Construction of a Hopf crossed square}\label{D_2}

From $(H,N,M,s_N,t_N,s_M,t_M)$ a cat$^2$-Hopf algebra, we construct 
\begin{equation}\label{square annex}
\begin{tikzpicture}[descr/.style={fill=white},baseline=(A.base)]
\node (X) at (0,-3) {$HKer(s_N) \cap M$};
\node (A) at (0,0) {$HKer(s_N) \cap HKer(s_M)$};
\node (B) at (5,0) {$N \cap HKer(s_M)$};
\node (C) at (5,-3) {$N \cap M,$};
\path[-stealth] 
  (A.south) edge node[right] {$t_M$}  (X.north)
  (X.east) edge node[above] {$t_N$}  (C.west)
   (A.east)  edge node[above] {$t_N$} (B.west)
(B.south) edge node[right] {$t_M$} (C.north);
\end{tikzpicture}
\end{equation}
where the modules are defined thanks to the action by conjugation and $h \colon (N \cap HKer(s_M)) \otimes (HKer(s_N) \cap M) \to HKer(s_N) \cap HKer(s_M)$ is given by  $h( n \otimes m )= n_1m_1S(n_2)S(m_2)$, as in the proof of Proposition \ref{equi action/split epi}.

As an intermediate step, we prove that $t_N(HKer(s_N) \cap HKer(s_M))$ is included in $N \cap HKer(s_M)$. Let $k$ be an element of $HKer(s_N) \cap HKer(s_M)$, then $t_N(k)$ is in $N$ and moreover $t_N(k)$ belongs to $HKer(s_M)$, since we have
\begin{align*}
(s_M \cdot t_N)(k_1) \otimes t_N(k_2) &\stackrel{(2C4)}{=} (t_N \cdot s_M)(k_1) \otimes t_N(k_2) \\
&= 1_M \otimes t_N(k).
\end{align*} 
  
The morphism $t_N \colon HKer(s_N) \cap M \to N \cap M$ is well-defined since for any element $m$ in $M$, 
\[t_N(m) \stackrel{t_M \cdot i_M = Id_M}{=} t_N \cdot t_M (m) \stackrel{(2C2)}{=}  t_M \cdot t_N (m)  \in M. \]
Same computations work for the morphisms $t_M \colon HKer(s_N) \cap HKer(s_M) \to HKer(s_N) \cap M$ and $t_M \colon N \cap HKer(s_M) \to N \cap M$ by using $(2C2)$, $(2C3)$ and $t_N \cdot i_N = Id_N$.

  Moreover, the square \eqref{square annex} commutes thanks to the compatibility condition $t_M \cdot t_N = t_N \cdot t_M$ (2C2). 
  
  Thanks to the equivalence between crossed modules and internal groupoids in $\sf{Hopf}_{K,coc}$ (Proposition \ref{equi xmod 1}), each arrow in the square \eqref{square annex} is a crossed module, which gives the condition $(CS1)$.  Moreover, $t_N \colon HKer(s_N) \cap HKer(s_M) \to N \cap HKer(s_M)$ and $t_M \colon HKer(s_N) \cap HKer(s_M) \to HKer(s_N) \cap M $ are $(N \cap M )$-module morphisms $(CS2)$,  since for $p \in N \cap M$ and $l \in HKer(s_N) \cap HKer(s_M)$,  we have 
\begin{align*}
t_N(p \triangleright l) &= t_N(p_1lS(p_2))\\
&= t_N(p_1)t_N(l)t_N(S(p_2))\\
&= p_1t_N(l)S(p_2)\\
&= p \triangleright t_N(l),
\end{align*} 
since $t_N \cdot i_N = Id_N$, and a similar computation holds for $t_M$. 

The condition $(CS3)$ is satisfied since it is the condition $(2A2)$ of the definition of Hopf $2$-action.

The condition $(CS4)$ holds for $t_M$ (and similarly for $t_N$) thanks to $t_M \cdot i_M = Id_M$, since for $n$ an element of $N \cap HKer(s_M)$ and $m$ an element of $HKer(s_N) \cap M$, we have the following equalities,
\begin{align*}
t_M \cdot h(n \otimes m) &= t_M(n_1m_1S(n_2)S(m_2)) \\
&= t_M(n_1)m_1t_M(S(n_2))S(m_2)\\
&= ({t_M(n)} \triangleright m_1)S(m_2).
\end{align*} 

In order to obtain the conditions $(CS5)$ and $(CS6)$, we show that thanks the following equalities $[Hker(s_N),HKer(t_N)] = 0$ and $[HKer(s_M),HKer(t_M)] = 0$, in the definition of a cat$^2$-Hopf algebras, we have
\begin{equation}\label{aide'}
{t_M(n)} \triangleright k = n \triangleright k ,
\end{equation} 
\begin{equation} \label{aide}
{t_N(m)} \triangleright k = m \triangleright k,
\end{equation}
 for any $n \in N \cap HKer(s_M)$, $m \in HKer(s_N) \cap M$ and $k \in HKer(s_N)  \cap HKer(s_M)$. 

First, we can check that $t_M(S(n_1))n_2$ belongs to $HKer(t_M)$ for any $n \in N \cap HKer(s_M)$. Indeed, since $t_M$ is split ($t_M \cdot i_M = Id_M$) it implies that  $t_M \cdot t_M = t_M$, and we have the following equalities
\begin{align*} 
t_M(t_M(S(n_1))n_2)_1 \ox  (t_M (S(n_1))n_2)_2 &= t_M (t_M(S(n_2))n_3) \ox t_M (S(n_1))n_4 \\
&= (t_M \cdot t_M)(S(n_2)) t_M (n_3) \ox t_M (S(n_1))n_4\\ 
&= 1_M \ox  t_M(S(n_1))n_2.
\end{align*} 
   
Thanks to this observation we can use that  $[HKer(s_M),HKer(t_M)]=0$ to obtain our equality for $n \in N \cap HKer(s_M)$ and $k \in HKer(s_N)  \cap HKer(s_M)$,  
\begin{align*} {t_M(n)} \triangleright k
&= t_M(n_1)kt_M(S(n_2)) \\
&= t_M(n_1)k\left(t_M(S(n_2))n_3\right)S(n_4) \\
&= t_M(n_1)\left(t_M(S(n_2))n_3\right)kS(n_4) \\
&= n_1kS(n_2) \\
&=   n \triangleright k.
\end{align*} 
   And similarly for $m \in HKer(s_N) \cap M$ and $k \in HKer(s_N)  \cap HKer(s_M)$, we can show \eqref{aide}.

Thanks to \eqref{aide'}, \eqref{aide} and the fact that $[HKer(s_N),HKer(t_N)]=0$, we obtain the condition (CS5): for $l \in HKer(s_N) \cap HKer(s_M)$ and $m \in HKers(s_N) \cap M$, we have
\begin{align*}
h(t_N(l) \otimes m) &= t_N(l_1)m_1t_N(S(l_2))S(m_2)\\
&= t_N(l_1)m_1t_N(S(l_2))l_3S(l_4)S(m_2)\\
&= t_N(l_1)\big(t_N(S(l_2))l_3\big)m_1S(l_4)S(m_2)\\
&= l_1m_1S(l_2)S(m_2)\\
&= l_1({m} \triangleright S(l_2)) \\
&\stackrel{\eqref{aide}}{=} l_1({t_N(m)} \triangleright S(l_2)).
\end{align*} 
  
Similarly, for $l \in HKer(s_N) \cap HKer(s_M)$ and $n \in N \cap HKers(s_M)$, by using the fact that $[HKer(s_M),HKer(t_M)]=0$, we can show that $({t_M(n)} \triangleright l_1)S(l_2) = h(n  \otimes t_M(l)) $.

Moreover, thanks to (2A4) of Definition \ref{2-action}, \eqref{aide'} and \eqref{aide}, we obtain the condition $(CS6)$ for this construction, and we can conclude that the square \eqref{square annex} is a Hopf crossed square.

\subsection{Isomorphisms of Hopf algebras}\label{C_2}
We check that the following linear maps \[\psi \colon (L\rtimes N)\rtimes (M \rtimes P) \to (L\rtimes M)\rtimes (N \rtimes P) \colon (l \otimes n \otimes m \otimes p) \mapsto lS(h(m_1 \otimes n_1)) \otimes m_2 \otimes n_2 \otimes p,\]
\[\psi^{-1} \colon (L\rtimes M)\rtimes (N \rtimes P) \to (L\rtimes N)\rtimes (M \rtimes P) \colon (l \otimes m \otimes n \otimes p) \mapsto lh(m_1 \otimes n_1) \otimes n_2 \otimes m_2 \otimes p,\]  are isomorphisms in $\sf{Hopf}_{K,coc}$. On the one hand we obtain, for $l,l' \in L$, $m,m' \in M$, $n,n' \in N$ and $p,p' \in P$
\begin{align*}
& \psi((l \otimes n \otimes m \otimes p) (l' \otimes n' \otimes m' \otimes p')) \\ &\stackrel{\; \; \; \; \; }{=} \psi (l \Big( {n_1} \triangleright \big({m_1} \triangleright ({p_1} \triangleright l')\big) \Big)({n_2} \triangleright h(m_2 \otimes {p_2} \triangleright n'_1)) \otimes n_3({p_3} \triangleright n'_2) \otimes m_3({p_4}\triangleright m') \otimes p_5p')\\
&\stackrel{\; \; \; \; \; }{=} l \Big( {n_1} \triangleright \big({m_1} \triangleright ({p_1} \triangleright l')\big) \Big)({n_2} \triangleright h(m_2 \otimes {p_2} \triangleright n'_1))S(h( (m_3({p_4} \triangleright m'))_1 \otimes  (n_3({p_3} \triangleright n'_2))_1)  \otimes  \\ & \hspace{2cm} (m_3 ({p_4} \triangleright m'))_2 \otimes (n_3({p_3} \triangleright n'_2))_2\otimes p_5p'\\
&\stackrel{\eqref{act et delta}}{=}l \Big( {n_1} \triangleright \big({m_1} \triangleright ({p_1} \triangleright l')\big) \Big)({n_2} \triangleright h(m_2 \otimes {p_2} \triangleright n'_1))S(h( m_3 ({p_5} \triangleright m'_1) \otimes  n_3 ({p_3} \triangleright n'_2)))  \otimes  \\ & \hspace{2cm} m_4 ({p_6} \triangleright m'_2) \otimes n_4({p_4} \triangleright n'_3)\otimes p_7p'\\
&\stackrel{(2A4)}{=} l \Big( {n_1} \triangleright \big({m_1} \triangleright ({p_1} \triangleright l')\big) \Big)({n_2} \triangleright h(m_2 \otimes {p_2} \triangleright n'_1))S\big(({m_3} \triangleright h({p_6} \triangleright m'_1 \otimes  n_3({p_3} \triangleright n'_2)))\\ & \hspace{2cm}h(m_4 \otimes   n_4({p_4} \triangleright n'_3))\big)\otimes m_5({p_7} \triangleright m'_2)  \otimes n_5({p_5} \triangleright n'_4)\otimes p_8p'\\
&\stackrel{(2A4)}{=} l \Big( {n_1} \triangleright \big({m_1} \triangleright ({p_1} \triangleright l')\big) \Big)({n_2} \triangleright h(m_2 \otimes {p_2} \triangleright n'_1))S\big(({m_3} \triangleright h({p_6} \triangleright m'_1 \otimes  n_3 ({p_3} \triangleright n'_2)))\\ & \hspace{2cm}h(m_4 \otimes  n_4)({n_5} \triangleright h(m_5 \otimes{p_4} \triangleright n'_3))\big) \ \otimes  m_6({p_7} \triangleright m'_2) \otimes n_6 ({p_5} \triangleright n'_4) \otimes p_8p'\\
&\stackrel{ \eqref{coco}}{=}  l \Big( {n_1} \triangleright \big({m_1} \triangleright ({p_1} \triangleright l')\big) \Big)({n_2} \triangleright h(m_2 \otimes {p_2} \triangleright n'_1))({n_3} \triangleright S(h(m_3 \otimes{p_3} \triangleright n'_2))) \\ & \hspace{2cm}  S\big({m_4} \triangleright h({p_5} \triangleright m'_1 \otimes  n_6 ({p_4} \triangleright n'_3))h(m_5 \otimes  n_4)\big)\otimes   m_6({p_6} \triangleright m'_2)  \otimes n_5 ({p_7} \triangleright n'_4) \otimes p_8p'\\
&\stackrel{ \eqref{act et m}}{=}  l \Big( {n_1} \triangleright \big({m_1} \triangleright ({p_1} \triangleright l')\big) \Big)({n_2} \triangleright (h(m_2 \otimes {p_2} \triangleright n'_1)S(h(m_3 \otimes{p_3} \triangleright n'_2)))) \\ & \hspace{2cm}  S({m_4} \triangleright h({p_5} \triangleright m'_1 \otimes  n_5 ({p_4} \triangleright n'_3))h(m_5 \otimes  n_3))\otimes   m_6({p_6} \triangleright m'_2)  \otimes n_4 ({p_7} \triangleright n'_4) \otimes p_8p'\\
&\stackrel{\; \; \;\; \; }{=}  l \Big( {n_1} \triangleright \big({m_1} \triangleright ({p_1} \triangleright l')\big) \Big) S\left({m_2} \triangleright (h({p_3} \triangleright m'_1 \otimes  n_4 ({p_2} \triangleright n'_1))h(m_3 \otimes  n_2))\right) \\ & \hspace{2cm} \otimes   m_4({p_4} \triangleright m'_2)  \otimes n_3 ({p_5} \triangleright n'_2) \otimes p_6p'\\
&\stackrel{ \eqref{coco}}{=} l \Big( {n_1} \triangleright \big({m_1} \triangleright ({p_1} \triangleright l')\big) \Big)S(h(m_2 \otimes  n_2))S(({m_3} \triangleright h({p_2} \triangleright m'_1 \otimes  n_3({p_3} \triangleright n'_1)))) \otimes  m_4({p_4} \triangleright m'_2) \\ & \hspace{2cm} \otimes n_4({p_5} \triangleright n'_2) \otimes p_6p'
\end{align*} 
and on the other hand, by using the properties of $(L,N,M,P,S \cdot h \cdot \sigma)$ we have 
\begin{align*}
& \psi(l \otimes n \otimes m \otimes p) \psi (l' \otimes n' \otimes m' \otimes p') \\ &= (lS(h(m_1 \otimes n_1)) \otimes m_2 \otimes n_2 \otimes p)(l'S(h(m'_1 \otimes n'_1)) \otimes m'_2 \otimes n'_2 \otimes p') \\
&= lS(h(m_1 \otimes n_1))({m_2} \triangleright({n_2} \triangleright ({p_1}\triangleright l'S(h(m'_1 \otimes n'_1)))))\\ & \hspace{1cm} ({m_3} \triangleright S(h({p_2} \triangleright m'_2 \otimes n_3))) \otimes  m_4({p_3}\triangleright m'_3)\otimes  n_4({p_4} \triangleright n'_2) \otimes p_5p'\\
&\stackrel{\mathclap{\eqref{act et m}}}{=} \;  lS(h(m_1 \otimes n_1))({m_2} \triangleright ({n_2} \triangleright ({p_1} \triangleright l')))\\ & \hspace{1cm} ({m_3} \triangleright ({n_3} \triangleright ({p_2} \triangleright S(h(m'_1 \otimes n'_1)))))  ( {m_4}\triangleright  S(h({p_3} \triangleright m'_2 \otimes n_4))  \\ & \hspace{1cm} \otimes  m_5({p_4}\triangleright m'_3)  \otimes n_5({p_5} \triangleright n'_2) \otimes p_6p'\\
&\stackrel{\mathclap{(2A5) + (2A2)}}{=} \; \; \;  \; \; \;\;  l({n_1} \triangleright ({m_1} \triangleright ({p_1} \triangleright l')))S(h(m_2 \otimes n_2)) \\ & \hspace{1cm} ({m_3} \triangleright ({n_3} \triangleright (S(h({p_2} \triangleright m'_1 \otimes {p_3} \triangleright n'_1))))  ({m_4}\triangleright S(h({p_4} \triangleright m'_2 \otimes n_4))) \\ & \hspace{1cm} \otimes m_5({p_5} \triangleright m'_3)\otimes  n_5 ({p_6}\triangleright n'_2) \otimes p_7p'\\
&\stackrel{\mathclap{\eqref{act et m}}}{=} \;l ({n_1} \triangleright ({m_1} \triangleright ({p_1} \triangleright l')))S(h(m_2 \otimes n_2))\Big({m_3} \triangleright ({n_3} \triangleright S(h({p_2} \triangleright m'_1 \otimes {p_3} \triangleright n'_1)) S(h({p_4} \triangleright m'_2 \otimes n_4))\Big) \\ & \hspace{1cm}  \otimes m_4{(p_5} \triangleright m'_3)\otimes  n_5 ({p_6} \triangleright n'_2) \otimes p_7p'\\
&\stackrel{(2A4)}{=} l({n_1} \triangleright ({m_1} \triangleright ({p_1} \triangleright l')))S(h(m_2 \otimes n_2)) ({m_3} \triangleright S(h({p_2} \triangleright m'_1 \otimes n_3 ({p_3} \triangleright n'_1))))     \otimes m_4({p_4}\triangleright m'_2)\\ & \hspace{1cm}\otimes  n_4({p_5} \triangleright n'_2) \otimes p_6p'.
\end{align*} 
  Thus, the linear map $\psi$ is a morphism of Hopf algebras between $ (L\rtimes N)\rtimes (M \rtimes P) $ and $ (L\rtimes M)\rtimes (N \rtimes P) $, since it is a coalgebra morphism thanks to the cocommutativity.
Moreover, it is an isomorphism as we can check as follows 
\begin{align*}
\psi \cdot \psi^{-1} (l \ox n \ox m \ox p ) &= \psi (lh(m_1 \ox n_1) \ox m_2 \ox n_1 \ox p) \\
&= lh(m_1 \ox n_1)S(h(m_{2_1} \ox n_{2_1})) \ox n_{2_2} \ox m_{2_2} \ox p \\&
 = lh(m_1 \ox n_1)_1S(h(m_{1} \ox n_1)_2) \ox n_{2} \ox m_{2} \ox p \\
 &= l \ox n \ox m \ox p,
\end{align*}
and, thanks to similar computations, the identity $\psi^{-1} \cdot \psi = Id_{(L \rtimes M) \rtimes (N \rtimes P)}$ holds. 

 \subsection{ The triples $(M \rtimes P, L \rtimes N, d)$ and $(N \rtimes P, L \rtimes M, d')$ are crossed modules}\label{D_3}
  First we check that $d \colon L \rtimes N \to M \rtimes P$ defined by $d (l \otimes n) = \lambda(l) \otimes \nu(n)$ is a morphism of algebras thanks to the fact that $\lambda$ is $P$-linear. For $l$, $l'$ $\in L$, and $n$, $n'$ $\in N$, we have the following equalities
\begin{align*}
d((l \otimes n)(l' \otimes n')) &= d(l({\nu(n_1)} \triangleright l') \otimes n_2n')\\
&= \lambda(l)\lambda({\nu(n_1)} \triangleright l') \otimes \nu(n_2n') \\
&= \lambda(l)({\nu(n_1)} \triangleright \lambda(l')) \otimes \nu(n_2)\nu(n')\\
&= (\lambda(l) \otimes \nu(n))(\lambda(l') \otimes \nu(n'))\\
&=d(l \otimes n)d(l' \otimes n').
\end{align*} 
  
Moreover, this morphism satisfies the two needed properties to be a crossed module, since for any $m \in M$, $p \in P$, $l,l' \in L$ and $n,n' \in N$, on the one hand, we have
\begin{align*}
d({(m \otimes p)}\triangleright (l \otimes n)) &= d(({\mu(m_1)p_1} \triangleright l) h(m_2 \otimes {p_2} \triangleright n_1) \otimes {p_3} \triangleright n_2) \\
&= \lambda(({\mu(m_1)p_1} \triangleright l)h(m_2 \otimes {p_2} \triangleright n_1)) \otimes \nu( {p_3} \triangleright n_2)\\
&\stackrel{\mathclap{\eqref{m et act}(CS2)}}{=}  \; \; \; \; \; ({\mu(m_1)} \triangleright ({p_1} \triangleright \lambda(l)))\lambda(h(m_2 \otimes {p_2} \triangleright n_1)) \otimes \nu( {p_3} \triangleright n_2)\\
&\stackrel{\mathclap{(CM2)}}{=}\;  \; m_1({p_1} \triangleright \lambda(l))S(m_2)\lambda(h(m_3 \otimes {p_2} \triangleright n_1)) \otimes \nu( {p_3} \triangleright n_2)\\
&\stackrel{\mathclap{(CS4)}}{=}  \;  \; m_1({p_1} \triangleright \lambda(l))S(m_2)m_3({\nu({p_2} \triangleright n_1)}\triangleright S(m_4)) \otimes \nu({p_3} \triangleright n_2)\\
&\stackrel{\mathclap{(CM1)}}{=} \;  \; m_1({p_1} \triangleright \lambda(l)) ({p_2\nu(n_1)S(p_3)} \triangleright S(m_2)) \otimes p_4\nu(n_2)S(p_5)\\
&\stackrel{\mathclap{\eqref{m et act}}}{=} \; (m_1 ({p_1} \triangleright \lambda(l)) \otimes p_2\nu(n)) ({S(p_3)} \triangleright S(m_2) \otimes S(p_4)) \\
&=  (m_1 \otimes p_1)(\lambda(l) \otimes \nu(n)) ({S(p_2)} \triangleright S(m_2) \otimes S(p_3)) \\
&= (m_1 \otimes p_1)d(l \otimes n) S(m_2 \otimes p_2) ,
\end{align*} 
  
  and on the other hand, we have
\begin{align*}
{d(l \otimes n)} \triangleright (l' \otimes n') &= {(\lambda(l) \otimes \nu(n))} \triangleright (l' \otimes n')\\
&= ({\mu \cdot \lambda (l_1) \triangleright (\nu(n_1)} \triangleright l')) h(\lambda(l_2) \otimes {\nu(n_2)} \triangleright n'_1) \otimes {\nu(n_3)} \triangleright n'_2\\
&\stackrel{\mathclap{(CS5)}}{=} \; \;  ({\mu \cdot \lambda (l_1)\triangleright (\nu(n_1)} \triangleright l')) l_2 ({\nu({\nu(n_2)} \triangleright n_1')} \triangleright S(l_3))\otimes {\nu(n_3)} \triangleright n_2'\\
&\stackrel{\mathclap{(CM2)}}{=}  \; \;  l_1({\nu(n_1)}\triangleright l')S(l_2) l_3 ({\nu({\nu(n_2)} \triangleright n_1')} \triangleright S(l_4))\otimes {\nu(n_3)} \triangleright n_2'\\
&\stackrel{\mathclap{(CM1)}}{=} \; \;  l_1({\nu(n_1)} \triangleright l)'\big({\nu(n_2)\nu(n'_1)}{S(\nu(n_3))} \triangleright S(l_2)\big) \otimes n_4n'_2S(n_5)\\
&\stackrel{\mathclap{\eqref{m et act}}}{=} \; l_1({\nu(n_1)} \triangleright l)'({\nu(n_2n'_1)}\triangleright ({S(\nu(n_3)} \triangleright S(l_2))) \otimes n_4n'_2S(n_5)\\
&=(l_1({\nu(n_1)} \triangleright l') \otimes n_2n') ({S(\nu(n_3)} \triangleright S(l_2)) \otimes S(n_4))\\
&=(l_1 \otimes n_1)(l' \otimes n') ({S(\nu(n_2)} \triangleright S(l_2)) \otimes S(n_3))\\
&=(l_1 \otimes n_1)(l' \otimes n') S(l_2 \otimes n_2).
\end{align*} 
   
We obtain the same conclusion for $d'$ with similar computations. Hence, we have two Hopf crossed modules.
 
 \subsection{Conditions of compatibility of the cat$^2$-Hopf algebra}\label{D_4}
 We check that the morphisms defined by  
\begin{align*}
&s_1 \colon  (L\rtimes N)\rtimes (M \rtimes P) \to (M \rtimes P) \colon s_1(l \otimes n \otimes m \otimes p) = \epsilon(l)\epsilon(n) (m \otimes p), \\
&t_1 \colon  (L\rtimes N)\rtimes (M \rtimes P) \to (M \rtimes P) \colon t_1(l \otimes n \otimes m \otimes p) = \lambda(l)({\nu(n_1)} \triangleright m ) \otimes \nu(n_2)p,\\
&s_2 \colon  (L\rtimes N)\rtimes (M \rtimes P) \to (N \rtimes P) \colon s_2(l \otimes n \otimes m \otimes p) = \epsilon(l)\epsilon(m) (n \otimes p),\\
&t_2 \colon  (L\rtimes N)\rtimes (M \rtimes P) \to (N \rtimes P) \colon t_2(l \otimes n \otimes m \otimes p) =  \lambda'(l)n \otimes \mu(m)p,
\end{align*} 
    satisfy the compatibility conditions of a cat$^2$-Hopf algebra. Let $(l \otimes n \otimes m \otimes p)$ be an element in $(L\rtimes N)\rtimes (M \rtimes P)$, we have the following identities 
\begin{align*}
t_1 \cdot t_2(l \otimes n \otimes m \otimes p) &= t_1(1_L \otimes \lambda'(l)n \otimes 1_M \otimes \mu(m)p)\\
&= 1_L \ox 1_N \ox  \lambda(1_L)(({\nu((\lambda'(l)n)_1}) \triangleright 1_M) \otimes \nu((\lambda'(l)n)_2)\mu(m)p\\
&\stackrel{\mathclap{\eqref{act et 1}}}{=} \; 1_L \ox 1_N \ox 1_M \otimes \nu\cdot \lambda' (l) \nu(n) \mu(m)p\\
&= 1_L \ox 1_N \ox 1_M \otimes \mu\cdot \lambda (l) \nu(n) \mu(m)p\\
&\stackrel{\mathclap{(CM2)}}{=} \; \; 1_L \ox \lambda'(1_L)1_N \ox 1_M \otimes \mu(\lambda(l)({\nu(n_1)} \triangleright m))\nu(n_2)p\\
&=t_2( 1_L \otimes 1_N \otimes \lambda(l)({\nu(n_1)} \triangleright m) \otimes \nu(n_2)p)\\
&= t_2\cdot t_1(l \otimes n \otimes m \otimes p)
\end{align*} 
  
\begin{align*}
s_1 \cdot t_2(l \otimes n \otimes m \otimes p) &= s_1(1_L \otimes \lambda'(l)n \otimes 1_M \otimes \mu(m)p)\\
&= 1_L \ox 1_N \ox \epsilon(l)\epsilon(n) 1_M \otimes \mu(m)p\\
&=t_2(\epsilon(l)\epsilon(n) 1_L\otimes 1_N \otimes m \otimes p)\\
&= t_2 \cdot s_1 (l \otimes n \otimes m \otimes p) 
\end{align*} 
  
\begin{align*}
t_1 \cdot s_2(l \otimes n \otimes m \otimes p) &= t_1(\epsilon(l)\epsilon(m) 1_L \otimes n \otimes 1_M \otimes p)\\
&= \epsilon(l)\epsilon(m) 1_L \ox 1_N \ox ({\nu(n_1)} \triangleright 1_M) \otimes \nu(n_2)p\\
&\stackrel{\mathclap{\eqref{act et 1}}}{=} \; \epsilon(l)\epsilon(m) 1_L \otimes 1_N \ox 1_M \ox \nu(n)p\\
&=s_2(1_L \otimes 1_N \otimes \lambda(l)({\nu(n_1)} \triangleright m ) \otimes \nu(n_2)p)\\
&= s_2 \cdot t_1 (l \otimes n \otimes m \otimes p) .
\end{align*} 
  
Hence, the 7-tuple $((L \rtimes N ) \rtimes (M \rtimes P), M\rtimes P, N \rtimes P, s_1, t_1,s_2,t_2)$ is a cat$^2$-Hopf algebra.
\subsection{Extension of the functors to the morphisms }\label{D_5}
 Let $f$ be a morphism of cat$^2$-Hopf algebras \begin{center}
\begin{tikzpicture}[descr/.style={fill=white},scale=0.9]
\node (A) at (0,0) {$H'$};
\node (B) at (0,2) {$H$};
\node (C) at (2,2) {$N$};
\node (D) at (2,0) {$N'$};
  \path[-stealth]
 (B.south) edge node[right] {$f$} (A.north) 
 (C.south) edge node[right] {$f|_N$} (D.north);
 \path[->,font=\scriptsize]
([yshift=4pt]A.east) edge node[above] {$s'_{N'}$} ([yshift=4pt]D.west)
([yshift=-4pt]A.east) edge node[below] {$t'_{N'}$} ([yshift=-4pt]D.west)
([yshift=4pt]B.east) edge node[above] {$s_N$} ([yshift=4pt]C.west)
([yshift=-4pt]B.east) edge node[below] {$t_N$} ([yshift=-4pt]C.west);
\end{tikzpicture}
\begin{tikzpicture}[descr/.style={fill=white},scale=0.9]
\node (A) at (0,0) {$H'$};
\node (B) at (0,2) {$H$};
\node (C) at (2,2) {$M$};
\node (D) at (2,0) {$M'.$};
  \path[-stealth]
 (B.south) edge node[right] {$f$} (A.north) 
 (C.south) edge node[right] {$f|_M$} (D.north);
 \path[->,font=\scriptsize]
([yshift=4pt]A.east) edge node[above] {$s'_{M'}$} ([yshift=4pt]D.west)
([yshift=-4pt]A.east) edge node[below] {$t'_{M'}$} ([yshift=-4pt]D.west)
([yshift=4pt]B.east) edge node[above] {$s_M$} ([yshift=4pt]C.west)
([yshift=-4pt]B.east) edge node[below] {$t_M$} ([yshift=-4pt]C.west);
\end{tikzpicture}
\end{center} We already know that from a morphism of cat$^2$-Hopf algebras we can obtain a morphism of Hopf $2$-actions (See Appendix \ref{C_3}). Moreover, the following cube commutes since $f|_N \cdot t_N = t_{N'} \cdot f $, $f|_M \cdot t_M = t_M \cdot f $ 

\begin{center}
\begin{tikzpicture}[descr/.style={fill=white},baseline=(A.base)]
\node (X) at (0,-5) {$HKer(s_N) \cap M$};
\node (A) at (0,0) {$HKer(s_N) \cap HKer(s_M)$};
\node (B) at (5,0) {$N \cap HKer(s_M)$};
\node (C) at (5,-5) {$N \cap M.$};
\node (X') at (2.5,-2.5) {$HKer(s'_{N'}) \cap M'$};
\node (A') at (2.5,2.5) {$HKer(s'_{N'}) \cap HKer(s'_{M'})$};
\node (B') at (7.5,2.5) {$N' \cap HKer(s'_{M'})$};
\node (C') at (7.5,-2.5) {$N' \cap M'$};
\path[-stealth] 
([xshift=2mm]X.north) edge node[descr] {$f_{HKer(s_N) \cap M}$} (X'.south)
([xshift=2mm]C.north) edge node[descr] {$f_{N \cap M}$} (C'.south)

  (B.north) edge node[descr] {$f_{N \cap HKer(s_M)}$} ([xshift=-2mm]B'.south)
(A'.south) edge node[below,xshift=1mm,yshift=-3mm] {$\; \; \; \; \; t'_{M'}$} (X'.north)
  (X'.east) edge node[above] {$\; \; \; \; \; \; \; \; t'_{N'}$}(C'.west)
   (A'.east)  edge node[above] {$ t'_{N'} $} (B'.west)
(B'.south) edge node[right] {$ t'_{M'}$} (C'.north)  
  (A.south) edge node[right] {$t_M $}  (X.north)
  (X.east) edge node[above] {$t_N$}  (C.west)
   (A.east)  edge node[above] {$\; \; \; \; \; \; t_N$} (B.west) 
   (A.north) edge node[descr]{$f_{HKer(s_N) \cap HKer(s_M)}$}([xshift=-2mm]A'.south)
(B.south) edge node[descr] {$t_M$} (C.north);
\end{tikzpicture}
\end{center}

Conversely, if we consider a morphism of Hopf  crossed squares $(\alpha, \beta, \gamma, \delta)$
\begin{center}
\begin{tikzpicture}[descr/.style={fill=white},baseline=(A.base)]
\node (X) at (0,-3) {$N$};
\node (A) at (0,0) {$L$};
\node (B) at (3,0) {$M$};
\node (C) at (3,-3) {$P,$};
\node (X') at (1.5,-1.5) {$\hat{N}$};
\node (A') at (1.5,1.5) {$\hat{L}$};
\node (B') at (4.5,1.5) {$\hat{M}$};
\node (C') at (4.5,-1.5) {$\hat{P}$};
\path[-stealth] 
(X.north east) edge node[above] {$\gamma$} (X'.south west)
(C.north east) edge node[above] {$\delta$} (C'.south west)
 (A.north east) edge node[above]{$\alpha$}(A'.south west)
  (B.north east) edge node[above] {$\beta$} (B'.south west)
(A'.south) edge node[below] {$\hat{\lambda}'\; \; \; \; \;$} (X'.north)
  (X'.east) edge node[above] {$\; \; \; \; \; \; \; \; \hat{\nu}$}(C'.west)
   (A'.east)  edge node[above] {$\hat{\lambda}$} (B'.west)
(B'.south) edge node[right] {$\hat{\mu} $} (C'.north)  
  (A.south) edge node[left] {$\lambda' $}  (X.north)
  (X.east) edge node[above] {$\nu $}  (C.west)
   (A.east)  edge node[above] {$\; \; \; \; \; \; \lambda$} (B.west)
(B.south) edge node[below] {$\mu \; \; \; \;$} (C.north);
\end{tikzpicture}
\end{center} the morphism of $2$-fold split epimorphisms $\alpha \ox \beta \ox \gamma \ox \delta$ obtained in \ref{C_3} is, in particular, a morphism of cat$^2$-Hopf algebras \begin{center}
\begin{tikzpicture}[descr/.style={fill=white},scale=0.9]
\node (A) at (0,0) {$(L' \rtimes N') \rtimes (M' \rtimes P')$};
\node (B) at (0,2) {$(L \rtimes N) \rtimes (M \rtimes P)$};
\node (C) at (4,2) {$M \rtimes P$};
\node (D) at (4,0) {$M' \rtimes P'$};
  \path[-stealth]
 (B.south) edge node[descr] {$\alpha \otimes \gamma \otimes \beta \otimes \delta$} (A.north) 
 (C.south) edge node[right] {$\beta \otimes \delta$} (D.north);
 \path[->,font=\scriptsize]
([yshift=4pt]A.east) edge node[above] {$s'_1$} ([yshift=4pt]D.west)
([yshift=-4pt]A.east) edge node[below] {$t'_1$} ([yshift=-4pt]D.west)
([yshift=4pt]B.east) edge node[above] {$s_1$} ([yshift=4pt]C.west)
([yshift=-4pt]B.east) edge node[below] {$t_1$} ([yshift=-4pt]C.west);
\end{tikzpicture}
\begin{tikzpicture}[descr/.style={fill=white},scale=0.9]
\node (A) at (0,0) {$(L' \rtimes N') \rtimes (M' \rtimes P')$};
\node (B) at (0,2) {$(L \rtimes N) \rtimes (M \rtimes P)$};
\node (C) at (4,2) {$N \rtimes P$};
\node (D) at (4,0) {$N' \rtimes P'.$};
  \path[-stealth]
 (B.south) edge node[descr] {$\alpha \otimes \gamma \otimes \beta \otimes \delta$} (A.north) 
 (C.south) edge node[right] {$ \gamma \otimes \delta$} (D.north);
 \path[->,font=\scriptsize]
([yshift=4pt]A.east) edge node[above] {$s'_2$} ([yshift=4pt]D.west)
([yshift=-4pt]A.east) edge node[below] {$t'_2$} ([yshift=-4pt]D.west)
([yshift=4pt]B.east) edge node[above] {$s_2$} ([yshift=4pt]C.west)
([yshift=-4pt]B.east) edge node[below] {$t_2$} ([yshift=-4pt]C.west);
\end{tikzpicture}
\end{center}
Indeed, we just need that the squares involving t commute, which is checked for any $l \ox n \ox m \ox p \in (L \rtimes N) \rtimes (M \rtimes P)$, thanks to the following equalities
\begin{align*}
(\beta \otimes \delta) \cdot t_1 (l \otimes n \otimes m \otimes p) &= (\beta \otimes \delta)(\lambda(l)({\nu(n_1)} \triangleright m) \otimes \nu(n_2)p) \\
&= \beta(\lambda(l)({\nu(n_1)} \triangleright m)) \otimes \delta(\nu(n_2)p) \\
&= \beta \cdot \lambda(l)\beta({\nu(n_1)} \triangleright m) \otimes \delta(\nu(n_2)p) \\
&= \hat{\lambda} \cdot \alpha(l) ({\hat{\nu}\cdot \gamma (n_1)} \triangleright \beta(m)) \otimes \hat{\nu}\cdot \gamma (n_2)\delta(p)\\
&=  t'_1  \cdot (\alpha \otimes \gamma \otimes \beta \otimes \delta)(l \otimes n \otimes m \otimes p),
\end{align*} 
   
\begin{align*}
(\gamma \otimes \delta) \cdot t_2 (l \otimes n \otimes m \otimes p) &= (\gamma \otimes \delta)(\lambda'(l)n \otimes \mu(m)p) \\
&= \gamma \cdot \lambda'(l) \gamma(n) \otimes \delta \cdot  \mu(m) \delta(p) \\
&= \hat{ \lambda'} \cdot \alpha(l) \gamma(n) \otimes \hat{\mu} \cdot  \beta(m) \delta(p) \\
&=  t'_2  \cdot (\alpha \otimes \gamma \otimes \beta \otimes \delta)(l \otimes n \otimes m \otimes p).
\end{align*} 
  
Hence, we can conclude that this is a morphism of cat$^2$-Hopf algebras and then the definition of $\hat{F}$ and $\hat{G}$ naturally extents to the morphisms.

\subsection{Equivalence between the categories $\sf X^2(Hopf_{K,coc})$ and $\sf cat^2(Hopf_{K,coc})$  }\label{D_6}
On the one hand, let us consider a cat$^2$-Hopf algebra $(H,M,N,s_N, t_N, s_M, t_M)$, by applying the functors $\hat{F}$ and $\hat{G}$, we obtain the following  cat$^2$-Hopf algebra,
\begin{align*}
&H' = (HKer(s_N) \cap HKer(s_M)) \rtimes( HKer(s_N) \cap M ) \rtimes (N \cap HKer(s_M)) \rtimes( N \cap M), \\ &M'=(HKer(s_N) \cap M) \rtimes (N \cap M), \\
&N' =(N \cap HKer(s_M)) \rtimes (N \cap M) \\
&s_1 \colon  H' \to N'\colon s_1(l \otimes n \otimes m \otimes p) = \epsilon(l)\epsilon(n) (m \otimes p) \\
&t_1 \colon  H' \to N'\colon t_1(l \otimes n \otimes m \otimes p) = t_N(l)t_N(n_1)mt_N(S(n_2)) \otimes t_N(n_3)p \\
&s_2 \colon  H' \to M' \colon s_2(l \otimes n \otimes m \otimes p) = \epsilon(l)\epsilon(m) (n \otimes p)\\
&t_2 \colon  H' \to M' \colon t_2(l \otimes n \otimes m \otimes p) = t_M(l) n \ox t_M(m)  p.
\end{align*} 

The isomorphism of $2$-fold split epimorphism, $\phi \colon H' \to H \colon (l \otimes n \otimes m \otimes p) \mapsto lnmp$, defined in \eqref{iso split epi equi}, is a morphism of cat$^2$-Hopf algebras, since we verify that the two following squares commute  \begin{center}
\begin{tikzpicture}[descr/.style={fill=white},scale=0.9]
\node (A) at (0,0) {$H$};
\node (B) at (0,2) {$H'$};
\node (C) at (4,2) {$N'$};
\node (D) at (4,0) {$N$};
  \path[-stealth]
 (B.south) edge node[descr] {$\phi$} (A.north) 
 (C.south) edge node[right] {$\phi|_{N'}$} (D.north);
 \path[->,font=\scriptsize]
(A.east) edge node[above] {$t_N$} (D.west)
(B.east) edge node[above] {$t_1$} (C.west);
\end{tikzpicture}
\begin{tikzpicture}[descr/.style={fill=white},scale=0.9]
\node (A) at (0,0) {$H$};
\node (B) at (0,2) {$H'$};
\node (C) at (4,2) {$M'$};
\node (D) at (4,0) {$M.$};
  \path[-stealth]
 (B.south) edge node[descr] {$\phi$} (A.north) 
 (C.south) edge node[right] {$ \phi|_{M'}$} (D.north);
 \path[->,font=\scriptsize]
(A.east) edge node[above] {$t_M$} (D.west)
(B.east) edge node[above] {$t_2$} (C.west);
\end{tikzpicture}
\end{center} Indeed, for any $l \in (HKer(s_N) \cap HKer(s_M))$, $ n \in ( HKer(s_N) \cap M )$, $m \in (N \cap HKer(s_M))$ and $p \in ( N \cap M)$, we have the following identities
\begin{align*}
\phi|_{N'} \cdot t_1(l \otimes n \otimes m \otimes p) &= \phi|_{N'} (t_N(l)t_N(n_1)mt_N(S(n_2)) \otimes t_N(n_3)p) \\
&= t_N(l)t_N(n_1)mt_N(S(n_2))t_N(n_3)p\\
&= t_N(l)t_N(n)mp\\
&\stackrel{\mathclap{t_N \cdot i_N = Id_N}}{=} \; \; \; \; \; \;  t_N(lnmp)\\
&= t_N \cdot \phi( l \otimes n \otimes m \otimes p).
\end{align*} 
  
Same computation as above makes the right diagram commutes. 

On the other hand, we already know that $\tilde{F} \cdot \tilde{G} = Id_{\sf Act^2(Hopf_{K,coc})}$. By applying $\hat{F} \cdot \hat{G}$ on the Hopf crossed square $(L, M, N, P, h, \lambda, \lambda', \mu, \nu)$, we obtain the following Hopf crossed square 

\begin{center}
\begin{tikzpicture}[descr/.style={fill=white},baseline=(A.base)]
\node (X) at (0,-3) {$HKer(s_1) \cap (N \rtimes P)$};
\node (A) at (0,0) {$HKer(s_1) \cap HKer(s_2)$};
\node (B) at (5,0) {$(M \rtimes P) \cap HKer(s_2)$};
\node (C) at (5,-3) {$(N \rtimes P)  \cap (M \rtimes P),$};
\path[-stealth] 
  (A.south) edge node[right] {$t_2$}  (X.north)
  (X.east) edge node[above] {$t_1$}  (C.west)
   (A.east)  edge node[above] {$t_1$} (B.west)
(B.south) edge node[right] {$t_2$} (C.north);
\end{tikzpicture}
\end{center}
where the intersections are defined as in Appendix \ref{C_4} and 
\begin{align*}
&t_1 \colon  (L\rtimes N)\rtimes (M \rtimes P) \to (M \rtimes P) \colon t_1(l \otimes n \otimes m \otimes p) = \lambda(l)({\nu(n_1)} \triangleright m) \otimes \nu(n_2)p,\\
&t_2 \colon  (L\rtimes N)\rtimes (M \rtimes P) \to (N \rtimes P) \colon t_2(l \otimes n \otimes m \otimes p) =  \lambda'(l)n \otimes \mu(m)p.
\end{align*} 
  
We observe that the canonical isomorphism of Hopf $2$-actions defined in Appendix \ref{C_4} is also an isomorphism of Hopf crossed squares.

Hence, this isomorphism implies that $\hat{F} \cdot \hat{G} = Id_{\sf X^2(Hopf_{K,coc})}$, and we conclude that $\sf X^2(Hopf_{K,coc})$ and $\sf cat^2(Hopf_{K,coc})$ are equivalent.

\end{document}